\documentclass[a4paper,12pt]{article}

\usepackage{amsmath,amssymb,amsthm,amscd}

\newcommand{\Fg}{\mathfrak{g}}
\newcommand{\Fh}{\mathfrak{h}}
\newcommand{\CB}{\mathcal{B}}
\newcommand{\CG}{\mathcal{G}}
\newcommand{\CK}{\mathcal{K}}
\newcommand{\CO}{\mathcal{O}}
\newcommand{\CZ}{\mathcal{Z}}

\newcommand{\BC}{\mathbb{C}}
\newcommand{\BR}{\mathbb{R}}
\newcommand{\BQ}{\mathbb{Q}}
\newcommand{\BZ}{\mathbb{Z}}

\newcommand{\mv}{\mathcal{MV}}
\newcommand{\Hom}{\mathop{\rm Hom}\nolimits}
\newcommand{\wt}{\mathop{\rm wt}\nolimits}
\newcommand{\Conv}{\mathop{\rm Conv}\nolimits}
\newcommand{\id}{\mathop{\rm id}\nolimits}

\newcommand{\pair}[2]{\langle #1,\,#2 \rangle}
\newcommand{\Bpair}[2]{\bigl\langle #1,\,#2 \bigr\rangle}

\newcommand{\ve}{\varepsilon}
\newcommand{\vp}{\varphi}

\newcommand{\ha}[1]{\widehat{#1}}
\newcommand{\ol}[1]{\overline{#1}}

\newcommand{\bzero}{{\bf 0}}

\newcommand{\ba}{{\bf a}}
\newcommand{\bb}{{\bf b}}
\newcommand{\bc}{{\bf c}}
\newcommand{\bd}{{\bf d}}
\newcommand{\bi}{{\bf i}}
\newcommand{\bj}{{\bf j}}
\newcommand{\wi}[1]{w^{\bi}_{#1}}
\newcommand{\wj}[1]{w^{\bj}_{#1}}
\newcommand{\yi}[1]{y^{\bi}_{#1}}

\newcommand{\vi}[1]{v^{\bi}_{#1}}
\newcommand{\vj}[1]{v^{\bj}_{#1}}
\newcommand{\bti}[1]{\beta^{\bi}_{#1}}
\newcommand{\btj}[1]{\beta^{\bj}_{#1}}
\newcommand{\si}[1]{s_{i_{#1}}}
\newcommand{\sj}[1]{s_{j_{#1}}}
\newcommand{\Ni}[1]{n^{\bi}_{#1}}
\newcommand{\Nj}[1]{n^{\bj}_{#1}}

\newcommand{\NNj}[1]{N^{\bj}_{#1}}
\newcommand{\xii}[1]{\xi^{\bi}_{#1}}
\newcommand{\xij}[1]{\xi^{\bj}_{#1}}

\newcommand{\bqed}{\quad \hbox{\rule[-0.5pt]{3pt}{8pt}}}

\makeatletter
\@addtoreset{equation}{subsection}

\renewcommand\section{\@startsection{section}{1}{0pt}
{-3.5ex plus -1ex minus -.2ex}{1.0ex plus .2ex}{\large\bf}}
\renewcommand\subsection{\@startsection{subsection}{1}{0pt}
{2.5ex plus 1ex minus .2ex}{-1em}{\bf}}
\makeatother

\newcommand{\vsp}{\vspace{3mm}}

\theoremstyle{plain}
\newtheorem{thm}{Theorem}[subsection]
\newtheorem{lem}[thm]{Lemma}
\newtheorem{prop}[thm]{Proposition}
\newtheorem{cor}[thm]{Corollary}

\newtheorem{claim}{Claim}[thm]

\newtheorem*{claim*}{Claim}

\theoremstyle{definition}
\newtheorem{dfn}[thm]{Definition}
\newtheorem*{question}{Open Problem}

\theoremstyle{remark}
\newtheorem{rem}[thm]{Remark}
\newtheorem{ex}[thm]{Example}

\begin{document}

\setlength{\baselineskip}{18pt}

\title{\Large\bf 
 Mirkovi\'c-Vilonen polytopes lying in \\ 
 a Demazure crystal and an opposite Demazure crystal}
\author{
 Satoshi Naito \\ 
 \small Institute of Mathematics, University of Tsukuba, \\
 \small Tsukuba, Ibaraki 305-8571, Japan \ 
 (e-mail: {\tt naito@math.tsukuba.ac.jp})
 \\[2mm] and \\[2mm]
 Daisuke Sagaki \\ 
 \small Institute of Mathematics, University of Tsukuba, \\
 \small Tsukuba, Ibaraki 305-8571, Japan \ 
 (e-mail: {\tt sagaki@math.tsukuba.ac.jp})
}
\date{}
\maketitle

%
\begin{abstract} \setlength{\baselineskip}{16pt}
We give a necessary and sufficient condition 
for an MV polytope $P$ 
in a highest weight crystal to lie 
in an arbitrary fixed Demazure crystal 
(resp., opposite Demazure crystal), 
in terms of the lengths of 
edges along a path through the $1$-skeleton of $P$ 
corresponding to a reduced word for the longest element of 
the Weyl group $W$. 
Also, we give an explicit description 
as a pseudo-Weyl polytope for extremal MV polytopes
in a highest weight crystal. 
Finally, by combining the results above, 
we obtain a polytopal condition for an MV polytope 
$P$ to lie in an arbitrary fixed opposite Demazure crystal. 
\end{abstract}
%
%
\section{Introduction.}
\label{sec:intro}
Let $G$ be a complex, connected, semisimple algebraic group 
with Lie algebra $\Fg$, $T$ a maximal torus with Lie algebra $\Fh$, 
$B$ a Borel subgroup containing $T$, $N$ the unipotent radical of $B$, 
and $G^{\vee}$ the Langlands dual group of $G$ 
with Lie algebra $\Fg^{\vee}$;
we choose the convention that 
the roots in $B$ are the negative roots. 
Throughout this paper, we assume that $G$ is simply-laced;
all highest weight crystals and Demazure crystals 
--- the ones of our interest --- can be obtained from those 
in the simply-laced cases by a standard technique of 
``folding'' by diagram automorphisms 
(see \cite{NS1}\,--\,\cite{NS3}, and \cite{H}). 
The affine Grassmannian 
$\CG r$ for $G$ is defined to be the quotient set 
$G(\CK)/G(\CO)$, equipped with the structure of a complex algebraic 
ind-variety, where $G(\CK)$ denotes the group of 
$\CK:=\BC(\!(t)\!)$-valued points of $G$, and $G(\CO) \subset G(\CK)$ denotes the 
subgroup of $\CO:=\BC[[t]]$-valued points of $G$. 
For each dominant coweight 
$\lambda \in X_{\ast}(T):=\Hom(\BC^{\ast},\,T)$, 
Mirkovi\'c and Vilonen \cite{MV1}, \cite{MV2} 
discovered a family $\CZ(\lambda)$ of 
closed, irreducible, algebraic subvarieties, 
called Mirkovi\'c-Vilonen (MV for short) cycles, 
in the closure $\overline{{\CG r}^{\lambda}} \subset \CG r$ of 
the $G(\CO)$-orbit ${\CG r}^{\lambda}:=G(\CO)[\lambda] \subset \CG r$
through the image $[\lambda] \in \CG r$ of 
$\lambda \in X_{\ast}(T) \subset G(\CK)$, 
and proved that these algebraic subvarieties provide 
a basis for the irreducible highest weight module $L(\lambda)$ 
of highest weight $\lambda$ 
over the Langlands dual group $G^{\vee}$; 
more precisely, the family $\CZ(\lambda)$ consists of the irreducible 
components of the closure of the intersection 
${\CG r}^{\lambda} \cap N(\CK)[\nu] \subset \CG r$ 
for weights $\nu \in X_{\ast}(T)$ of $L(\lambda)$, 
where $N(\CK) \subset G(\CK)$ denotes 
the subgroup of $\CK$-valued points of $G$.
Furthermore, Braverman and Gaitsgory \cite{BG} (see also \cite{BFG})
endowed the set $\CZ(\lambda)$ of MV cycles 
with a crystal structure, and gave an isomorphism of crystals
between the resulting crystal and the crystal basis $\CB(\lambda)$ 
of the irreducible highest weight module $V(\lambda)$ of highest 
weight $\lambda \in X_{\ast}(T) \subset \Fh$ for the quantized universal 
enveloping algebra $U_{q}(\Fg^{\vee})$ over $\BC(q)$ of the (Langlands) 
dual Lie algebra $\Fg^{\vee}$ of $\Fg$ with Cartan subalgebra $\Fh^{\ast}$. 

In order to obtain an explicit combinatorial description of 
the MV cycles in $\CZ(\lambda)$ for a dominant coweight 
$\lambda \in X_{\ast}(T) \subset \Fh$, Anderson \cite{A} 
defined a family $\mv(\lambda)$ of convex polytopes 
in the real form $\Fh_{\BR}$ of $\Fh$, 
which he called Mirkovi\'c-Vilonen (MV for short) polytopes, 
to be the moment map images of 
the MV cycles in $\CZ(\lambda)$. 
Furthermore, Kamnitzer \cite{Kam1}, \cite{Kam2} characterized 
the MV polytopes in $\mv(\lambda)$ as certain pseudo-Weyl polytopes 
with highest vertex $\lambda$ satisfying 
the ``tropical Pl\"ucker relations'', and then showed 
the existence of a bijection (in fact, an isomorphism of crystals)
between the crystal basis $\CB(\lambda)$ 
and the set $\mv(\lambda)$ of MV polytopes, which is 
endowed with a crystal structure coming from the one
(due to Lusztig \cite{Lu1} and Berenstein-Zelevinsky \cite{BZ}) 
on the canonical basis for the negative part 
$U_{q}^{-}(\Fg^{\vee})$ of $U_{q}(\Fg^{\vee})$;
in \cite{Kam2}, Kamnitzer also proved that 
this crystal structure agrees with the crystal structure
coming from the one on the set $\CZ(\lambda)$ of MV cycles. 

Let $W$ denote the Weyl group of $\Fg$, 
with $e$ the unit element and $w_{0}$ 
the longest element. We fix a dominant coweight 
$\lambda \in X_{\ast}(T) \subset \Fh_{\BR}$. 
Now, to understand better the MV polytopes 
in $\mv(\lambda)$, we consider a filtration 
(compatible with the Bruhat ordering $\le$ on $W$) 
of the crystal basis $\CB(\lambda)$ 
by Demazure crystals 
$\CB_{x}(\lambda) \subset \CB(\lambda)$ 
(resp., opposite Demazure crystals 
$\CB^{x}(\lambda) \subset \CB(\lambda)$), 
$x \in W$, and then their images
$\mv_{x}(\lambda) \subset \mv(\lambda)$ 
(resp., $\mv^{x}(\lambda) \subset \mv(\lambda)$), 
$x \in W$, under the isomorphism of crystals 
$\CB(\lambda) \stackrel{\sim}{\rightarrow} 
\mv(\lambda)$ above. 
Here, for each $x \in W$, 
the Demazure module $V_{x}(\lambda)$ 
(resp., opposite Demazure module $V^{x}(\lambda)$) is 
the $U_{q}^{+}(\Fg^{\vee})$-submodule 
(resp., $U_{q}^{-}(\Fg^{\vee})$-submodule) of $V(\lambda)$ 
generated by the one-dimensional weight space 
$V(\lambda)_{x \cdot \lambda} 
 \subset V(\lambda)$ of weight 
$x \cdot \lambda \in X_{\ast}(T) \subset \Fh_{\BR}$, 
where $U_{q}^{+}(\Fg^{\vee})$ (resp., $U_{q}^{-}(\Fg^{\vee})$) is the 
positive part (resp., negative part) of $U_{q}(\Fg^{\vee})$; recall 
from \cite{Kas} that the Demazure crystal
$\CB_{x}(\lambda)$ 
(resp., opposite Demazure crystal $\CB^{x}(\lambda)$)
is a subset of $\CB(\lambda)$ such that 
\begin{equation*}
V(\lambda) \supset 
V_{x}(\lambda) = 
 \bigoplus_{b \in \CB_{x}(\lambda)}
 \BC(q) G_{\lambda}(b)
\qquad
\left(
\text{resp.,} \ 
V(\lambda) \supset 
V^{x}(\lambda) = 
 \bigoplus_{b \in \CB^{x}(\lambda)}
 \BC(q) G_{\lambda}(b)
\right), 
\end{equation*}
where $G_{\lambda}(b)$, $b \in \CB(\lambda)$, 
form the lower global basis of $V(\lambda)$.
In this paper, for a dominant coweight 
$\lambda \in X_{\ast}(T) \subset \Fh_{\BR}$ and $x \in W$, 
we give a necessary and sufficient condition for 
an MV polytope in $\mv(\lambda)$ to lie in 
the image $\mv_{x}(\lambda) \subset \mv(\lambda)$ 
(resp., $\mv^{x}(\lambda) \subset \mv(\lambda)$) 
of $\CB_{x}(\lambda) \subset \CB(\lambda)$ 
(resp., $\CB^{x}(\lambda) \subset \CB(\lambda)$) 
under the isomorphism $\CB(\lambda) 
\stackrel{\sim}{\rightarrow} \mv(\lambda)$, 
in terms of the lengths of edges along a path 
through the $1$-skeleton of an MV polytope corresponding to 
a reduced word of $w_{0} \in W$; 
see Theorem~\ref{thm:main}\,(1) 
(resp., Theorem~\ref{thm:opdem}) for details. 
Also, we give a similar result for the crystal basis 
$\CB(\infty)$ of the negative part $U_{q}^{-}(\Fg^{\vee})$ 
of $U_{q}(\Fg^{\vee})$; 
see Theorem~\ref{thm:main}\,(2) for details. 

Furthermore, for each $x \in W$, we give an explicit 
description as a pseudo-Weyl polytope 
for the image (called an extremal MV polytope) 
$P_{x \cdot \lambda} \in \mv(\lambda)$ of the 
extremal element $u_{x \cdot \lambda} \in \CB(\lambda)$ 
of weight $x \cdot \lambda$ under 
the isomorphism of crystals $\CB(\lambda) 
\stackrel{\sim}{\rightarrow} \mv(\lambda)$; 
see Theorem~\ref{thm:ext}\,(1) for details.
As a corollary, we obtain a purely combinatorial proof of 
the fact that the extremal MV polytope 
$P_{x \cdot \lambda} \in \mv(\lambda)$ of weight 
$x \cdot \lambda \in X_{\ast}(T) \subset \Fh_{\BR}$ is 
identical to 
the convex hull $\Conv(W_{\le x} \cdot \lambda)$ 
in $\Fh_{\BR}$ of the subset $W_{\le x} \cdot \lambda:=
\bigl\{z \cdot \lambda \mid z \in W \text{ with } z \le x \bigr\}$ 
of the $W$-orbit $W \cdot \lambda$ through $\lambda$; 
see Theorem~\ref{thm:ext}\,(2) for details. 
This result is a generalization of the fact 
(see \cite{A} and \cite{Kam1}) that 
the MV polytope $P_{e \cdot \lambda} \in \mv(\lambda)$ 
of highest weight $\lambda$ consists only of 
the single element $\lambda \in X_{\ast}(T) \subset \Fh$, 
which is the moment map image of the point $[\lambda] \in \CG r$, 
and the MV polytope 
$P_{w_0 \cdot \lambda} \in \mv(\lambda)$ 
of lowest weight $w_{0} \cdot \lambda$ is identical to 
the convex hull $\Conv(W \cdot \lambda)$ in $\Fh_{\BR}$ 
of the $W$-orbit $W \cdot \lambda$, which is the moment map 
image of the finite-dimensional projective algebraic subvariety 
$\overline{{\CG r}^{\lambda}}$ of $\CG r$. 

By combining our result for opposite Demazure crystals and 
the explicit description above of extremal MV polytopes, 
we can prove that for a dominant coweight 
$\lambda \in X_{\ast}(T) \subset \Fh_{\BR}$ and $x \in W$, 
an MV polytope $P$ in $\mv(\lambda)$ lies in the opposite 
Demazure crystal $\mv^{x}(\lambda)$ if and only if 
$P$ contains (as a set) the extremal MV polytope 
$P_{x \cdot \lambda}= \Conv (W_{\le x} \cdot \lambda)$; 
see Theorem~\ref{thm:p-opdem}. 

Finally, for a dominant coweight 
$\lambda \in X_{\ast}(T) \subset \Fh_{\BR}$, 
we refine the filtration of the crystal $\mv(\lambda)$ 
consisting of Demazure crystals $\mv_{x}(\lambda)$
(resp., opposite Demazure crystals $\mv^{x}(\lambda)$), $x \in W$, 
to obtain the following decomposition into a disjoint union:
\begin{equation*}
\mv(\lambda)=\bigsqcup_{z \in W^{\lambda}_{\min}} 
\ol{\mv}_{z}(\lambda)
\qquad
\left(
\text{resp.,} \ 
\mv(\lambda)=\bigsqcup_{z \in W^{\lambda}_{\max}} 
\ol{\mv}^{z}(\lambda)
\right), 
\end{equation*}
where $W^{\lambda}_{\min}$ (resp., $W^{\lambda}_{\max}$) 
is the set of minimal (resp., maximal) coset representatives 
for the quotient set $W/W_{\lambda}$, 
with $W_{\lambda} \ (\subset W)$ the stabilizer of $\lambda$, 
 with respect to the Bruhat ordering $\le$ on $W$, 
 and where, for each $x \in W^{\lambda}_{\min}$ 
(resp., $x \in W^{\lambda}_{\max}$), we set 
\begin{align*}
& 
\ol{\mv}_{x}(\lambda):=
\mv_{x}(\lambda) \setminus 
\left(
 \bigcup_{z \in W^{\lambda}_{\min},\,\,z < x}
 \mv_{z}(\lambda)
\right) \\[3mm]
& \left(
\text{resp.,} \quad
\ol{\mv}^{x}(\lambda):=
\mv^{x}(\lambda) \setminus 
\left(
 \bigcup_{z \in W^{\lambda}_{\max},\,\,z > x}
 \mv^{z}(\lambda)
\right)
\right).
\end{align*}
Then, for each $x \in W$, we give a necessary and sufficient
condition for an MV polytope $P \in \mv(\lambda)$ to 
lie in $\ol{\mv}_{x}(\lambda)$ (resp., $\ol{\mv}^{x}(\lambda)$); 
see Proposition~\ref{prop:init} 
(resp., Proposition~\ref{prop:fin}) for details. 
This result tells us the smallest, with respect to 
the inclusion relation, Demazure crystal 
(resp., opposite Demazure crystal) in which $P \in \mv(\lambda)$ lies. 

Here, we should mention some closely related results: 
in \cite[\S11.2]{Schw}, Schwer showed that 
the dimension of each weight space of a Demazure module is 
given by the number of top-dimensional irreducible components of 
the intersection of an $N(\CK)$-orbit and an orbit 
under the standard Iwahori subgroup of $G(\CK)$ 
in the affine Grassmannian $\CG r$ 
(note that the closures of these irreducible components are not MV cycles) ; 
in \cite[\S4]{Ion}, Ion showed a very similar result by 
using certain intersections in the affine (full) flag manifold;  
in \cite{S}, Savage gave a geometric realization of 
Demazure crystals for simply-laced Kac-Moody algebras 
by using certain subvarieties 
(called Demazure quiver varieties) of Nakajima's quiver varieties; 
in \cite{N}, Nakashima gave a condition for an element
of a highest weight crystal to lie in a Demazure crystal, and also 
gave an explicit description of extremal elements
for symmetrizable Kac-Moody algebras, both of which are 
in terms of polyhedral realizations; 
among other realizations of Demazure crystals, 
we name \cite{LitInv}, \cite{Li2}, and \cite{LP-IMRN}. 

This paper is organized as follows. 
In \S\S\ref{subsec:MV} and \ref{subsec:cry-MV}, 
following Kamnitzer \cite{Kam1}, \cite{Kam2}, 
we review some of the basic facts on MV polytopes, 
and then give a description of the crystal structure on them. 
In \S\S\ref{subsec:demazure} and \ref{subsec:main}, 
after recalling the notion of Demazure crystals, 
we state a necessary and sufficient condition 
for an MV polytope to lie in 
a Demazure crystal in terms of the lengths of its edges; 
the proof is given in \S\ref{subsec:prf-main}. 
In \S\S\ref{subsec:opdem} and \ref{subsec:op-main}, 
a similar result for an opposite 
Demazure crystal is obtained. 
In \S\ref{subsec:extremal}, we provide an explicit 
description of extremal MV polytopes; 
\S\S\ref{subsec:length}--\ref{subsec:prf-thmext}
are devoted to its proof. 
In \S\ref{subsec:p-opdem}, 
we give a polytopal condition for an MV polytope 
to lie in an opposite Demazure crystal, and 
in \S\ref{subsec:question}, we pose 
a question concerning a necessary (but, not sufficient)
polytopal condition for an MV polytope to lie 
in a Demazure crystal. 
Finally, in \S\ref{sec:decom}, 
we give a necessary and sufficient 
condition for an MV polytope to lie in the 
$\ol{\mv}_{x}(\lambda)$ (resp., $\ol{\mv}^{x}(\lambda)$) 
above for $x \in W$. 

%
\section{Mirkovi\'c-Vilonen polytopes.}
\label{sec:MVda}

%
\subsection{Mirkovi\'c-Vilonen polytopes.}
\label{subsec:MV}

Here, following \cite{Kam1}, we briefly review some of 
the basic facts on Gelfand-Goresky-MacPherson-Serganova 
(GGMS for short) data and Mirkovi\'c-Vilonen (MV for short) polytopes. 
Let $G$ be a complex, connected, semisimple algebraic group 
with Lie algebra $\Fg$, $T$ a maximal torus with Lie algebra $\Fh$, 
$B$ a Borel subgroup containing $T$, and $N$ the unipotent radical of $B$; 
we choose the convention that 
the roots in $B$ are the negative roots. 
Throughout this paper, we assume that $G$ (and hence $\Fg$) 
is simply-laced, for reasons explained in the Introduction. 
Let 
$\bigl(A=(a_{ij})_{i,j \in I}, \, 
 \Pi:=\bigl\{\alpha_{j}\bigr\}_{j \in I}, \, 
 \Pi^{\vee}:=\bigl\{h_{j}\bigr\}_{j \in I}, \, 
 \Fh^{\ast},\,\Fh
 \bigr)$ be the root datum of $\Fg$, where 
$A=(a_{ij})_{i,j \in I}$ is the Cartan matrix, 
$\Fh$ is the Cartan subalgebra, 
$\Pi:=\bigl\{\alpha_{j}\bigr\}_{j \in I} \subset 
 \Fh^{\ast}:=\Hom_{\BC}(\Fh,\,\BC)$ 
is the set of simple roots, and 
$\Pi^{\vee}:=\bigl\{h_{j}\bigr\}_{j \in I} \subset \Fh$ 
is the set of simple coroots; note that 
$\pair{\alpha_{j}}{h_{i}}=a_{ij}$ for $i,\,j \in I$, 
where $\pair{\cdot}{\cdot}$ denotes the canonical pairing 
between $\Fh^{\ast}$ and $\Fh$. 
Let $W:=\langle s_{j} \mid j \in I \rangle$ 
be the Weyl group of $\Fg$, where $s_{j}$, $j \in I$, are 
the simple reflections, with length function $\ell:W \rightarrow \BZ_{\ge 0}$, 
and let $e,\,w_{0} \in W$ be the unit element 
and the longest element of $W$, respectively. 
Let $R(w_{0})$ denote the set of all reduced words for $w_{0}$, that is, 
all sequences $(i_{1},\,i_{2},\,\dots,\,i_{m})$ of elements of $I$ 
such that $s_{i_{1}}s_{i_{2}} \cdots s_{i_{m}}=w_{0}$, where 
$m$ is the length $\ell(w_{0})$ of the longest element $w_{0}$.
Also, we denote by $\le$ the (strong) Bruhat ordering on $W$. 

We denote by $G^{\vee}$ the Langlands dual group of $G$, and 
by $\Fg^{\vee}$ its Lie algebra. Thus, $\Fg^{\vee}$ is
the complex finite-dimensional semisimple Lie algebra
associated to the root datum 
$\bigl({}^{t}A=(a_{ji})_{i,j \in I}, \, 
 \Pi^{\vee}:=\bigl\{h_{j}\bigr\}_{j \in I}, \, 
 \Pi:=\bigl\{\alpha_{j}\bigr\}_{j \in I}, \, 
 \Fh,\,\Fh^{\ast}
 \bigr)$; 
note that the Cartan subalgebra of 
$\Fg^{\vee}$ is $\Fh^{\ast}$, not $\Fh$. 
Let $U_{q}(\Fg^{\vee})$ denote 
the quantized universal enveloping algebra 
over the field $\BC(q)$ of rational functions in $q$ 
associated to the dual Lie algebra $\Fg^{\vee}$, 
with positive part $U_{q}^{+}(\Fg^{\vee})$ and 
negative part $U_{q}^{-}(\Fg^{\vee})$. 
We denote by $\CB(\infty)$ the crystal basis of $U_{q}^{-}(\Fg^{\vee})$. 
Also, for a dominant coweight 
$\lambda \in X_{\ast}(T):=\Hom(\BC^{\ast},\,T) \subset 
 \Fh_{\BR}:= \sum_{j \in I} \BR h_{j} \subset \Fh$, denote by 
$V(\lambda)$ the integrable highest 
weight $U_{q}(\Fg^{\vee})$-module 
of highest weight $\lambda$, and by 
$\CB(\lambda)$ the crystal basis of 
$V(\lambda)$. 

Let $\mu_{\bullet}=(\mu_{w})_{w \in W}$ be 
a collection of elements of 
$\Fh_{\BR}=\sum_{j \in I} \BR h_{j}$. 
We call $\mu_{\bullet}$ a 
Gelfand-Goresky-MacPherson-Serganova 
(GGMS) datum if it satisfies 
the condition that 
$w^{-1} \cdot (\mu_{w'} - \mu_{w}) \in 
 \sum_{j \in I} \BZ_{\ge 0} h_{j}$ for all $w,\,w' \in W$. 
It follows by induction on $W$ 
with respect to the (weak) Bruhat ordering that 
$\mu_{\bullet}=(\mu_{w})_{w \in W}$ is a GGMS datum 
if and only if 
%
%
\begin{equation} \label{eq:length}
\mu_{ws_{i}}-\mu_{w} \in \BZ_{\ge 0}\,(w \cdot h_{i}) 
\quad 
\text{for every $w \in W$ and $i \in I$}. 
\end{equation}
Following \cite{Kam1} and \cite{Kam2}, 
to each GGMS datum $\mu_{\bullet}=(\mu_{w})_{w \in W}$, 
we associate a convex polytope $P(\mu_{\bullet}) \subset \Fh_{\BR}$ by:
%
%
\begin{equation} \label{eq:poly}
P(\mu_{\bullet})=
 \bigl\{ 
 h \in \Fh_{\BR} \mid 
 w^{-1} \cdot (h-\mu_{w}) \in 
 \textstyle{\sum_{j \in I}} \BR_{\ge 0} h_{j} \ 
 \text{for all $w \in W$}
 \bigr\};
\end{equation}
the polytope $P(\mu_{\bullet})$ is called a pseudo-Weyl polytope 
with GGMS datum $\mu_{\bullet}$. 
Note that the GGMS datum $\mu_{\bullet}=(\mu_{w})_{w \in W}$ is 
determined uniquely by the convex polytope $P(\mu_{\bullet})$. 
Also, we know from \cite[Proposition~2.2]{Kam1} that 
the set of vertices of the polytope $P(\mu_{\bullet})$ 
is given by the collection $\mu_{\bullet}=(\mu_{w})_{w \in W}$ 
(possibly, with repetitions). In particular, we have
\begin{equation*}
P(\mu_{\bullet}) = 
\Conv\,\bigl\{\mu_{w} \mid w \in W\bigr\},
\end{equation*}
where for a subset $X$ of $\Fh_{\BR}$, 
$\Conv X$ denotes the convex hull in 
$\Fh_{\BR}$ of $X$. 

%
\begin{dfn} \label{dfn:braid}
Let $\bi=(i_{1},\,i_{2},\,\dots,\,i_{m}) \in R(w_{0})$ and 
$\bj=(j_{1},\,j_{2},\,\dots,\,j_{m}) \in R(w_{0})$ be 
reduced words for the longest element $w_{0} \in W$. 

(1) We say that $\bi$ and $\bj$ are related 
by a $2$-move if there exist indices
$i,\,j \in I$ with $a_{ij}=a_{ji}=0$ and 
an integer $0 \le k \le m-2$ such that 
$i_{l}=j_{l}$ for all $1 \le l \le m$ with 
$l \ne k+1,\,k+2$, and such that 
$i_{k+1}=j_{k+2}=i$, $i_{k+2}=j_{k+1}=j$, i.e., 
%
%
\begin{equation} \label{eq:2move}
\begin{array}{l}
\bi=(i_{1},\,\dots,\,i_{k},\,i,\,j,\,i_{k+3},\,\dots,\,i_{m}), \\[1.5mm]
\bj=(i_{1},\,\dots,\,i_{k},\,j,\,i,\,i_{k+3},\,\dots,\,i_{m}).
\end{array}
\end{equation}

(2) We say that $\bi$ and $\bj$ are related 
by a $3$-move if there exist indices $i,\,j \in I$ 
with $a_{ij}=a_{ji}=-1$ and an integer $0 \le k \le m-3$ 
such that $i_{l}=j_{l}$ for all $1 \le l \le m$ with 
$l \ne k+1,\,k+2,\,k+3$, and such that
$i_{k+1}=i_{k+3}=j_{k+2}=i$, 
$i_{k+2}=j_{k+1}=j_{k+3}=j$, i.e., 
%
%
\begin{equation} \label{eq:3move}
\begin{array}{l}
\bi=(i_{1},\,\dots,\,i_{k},\,i,\,j,\,i,\,i_{k+4},\,\dots,\,i_{m}), \\[1.5mm]
\bj=(i_{1},\,\dots,\,i_{k},\,j,\,i,\,j,\,i_{k+4},\,\dots,\,i_{m}).
\end{array}
\end{equation}
\end{dfn}

%
\begin{rem} \label{rem:braid}
Let $\bi,\,\bj \in R(w_{0})$.
It is well-known (see, for example, 
\cite[Theorem~3.3.1\,(ii)]{BB}) that 
there exists a sequence 
$\bi=\bi_{0},\,\bi_{1},\,\dots,\,\bi_{t}=\bj$ 
of elements of $R(w_{0})$ such that 
$\bi_{u}$ and $\bi_{u+1}$ are related 
by a $2$-move or a $3$-move 
for each $0 \le u \le t-1$. 
\end{rem}

Let $\bi=(i_{1},\,i_{2},\,\dots,\,i_{m}) \in R(w_{0})$ 
be a reduced word for $w_{0}$. 
We set $\wi{l}:=\si{1}\si{2} \cdots \si{l} \in W$ 
for $0 \le l \le m$. For a GGMS datum 
$\mu_{\bullet}=(\mu_{w})_{w \in W}$, define integers 
$\Ni{l}=\Ni{l}(\mu_{\bullet}) \in \BZ_{\ge 0}$, 
$1 \le l \le m$, via the following 
``length formula'' (see \cite[Eq.(8)]{Kam1}): 
%
%
\begin{equation} \label{eq:n}
\mu_{\wi{l}}-\mu_{\wi{l-1}}=\Ni{l} \wi{l-1} \cdot h_{i_{l}}.
\end{equation}

\vsp

{\small 
\hspace{55mm}
\unitlength 0.1in
\begin{picture}( 15.0000,  3.3500)( 15.5000, -7.0000)
%
\special{pn 8}%
\special{sh 0.600}%
\special{ar 1600 600 50 50  0.0000000 6.2831853}%
\put(30.0000,-7.0000){\makebox(0,0)[lt]{$\mu_{\wi{l}}=\mu_{\wi{l-1}\si{l}}$}}%
\put(23.0000,-4.5000){\makebox(0,0){$\Ni{l}$}}%
%
\special{pn 8}%
\special{sh 0.600}%
\special{ar 3000 600 50 50  0.0000000 6.2831853}%
%
\special{pn 8}%
\special{pa 1600 600}%
\special{pa 3000 600}%
\special{fp}%
\put(16.0000,-7.0000){\makebox(0,0)[lt]{$\mu_{\wi{l-1}}$}}%
\end{picture}%
}

\vsp

%
\begin{dfn} \label{dfn:MV}
A GGMS datum $\mu_{\bullet}=(\mu_{w})_{w \in W}$ is 
said to be a Mirkovi\'c-Vilonen (MV) datum if 
it satisfies the following condition: 

(1) If two reduced words $\bi,\,\bj \in R(w_{0})$
are related by a $2$-move as in \eqref{eq:2move}, then 
$\Ni{l}=\Nj{l}$ for all $1 \le l \le m$ with 
$l \ne k+1,\,k+2$, and 
$\Ni{k+1}=\Nj{k+2}$, $\Ni{k+2}=\Nj{k+1}$. 

\vspace{5mm}

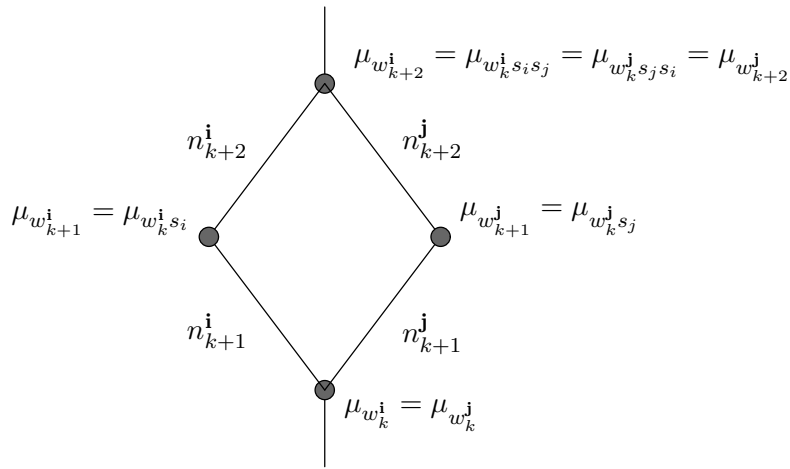
\newcommand{\vertexa}{
  $\mu_{\wi{k+2}}=\mu_{\wi{k}s_is_j}=
   \mu_{\wj{k}s_js_i}=\mu_{\wj{k+2}}$
}

{\small 
\hspace*{-20mm}
\unitlength 0.1in
\begin{picture}( 44.6000, 24.0000)(  4.4000,-28.0000)
%
\special{pn 8}%
\special{pa 4200 800}%
\special{pa 4200 400}%
\special{fp}%
%
\special{pn 8}%
\special{pa 4200 2400}%
\special{pa 4200 2800}%
\special{fp}%
%
\special{pn 8}%
\special{sh 0.600}%
\special{ar 4200 800 50 50  0.0000000 6.2831853}%
%
\special{pn 8}%
\special{sh 0.600}%
\special{ar 4200 2400 50 50  0.0000000 6.2831853}%
\put(35.0000,-16.0000){\makebox(0,0)[rb]{$\mu_{\wi{k+1}}=\mu_{\wi{k}s_{i}}$}}%
\put(43.0000,-24.5000){\makebox(0,0)[lt]{$\mu_{\wi{k}}=\mu_{\wj{k}}$}}%
%
\special{pn 8}%
\special{pa 4200 2400}%
\special{pa 4800 1600}%
\special{fp}%
\special{pa 4800 1600}%
\special{pa 4200 800}%
\special{fp}%
\special{pa 4200 800}%
\special{pa 3600 1600}%
\special{fp}%
\special{pa 3600 1600}%
\special{pa 4200 2400}%
\special{fp}%
%
\special{pn 8}%
\special{sh 0.600}%
\special{ar 4800 1600 50 50  0.0000000 6.2831853}%
%
\special{pn 8}%
\special{sh 0.600}%
\special{ar 3600 1600 50 50  0.0000000 6.2831853}%
\put(49.0000,-16.0000){\makebox(0,0)[lb]{$\mu_{\wj{k+1}}=\mu_{\wj{k}s_{j}}$}}%
\put(38.0000,-20.0000){\makebox(0,0)[rt]{$\Ni{k+1}$}}%
\put(38.0000,-12.0000){\makebox(0,0)[rb]{$\Ni{k+2}$}}%
\put(46.0000,-20.0000){\makebox(0,0)[lt]{$\Nj{k+1}$}}%
\put(46.0000,-12.0000){\makebox(0,0)[lb]{$\Nj{k+2}$}}%
\put(43.0000,-8.0000){\makebox(0,0)[lb]{\vertexa}}%
\end{picture}%
}

\vsp

(2) If two reduced words $\bi,\,\bj \in R(w_{0})$
are related by a $3$-move as in \eqref{eq:3move}, then 
$\Ni{l}=\Nj{l}$ for all $1 \le l \le m$ with 
$l \ne k+1,\,k+2,\,k+3$, and 
%
%
\begin{equation} \label{eq:3bm}
\begin{cases}
\Nj{k+1}=\Ni{k+2}+\Ni{k+3}-\min \bigl(\Ni{k+1},\, \Ni{k+3}\bigr), \\[1.5mm]
\Nj{k+2}=\min \bigl(\Ni{k+1},\, \Ni{k+3}\bigr), \\[1.5mm]
\Nj{k+3}=\Ni{k+1}+\Ni{k+2}-\min \bigl(\Ni{k+1},\, \Ni{k+3}\bigr).
\end{cases}
\end{equation}

\vspace{5mm}

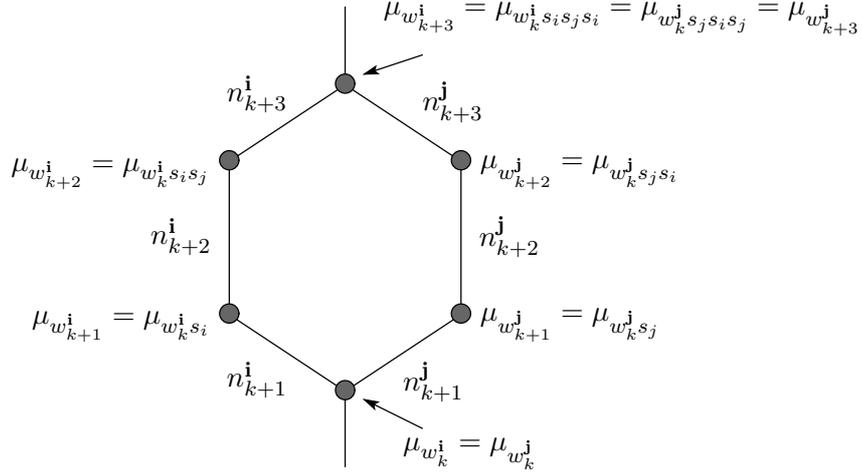
\newcommand{\vertexb}{
  $\mu_{\wi{k+3}}=\mu_{\wi{k}s_is_js_i}=
   \mu_{\wj{k}s_js_is_j}=\mu_{\wj{k+3}}$
}

{\small 
\hspace*{-22.5mm}
\unitlength 0.1in
\begin{picture}( 45.5000, 24.2000)(  1.5000,-28.0000)
%
\special{pn 8}%
\special{pa 4000 806}%
\special{pa 3400 1206}%
\special{fp}%
%
\special{pn 8}%
\special{pa 3400 1206}%
\special{pa 3400 2006}%
\special{fp}%
%
\special{pn 8}%
\special{pa 3400 2006}%
\special{pa 4000 2406}%
\special{fp}%
%
\special{pn 8}%
\special{pa 4000 2406}%
\special{pa 4600 2006}%
\special{fp}%
%
\special{pn 8}%
\special{pa 4600 2006}%
\special{pa 4600 1206}%
\special{fp}%
%
\special{pn 8}%
\special{pa 4600 1206}%
\special{pa 4000 806}%
\special{fp}%
%
\special{pn 8}%
\special{pa 4000 800}%
\special{pa 4000 400}%
\special{fp}%
%
\special{pn 8}%
\special{pa 4000 2400}%
\special{pa 4000 2800}%
\special{fp}%
%
\special{pn 8}%
\special{sh 0.600}%
\special{ar 4000 800 50 50  0.0000000 6.2831853}%
%
\special{pn 8}%
\special{sh 0.600}%
\special{ar 4600 1200 50 50  0.0000000 6.2831853}%
%
\special{pn 8}%
\special{sh 0.600}%
\special{ar 4600 2000 50 50  0.0000000 6.2831853}%
%
\special{pn 8}%
\special{sh 0.600}%
\special{ar 4000 2400 50 50  0.0000000 6.2831853}%
%
\special{pn 8}%
\special{sh 0.600}%
\special{ar 3400 2000 50 50  0.0000000 6.2831853}%
%
\special{pn 8}%
\special{sh 0.600}%
\special{ar 3400 1200 50 50  0.0000000 6.2831853}%
\put(41.5000,-5.5000){\makebox(0,0)[lb]{\vertexb}}%
\put(37.0000,-9.5000){\makebox(0,0)[rb]{$\Ni{k+3}$}}%
\put(44.0000,-10.0000){\makebox(0,0)[lb]{$\Nj{k+3}$}}%
\put(48.5000,-16.0000){\makebox(0,0){$\Nj{k+2}$}}%
\put(31.5000,-16.0000){\makebox(0,0){$\Ni{k+2}$}}%
\put(37.0000,-22.5000){\makebox(0,0)[rt]{$\Ni{k+1}$}}%
\put(43.0000,-22.5000){\makebox(0,0)[lt]{$\Nj{k+1}$}}%
\put(47.0000,-13.5000){\makebox(0,0)[lb]{$\mu_{\wj{k+2}}=\mu_{\wj{k}s_js_i}$}}%
\put(47.0000,-21.5000){\makebox(0,0)[lb]{$\mu_{\wj{k+1}}=\mu_{\wj{k}s_j}$}}%
\put(43.0000,-26.5000){\makebox(0,0)[lt]{$\mu_{\wi{k}}=\mu_{\wj{k}}$}}%
\put(33.0000,-21.5000){\makebox(0,0)[rb]{$\mu_{\wi{k+1}}=\mu_{\wi{k}s_i}$}}%
\put(33.0000,-13.5000){\makebox(0,0)[rb]{$\mu_{\wi{k+2}}=\mu_{\wi{k}s_is_j}$}}%
%
\special{pn 8}%
\special{pa 4400 650}%
\special{pa 4100 750}%
\special{fp}%
\special{sh 1}%
\special{pa 4100 750}%
\special{pa 4170 748}%
\special{pa 4152 734}%
\special{pa 4158 710}%
\special{pa 4100 750}%
\special{fp}%
%
\special{pn 8}%
\special{pa 4400 2600}%
\special{pa 4100 2450}%
\special{fp}%
\special{sh 1}%
\special{pa 4100 2450}%
\special{pa 4152 2498}%
\special{pa 4148 2474}%
\special{pa 4170 2462}%
\special{pa 4100 2450}%
\special{fp}%
\end{picture}%
}
\end{dfn}

The pseudo-Weyl polytope $P(\mu_{\bullet})$ 
with GGMS datum $\mu_{\bullet}=(\mu_{w})_{w \in W}$ 
(see \eqref{eq:poly}) is called a Mirkovi\'c-Vilonen (MV) 
polytope if the GGMS datum $\mu_{\bullet}=(\mu_{w})_{w \in W}$ 
is an MV datum. We denote by $\mv$ the set of all MV polytopes. 

\begin{rem}
In the definition above of MV polytopes, 
we first introduced the notion of MV datum  
(which was not introduced in \cite{Kam1}, \cite{Kam2}), and then 
called the associated pseudo-Weyl polytope (by \eqref{eq:poly})
an MV polytope. This definition of MV polytopes 
is equivalent to the one in \cite{Kam1}, \cite{Kam2}, 
which is easily seen by an argument in the proof of 
\cite[Proposition~5.4]{Kam1} (see the comment 
following \cite[Theorem~7.1]{Kam1}). 
\end{rem}

%
\subsection{Crystals $\mv(\infty)$ and $\mv(\lambda)$.}
\label{subsec:cry-MV}

For $j \in I$ and an MV datum
$\mu_{\bullet}=(\mu_{w})_{w \in W}$, 
we denote by $f_{j}\mu_{\bullet}$ 
(resp., $e_{j}\mu_{\bullet}$ if $\mu_{e} \ne \mu_{s_{j}}$)
a unique MV datum $\mu_{\bullet}'=(\mu_{w}')_{w \in W}$ 
such that $\mu_{e}'=\mu_{e}-h_{j}$ 
(resp., $\mu_{e}'=\mu_{e}+h_{j}$), and such that 
$\mu_{w}'=\mu_{w}$ for all $w \in W$ with $s_{j}w < w$ 
(see \cite[Theorem~3.5]{Kam2} and its proof); 
note that $\mu_{w_{0}}'=\mu_{w_{0}}$ and 
$\mu_{s_j}'=\mu_{s_j}$. 

Now, let $\mv(\infty)$ denote the 
set of MV polytopes $P=P(\mu_{\bullet})$ 
with GGMS (hence MV) datum $\mu_{\bullet}=(\mu_{w})_{w \in W}$ 
for which $\mu_{w_{0}}=0 \in \Fh_{\BR}$; 
note that, by the length formula, 
$\mu_{w} \in Q^{\vee}_{-}:=\sum_{j \in I} \BZ_{\le 0} h_{j}$
for all $w \in W$. 
Following \cite[\S\S3.3, 3.5, and 3.6]{Kam2}, 
we endow $\mv(\infty)$ 
with a crystal structure for $U_{q}(\Fg^{\vee})$ 
(due to Lusztig \cite{Lu2} and Berenstein-Zelevinsky \cite{BZ}) as follows. 
Let $P=P(\mu_{\bullet}) \in \mv(\infty)$ be an MV polytope 
with GGMS datum $\mu_{\bullet}=(\mu_{w})_{w \in W}$.
The weight $\wt(P)$ of $P$ is, by definition, 
equal to the vertex $\mu_{e} \in Q^{\vee}_{-}$. 
For each $j \in I$, 
we define the (lowering) Kashiwara operator
$f_{j}:\mv(\infty) \cup \{\bzero\} \rightarrow 
 \mv(\infty) \cup \{\bzero\}$ and 
the (raising) Kashiwara operator
$e_{j}:\mv(\infty) \cup \{\bzero\} \rightarrow 
 \mv(\infty) \cup \{\bzero\}$ by: 
\begin{align*}
e_{j}\bzero=f_{j}\bzero & :=\bzero, \\
f_{j}P=f_{j}P(\mu_{\bullet}) & :=P(f_{j}\mu_{\bullet}), \\[3mm]
e_{j}P=e_{j}P(\mu_{\bullet}) & :=
\begin{cases}
 P(e_{j}\mu_{\bullet}) & \text{if $\mu_{e} \ne \mu_{s_{j}}$}, \\[1.5mm]
 \bzero & \text{otherwise}, 
 \end{cases}
\end{align*}
where $\bzero$ is an additional element, 
not contained in $\mv(\infty)$. We set 
$\ve_{j}(P):=
 \max \bigl\{N \ge 0 \mid e_{j}^{N}P \ne \bzero \bigr\}$, and 
$\vp_{j}(P):=
 \pair{\alpha_{j}}{\wt(P)}+\ve_{j}(P)$. 

%
\begin{thm}[{\cite[\S3.3 and \S3.6]{Kam2}}] \label{thm:kam-inf}
The set $\mv(\infty)$, together with the maps 
$\wt$, $e_{j},\,f_{j} \ (j \in I)$, and 
$\ve_{j},\,\vp_{j} \ (j \in I)$ above, is 
a crystal for $U_{q}(\Fg^{\vee})$. 
Moreover, there exists a unique isomorphism 
$\Psi:\CB(\infty) 
 \stackrel{\sim}{\rightarrow} 
 \mv(\infty)$ of crystals for $U_{q}(\Fg^{\vee})$. 
\end{thm}
%
%
\begin{rem} \label{rem:inf}
Define a GGMS datum
$\mu_{\bullet}=(\mu_{w})_{w \in W}$ by: 
$\mu_{w}=0 \in \Fh_{\BR}$ for all $w \in W$. 
It is obvious that 
$\mu_{\bullet}$ is an MV datum, and 
the MV polytope $P_{0}:=P(\mu_{\bullet})$ is 
an element of $\mv(\infty)$ whose weight 
is $0 \in \Fh_{\BR}$. 
Therefore, under the isomorphism $\Psi$ of 
Theorem~\ref{thm:kam-inf}, 
the element $u_{\infty} \in \CB(\infty)$ 
corresponding to the identity element 
$1 \in U_{q}^{-}(\Fg^{\vee})$ is sent to 
the MV polytope $P_{0} \in \mv(\infty)$.
\end{rem}

Let $\lambda \in X_{\ast}(T) \subset \Fh_{\BR}$ 
be a dominant coweight. 
Let $\mv(\lambda)$ denote the 
set of MV polytopes $P=P(\mu_{\bullet})$ with 
GGMS (hence MV) datum $\mu_{\bullet}=(\mu_{w})_{w \in W}$ such that
$\mu_{w_{0}}=\lambda$ and such that 
$P$ is contained in the convex hull $\Conv (W \cdot \lambda)$
of the $W$-orbit $W \cdot \lambda \subset \Fh_{\BR}$; 
note that, by the length formula, 
$\mu_{w} \in \lambda+Q^{\vee}_{-}$ for all $w \in W$. 
Following \cite[\S6.2]{Kam2}, 
we endow $\mv(\lambda)$ 
with a crystal structure for $U_{q}(\Fg^{\vee})$
as follows. 
Let $P=P(\mu_{\bullet}) \in \mv(\lambda)$ be 
an MV polytope with GGMS datum 
$\mu_{\bullet}=(\mu_{w})_{w \in W}$.
The weight $\wt(P)$ of $P$ is, by definition, 
equal to the vertex $\mu_{e} \in \lambda+Q^{\vee}_{-}$. 
For each $j \in I$, 
we define the lowering Kashiwara operator
$f_{j}:\mv(\lambda) \cup \{\bzero\} \rightarrow 
 \mv(\lambda) \cup \{\bzero\}$ and 
the raising Kashiwara operator
$e_{j}:\mv(\lambda) \cup \{\bzero\} \rightarrow 
 \mv(\lambda) \cup \{\bzero\}$ by: 
\begin{align*}
e_{j}\bzero=f_{j}\bzero & :=\bzero, \\[3mm]
f_{j}P=f_{j}P(\mu_{\bullet}) & :=
\begin{cases}
 P(f_{j}\mu_{\bullet}) 
 & \text{if $P(f_{j}\mu_{\bullet}) \subset 
         \Conv (W \cdot \lambda)$}, \\[1.5mm]
 \bzero & \text{otherwise}, 
 \end{cases} \\[3mm]
e_{j}P=e_{j}P(\mu_{\bullet}) & :=
\begin{cases}
 P(e_{j}\mu_{\bullet}) & \text{if $\mu_{e} \ne \mu_{s_{j}}$}, \\[1.5mm]
 \bzero & \text{otherwise}, 
 \end{cases}
\end{align*}
where $\bzero$ is an additional element, 
not contained in $\mv(\lambda)$. 
We set 
$\ve_{j}(P):=
 \max \bigl\{N \ge 0 \mid e_{j}^{N}P \ne \bzero \bigr\}$, and 
$\vp_{j}(P):=
 \max \bigl\{N \ge 0 \mid f_{j}^{N}P \ne \bzero \bigr\}$. 

%
\begin{thm}[{\cite[Theorem~6.4]{Kam2}}] \label{thm:kam-int}
The set $\mv(\lambda)$, together with the maps 
$\wt$, $e_{j},\,f_{j} \ (j \in I)$, and 
$\ve_{j},\,\vp_{j} \ (j \in I)$ above, is 
a crystal for $U_{q}(\Fg^{\vee})$. 
Moreover, there exists a unique isomorphism 
$\Psi_{\lambda}:\CB(\lambda) 
 \stackrel{\sim}{\rightarrow} 
 \mv(\lambda)$ of crystals for $U_{q}(\Fg^{\vee})$. 
\end{thm}
%
%
\begin{rem} \label{rem:int}
(1) Define a GGMS datum $\mu_{\bullet}=(\mu_{w})_{w \in W}$ by: 
$\mu_{w}=\lambda \in X_{\ast}(T) \subset \Fh_{\BR}$ for all $w \in W$. 
It is obvious that 
$\mu_{\bullet}$ is an MV datum, and 
the MV polytope $P_{\lambda}:=P(\mu_{\bullet})$ is 
an element of $\mv(\lambda)$ whose weight 
is $\lambda \in X_{\ast}(T)$. 
Therefore, under the isomorphism $\Psi_{\lambda}$ of 
Theorem~\ref{thm:kam-int}, the highest weight element 
$u_{\lambda} \in \CB(\lambda)$ of weight $\lambda$ is sent to 
the MV polytope $P_{\lambda} \in \mv(\lambda)$.

(2) We know from \cite{A} that the polytope
$P_{w_{0} \cdot \lambda}:=\Conv (W \cdot \lambda)$ 
is an element of $\mv(\lambda)$, and is the lowest weight element 
of weight $w_{0} \cdot \lambda \in X_{\ast}(T)$; 
in fact, the GGMS (hence MV) datum of $P_{w_{0} \cdot \lambda}$ 
is given by: $\mu_{w}=ww_{0} \cdot \lambda \in \Fh_{\BR}$, $w \in W$ 
(see \cite{Kam1} and also \S\ref{subsec:extremal} below). 
Therefore, under the isomorphism $\Psi_{\lambda}$ of 
Theorem~\ref{thm:kam-int}, the lowest weight element 
$u_{w_{0} \cdot \lambda} \in \CB(\lambda)$ 
of weight $w_{0} \cdot \lambda$ is sent to 
the MV polytope $P_{w_{0} \cdot \lambda} \in \mv(\lambda)$.
\end{rem}

%
\section{
  MV polytopes lying in a Demazure crystal and 
  an opposite Demazure crystal.
}
\label{sec:demazure}

We fix (once and for all) 
an arbitrary dominant coweight 
$\lambda \in X_{\ast}(T) \subset \Fh_{\BR}$. 

%
\subsection{Demazure crystals.}
\label{subsec:demazure}
Let $x \in W$. 
The Demazure module $V_{x}(\lambda)$ is defined to be 
the $U_{q}^{+}(\Fg^{\vee})$-submodule of $V(\lambda)$ 
generated by the one-dimensional weight space 
$V(\lambda)_{x \cdot \lambda}  \subset V(\lambda)$ of 
weight $x \cdot \lambda \in X_{\ast}(T) \subset \Fh_{\BR}$. 
Recall from \cite{Kas} that 
the Demazure crystal $\CB_{x}(\lambda)$ is 
a subset of $\CB(\lambda)$ such that 
%
%
\begin{equation} \label{eq:dem-mod}
V(\lambda) \supset 
V_{x}(\lambda) = 
 \bigoplus_{b \in \CB_{x}(\lambda)}
 \BC(q) G_{\lambda}(b),
\end{equation}
where $G_{\lambda}(b)$, $b \in \CB(\lambda)$, 
form the lower global basis of $V(\lambda)$.
%
%
\begin{rem} \label{rem:demcos}
If $x,\,y \in W$ satisfies 
$x \cdot \lambda=y \cdot \lambda$, then 
we have $V_{x}(\lambda)=V_{y}(\lambda)$ since 
$V(\lambda)_{x \cdot \lambda}=
 V(\lambda)_{y \cdot \lambda}$. 
Therefore, it follows from \eqref{eq:dem-mod} that 
$\CB_{x}(\lambda)=\CB_{y}(\lambda)$. 
\end{rem}

We know from \cite[Proposition~3.2.3]{Kas} that 
the Demazure crystals $\CB_{x}(\lambda)$, 
$x \in W$, are characterized by the inductive relations: 
%
%
\begin{align}
& \CB_{e}(\lambda)=
   \bigl\{u_{\lambda}\bigr\}, \label{eq:demint1} \\[1.5mm]
& \CB_{x}(\lambda)=
    \bigcup_{N \ge 0} f_{j}^{N} 
     \CB_{s_{j}x}(\lambda) 
     \setminus \{\bzero\}
  \quad \text{for $x \in W$ and $j \in I$ with $s_{j}x < x$}. 
  \label{eq:demint2}
\end{align}

In addition, 
we know from \cite[Proposition~3.2.5]{Kas} that 
there exists a unique family $\CB_{x}(\infty)$, $x \in W$, of 
subsets of $\CB(\infty)$ satisfying 
the inductive relations: 
%
%
\begin{align}
& \CB_{e}(\infty)=
   \bigl\{u_{\infty}\bigr\}, \label{eq:deminf1} \\[1.5mm]
& \CB_{x}(\infty)=
    \bigcup_{N \ge 0} f_{j}^{N} 
     \CB_{s_{j}x}(\infty)
  \quad \text{for $x \in W$ and $j \in I$ with $s_{j}x < x$}; 
  \label{eq:deminf2}
\end{align}
for each $x \in W$, the subset $\CB_{x}(\infty) \subset \CB(\infty)$ 
is also called the Demazure crystal associated to $x$. 

%
\subsection{Condition for an MV polytope to lie in a Demazure crystal.}
\label{subsec:main}

Let us fix an arbitrary $x \in W$, and denote by $p$ 
the length $\ell(xw_{0})$ of $xw_{0} \in W$. 
For each reduced word 
$\bi=(i_{1},\,i_{2},\,\dots,\,i_{m}) \in R(w_{0})$, 
with $m=\ell(w_{0})$, we set 
\begin{equation*}
S(xw_{0},\,\bi)=\left\{
\begin{array}{l|l}
 (a_{1},\,a_{2},\,\dots,\,a_{p}) \in [1,\,m]^{p} \ & \ 
\begin{array}{l}
1 \le a_{1} < a_{2} < \cdots < a_{p} \le m \\[1.5mm]
\si{a_1}\si{a_2} \cdots \si{a_p}=xw_{0}
\end{array}
\end{array}
\right\},
\end{equation*}
where $[1,\,m]:=\bigl\{a \in \BZ \mid 1 \le a \le m\bigr\}$.
Note that for each sequence 
$(a_{1},\,a_{2},\,\dots,\,a_{p}) \in S(xw_{0},\,\bi)$, 
$xw_{0}=\si{a_1}\si{a_2} \cdots \si{a_p}$ is a reduced expression 
since $p=\ell(xw_{0})$.
Let $\mv_{x}(\lambda)$ (resp., $\mv_{x}(\infty)$) 
denote the subset of $\mv(\lambda)$ 
(resp., $\mv(\infty)$) consisting of those elements 
$P(\mu_{\bullet})$ with GGMS datum 
$\mu_{\bullet}=(\mu_{w})_{w \in W}$
which satisfy the condition: 

\vsp

(Dem.) for some $\bi \in R(w_{0})$, there exists a sequence 
$\ba=(a_{1},\,a_{2},\,\dots,\,a_{p}) \in S(xw_{0},\,\bi)$ 
such that $\Ni{a_{q}}=\Ni{a_{q}}(\mu_{\bullet})=0$ 
for all $1 \le q \le p$ 
(for the definition of $\Ni{l} \in \BZ_{\ge 0}$, 
$1 \le l \le m$, see \eqref{eq:n}). 

\vsp

The following is the first main result
of this paper; its proof will be given 
in the next subsection. 
%
%
\begin{thm} \label{thm:main}
Keep the notation above. 

{\rm (1)} 
Under the isomorphism 
$\Psi_{\lambda}:\CB(\lambda) 
 \stackrel{\sim}{\rightarrow} 
 \mv(\lambda)$ of Theorem~\ref{thm:kam-int}, 
the Demazure crystal 
$\CB_{x}(\lambda) \subset \CB(\lambda)$ 
associated to $x \in W$ is mapped to 
$\mv_{x}(\lambda)$, that is, 
\begin{equation*}
\Psi_{\lambda}(\CB_{x}(\lambda))=\mv_{x}(\lambda).
\end{equation*}

{\rm (2)} 
Under the isomorphism 
$\Psi:\CB(\infty) \stackrel{\sim}{\rightarrow} \mv(\infty)$ 
of Theorem~\ref{thm:kam-inf}, 
the Demazure crystal $\CB_{x}(\infty) \subset \CB(\infty)$ 
associated to $x \in W$ is mapped to 
$\mv_{x}(\infty)$, that is,
\begin{equation*}
\Psi(\CB_{x}(\infty))=\mv_{x}(\infty).
\end{equation*}
\end{thm}

For each reduced word 
$\bi=(i_{1},\,i_{2},\,\dots,\,i_{m}) \in R(w_{0})$, 
we set 
\begin{equation*}
\ha{S}(xw_{0},\,\bi) 
=
\left\{
 \begin{array}{l|l}
 (a_{1},\,a_{2},\,\dots,\,a_{l}) \in [1,\,m]^{l} \ & \ 
\begin{array}{l}
0 \le l \le m \\[1.5mm]
1 \le a_{1} < a_{2} < \cdots < a_{l} \le m \\[1.5mm]
\si{a_1}\si{a_2} \cdots \si{a_l}=xw_{0}
\end{array}
\end{array}
\right\};
\end{equation*}
note that we can replace the condition 
``$0 \le l \le m$'' in the definition above 
with ``$p \le l \le m$'' since $\ell(xw_{0})=p$. 
It is obvious that $S(xw_{0},\,\bi) \subset \ha{S}(xw_{0},\,\bi)$, 
since the set $S(xw_{0},\,\bi)$ is identical to the subset of 
$\ha{S}(xw_{0},\,\bi)$ consisting of those elements 
$(a_{1},\,a_{2},\,\dots,\,a_{l})$ for which $l=p$. 
Also, for each sequence 
$(a_{1},\,a_{2},\,\dots,\,a_{l}) \in \ha{S}(xw_{0},\,\bi)$, 
there exists a subsequence of it 
which is contained in $S(xw_{0},\,\bi)$ 
(see, for example, \cite[Chap.\,5, \S4, Corollary~2]{MP}).
Using these facts, we obtain immediately
the following corollary of Theorem~\ref{thm:main}.
%
%
\begin{cor} \label{cor:main}
The Demazure crystal $\mv_{x}(\lambda)$
{\rm(}resp., $\mv_{x}(\infty)${\rm)}
is identical to the subset of $\mv(\lambda)$ 
{\rm(}resp., $\mv(\infty)${\rm)} consisting of those elements 
$P(\mu_{\bullet}) \in \mv(\lambda)$
{\rm(}resp., $\in \mv(\infty)${\rm)}
with GGMS datum $\mu_{\bullet}=(\mu_{w})_{w \in W}$
which satisfy the following condition\,{\rm:} 
for some $\bi \in R(w_{0})$, there exists a sequence 
$(a_{1},\,a_{2},\,\dots,\,a_{l}) \in \ha{S}(xw_{0},\,\bi)$ 
such that $\Ni{a_{q}}=\Ni{a_{q}}(\mu_{\bullet})=0$ 
for all $1 \le q \le l$. 
\end{cor}

%
\subsection{Proof of Theorem~\ref{thm:main}.}
\label{subsec:prf-main}

We keep the notation and assumptions of \S\ref{subsec:main}. 
The following proposition plays a key role in the proof of 
Theorem~\ref{thm:main}. 
%
%
\begin{prop} \label{prop:every}
Let $P(\mu_{\bullet})$ be an MV polytope in 
$\mv_{x}(\lambda)$ or $\mv_{x}(\infty)$ 
with GGMS datum $\mu_{\bullet}=(\mu_{w})_{w \in W}$. 
Then, for every $\bj \in R(w_{0})$, 
there exists a sequence 
$\bb=(b_{1},\,b_{2},\,\dots,\,b_{p}) \in S(xw_{0},\,\bj)$ 
such that $\Nj{b_q}=\Nj{b_q}(\mu_{\bullet})=0$ 
for all $1 \le q \le p$. 
\end{prop}

\begin{proof}
In view of Remark~\ref{rem:braid} along with
the definitions of 
$\mv_{x}(\lambda)$ and $\mv_{x}(\infty)$, 
it suffices to show the following claim. 

\begin{claim*}
Let $\bi \in R(w_{0})$ be such that 
there exists a sequence 
$\ba=(a_{1},\,a_{2},\,\dots,\,a_{p}) \in 
 S(xw_{0},\,\bi)$ for which 
$\Ni{a_q}=\Ni{a_q}(\mu_{\bullet})=0$ for all $1 \le q \le p$. 
Assume that $\bj \in R(w_{0})$ is related to 
$\bi \in R(w_{0})$ by a $2$-move or a $3$-move.
Then, there exists a sequence 
$\bb=(b_{1},\,b_{2},\,\dots,\,b_{p}) \in 
 S(xw_{0},\,\bj)$ for which 
$\Nj{b_q}=0$ for all $1 \le q \le p$. 
\end{claim*}

\noindent
{\it Proof of Claim.}
We give a proof in the case 
that $\bi$ and $\bj$ are related 
by a $3$-move as in \eqref{eq:3move}, i.e., 

\begin{equation*}
\begin{array}{l}
\bi=(i_{1},\,\dots,\,i_{k},\,i,\,j,\,i,\,i_{k+4},\,\dots,\,i_{m}), \\[1.5mm]
\bj=(i_{1},\,\dots,\,i_{k},\,j,\,i,\,j,\,i_{k+4},\,\dots,\,i_{m})
\end{array}
\end{equation*}
for some indices $i,\,j \in I$ with $a_{ij}=a_{ji}=-1$ and 
an integer $0 \le k \le m-3$; the proof for the case of $2$-move 
is similar (or, even simpler).
Now, for $0 \le l_{1},\,l_{2} \le m$, 
we denote by $[l_{1},\,l_{2}] \cap \ba$ 
the subsequence of $\ba$ consisting of 
those $a_{q}$'s such that 
$a_{q} \in [l_{1},\,l_{2}]:=
 \bigl\{l \in \BZ \mid l_{1} \le l \le l_{2}\bigr\}$; 
by convention, we set 
$[l_{1},\,l_{2}] \cap \ba := \emptyset$ 
if there is no $a_{q}$ such that $l_{1} \le a_{q} \le l_{2}$. 
Then we have
\begin{align*}
& [1,\,k] \cap \ba = 
  (a_{1},\,\dots,\,a_{u_0}), \\
& [k+1,\,k+3] \cap \ba = 
  (a_{u_0+1},\,\dots,\,a_{u_1}), \\
& [k+4,\,m] \cap \ba = 
  (a_{u_1+1},\,\dots,\,a_{p})
\end{align*}
for $0 \le u_{0} \le u_{1} \le p$. 
We set $\ba':=[k+1,\,k+3] \cap \ba$. 
Since $xw_{0}=\si{a_1}\si{a_2} \cdots \si{a_p}$ is 
a reduced expression and $i_{k+1}=i_{k+3}=i$, 
the subsequence $\ba'$ cannot be equal to $(k+1,\,k+3)$. 
Namely, the subsequence $\ba'$ is 
equal to one of the following: 
$\emptyset$, $(k+1)$, $(k+2)$, 
$(k+3)$, $(k+1,\,k+2)$, 
$(k+2,\,k+3)$, $(k+1,\,k+2,\,k+3)$. 
Note that 
if $\ba'= (k+2)$ (and hence $\Ni{k+2}=0$), 
then $\Nj{k+1}=0$ or $\Nj{k+3}=0$ by \eqref{eq:3bm}. 
We define a strictly increasing sequence 
$\bb=(b_{1},\,b_{2},\,\dots,\,b_{p}) \in [1,\,m]^{p}$ 
as follows: 
\begin{equation*}
b_{q}=a_{q} \quad 
 \text{for $1 \le q \le u_{0}$ and 
 $u_{1}+1 \le q \le p$}, 
\end{equation*}
\begin{align*}
& (b_{u_{0}+1},\,\dots,\,b_{u_{1}}) = 
\begin{cases}
\emptyset & 
  \text{if $\ba'=\emptyset$}, \\[1.5mm]
(k+2) & 
  \text{if $\ba'=(k+1)$ or $(k+3)$}, \\[1.5mm]
(k+1) & 
  \text{if $\ba'=(k+2)$ and $\Nj{k+1}=0$}, \\[1.5mm]
(k+3) & 
  \text{if $\ba'=(k+2)$ and $\Nj{k+1} \ne 0$, $\Nj{k+3}=0$}, \\[1.5mm]
(k+2,\,k+3) & 
  \text{if $\ba'=(k+1,\,k+2)$}, \\[1.5mm]
(k+1,\,k+2) & 
  \text{if $\ba'=(k+2,\,k+3)$}, \\[1.5mm]
(k+1,\,k+2,\,k+3) & 
  \text{if $\ba'=(k+1,\,k+2,\,k+3)$}.
\end{cases}
\end{align*}
Then, it follows 
that $\bb \in S(xw_{0},\,\bj)$.
Also, by using \eqref{eq:3bm}, we can easily verify that 
$\Nj{b_{q}}=0$ for all $1 \le q \le p$.
This proves the claim, completing the proof of 
the proposition.
\end{proof}
%
%
\begin{rem} \label{rem:every}
By Proposition~\ref{prop:every}, we can replace 
the phrase ``for some $\bi \in R(w_{0})$'' in 
the condition (Dem.) (and in Corollary~\ref{cor:main}) 
for $\mv_{x}(\lambda)$ and $\mv_{x}(\infty)$ 
with the phrase ``for every $\bi \in R(w_{0})$''. 
\end{rem}

%
\begin{prop} \label{prop:demply}
{\rm (1)} We have 
$\mv_{e}(\lambda)=\bigl\{P_{\lambda}\bigr\}$ and 
$\mv_{e}(\infty)=\bigl\{P_{0}\bigr\}$. 

{\rm (2)}
Let $x \in W$ and $j \in I$ be such that $s_{j}x < x$. 
Then, we have
%
%
\begin{align}
& \mv_{x}(\lambda)=
  \bigcup_{N \ge 0} 
  f_{j}^{N} \mv_{s_{j}x}(\lambda) \setminus \{\bzero\}, 
  \label{eq:demintp} \\[3mm]
& \mv_{x}(\infty)=
  \bigcup_{N \ge 0} 
  f_{j}^{N} \mv_{s_{j}x}(\infty).
  \label{eq:deminfp}
\end{align}
\end{prop}

\begin{proof}
(1) It is obvious from the definitions of 
$\mv_{e}(\lambda)$ and $\mv_{e}(\infty)$ that 
$P_{\lambda} \in \mv_{e}(\lambda)$ and 
$P_{0} \in \mv_{e}(\infty)$. 
Now, let $P=P(\mu_{\bullet}) \in \mv_{e}(\lambda)$ 
(resp., $P=P(\mu_{\bullet}) \in \mv_{e}(\infty)$) be 
an MV polytope with GGMS datum 
$\mu_{\bullet}=(\mu_{w})_{w \in W}$.
Note that the set $S(ew_{0},\,\bi)=S(w_{0},\,\bi)$ 
consists of the single element $(1,\,2,\,\dots,\,m)$ 
for all $\bi=(i_{1},\,i_{2},\,\dots,\,i_{m}) \in R(w_{0})$. 
Therefore, we see from Proposition~\ref{prop:every} 
that $\Ni{l}(\mu_{\bullet})=0$ 
for all $\bi \in R(w_{0})$ and $1 \le l \le m$, 
which implies that
$\mu_{\wi{0}}=\mu_{\wi{1}}= \cdots = \mu_{\wi{m}} = 
 \mu_{w_{0}}=\lambda$ (resp., $=0$) 
for all $\bi \in R(w_{0})$. 
Here we recall the well-known fact 
(see, for example, \cite[Proposition~3.1.2]{BB}) that 
for each $w \in W$, there exist some $\bi \in R(w_{0})$ and 
an integer $0 \le l \le m$ such that $w=\wi{l}$.
From this fact, 
it follows that $\mu_{w}=\lambda$ 
(resp., $\mu_{w}=0$) for all $w \in W$, 
and hence $P=P_{\lambda}$ (resp., $P=P_{0}$). 
Thus we have shown part (1).

(2) We denote by $p$ the length $\ell(xw_{0})$ of $xw_{0}$ 
as in \S\ref{subsec:main}. Then the length 
$\ell(s_{j}xw_{0})$ of $s_{j}xw_{0}$ is equal to $p+1$, 
since $s_{j}x < x$ by assumption. We take and fix 
$\bi=(i_{1},\,i_{2},\,\dots,\,i_{m}) \in R(w_{0})$ 
such that $i_{1}=j$. 

First we show the inclusion $\supset$ of \eqref{eq:demintp} 
(resp., \eqref{eq:deminfp}).
Let $P=P(\mu_{\bullet}) \in \mv_{s_{j}x}(\lambda)$ 
(resp., $P=P(\mu_{\bullet}) \in \mv_{s_{j}x}(\infty)$) be 
an MV polytope with GGMS datum $\mu_{\bullet}=(\mu_{w})_{w \in W}$. 
Assume that $f_{j}^{N}P \ne \bzero$ for some $N \ge 0$, and 
let $\mu_{\bullet}'=(\mu_{w}')_{w \in W}$ denote
the GGMS (hence MV) datum of the MV polytope $f_{j}^{N}P$, i.e., 
$f_{j}^{N}P=P(\mu_{\bullet}')$. 
In order to show that 
$f_{j}^{N}P=P(\mu_{\bullet}') \in \mv_{x}(\lambda)$ 
(resp., $\in \mv_{x}(\infty)$), 
we define a sequence $\bb=(b_{1},\,b_{2},\,\dots,\,b_{p}) \in 
S(xw_{0},\,\bi)$ as follows. 
By Proposition~\ref{prop:every}, 
there exists a sequence 
$\ba=(a_{1},\,a_{2},\,\dots,\,a_{p+1}) \in 
S(s_{j}xw_{0},\,\bi)$ such that 
$\Ni{a_{q}}(\mu_{\bullet})=0$ 
for all $1 \le q \le p+1$; recall that 
$\si{a_{1}}\si{a_{2}} \cdots \si{a_{p+1}}$ is 
a reduced expression of $s_{j}xw_{0}$ 
by the definition of $S(s_{j}xw_{0},\,\bi)$. 
Since $s_{j}xw_{0} > xw_{0}$, 
it follows from the ``exchange condition''
(see, for example, \cite[Chap.\,5, \S3, Proposition~2]{MP})
that $xw_{0}$ has a reduced expression of the form: 
$xw_{0}=\si{a_{1}} \cdots \si{a_{q-1}}\si{a_{q+1}} \cdots \si{a_{p+1}}$
for some (uniquely determined) $1 \le q \le p+1$. 
Here we note that if $a_{1}=1$, 
then $q=1$ (i.e., $\si{a_{1}}$ is removed) 
since $i_{a_{1}}=i_{1}=j$ and $s_{j}xw_{0} > xw_{0}$. 
Thus, the sequence 
$(a_{1},\,\dots,\,a_{q-1},\,a_{q+1},\,\dots,\,a_{p+1})$ 
obtained from the sequence $\ba$ by removing $a_{q}$ is 
an element of $S(xw_{0},\,\bi)$. 
We now define $\bb=(b_{1},\,b_{2},\,\dots,\,b_{p})$ to be 
this sequence 
$(a_{1},\,\dots,\,a_{q-1},\,a_{q+1},\,\dots,\,a_{p+1}) 
\in S(xw_{0},\,\bi)$. It follows that 
$b_{q} \ge 2$ for all $1 \le q \le p$. 
%
%
\begin{claim} \label{dp-c1}
We have $\Ni{b_{q}}(\mu_{\bullet}')=0$ for all $1 \le q \le p$. 
\end{claim}

\noindent
{\it Proof of Claim~\ref{dp-c1}. }
Note that $s_{j}\wi{l} < \wi{l}$ for all $1 \le l \le m$, 
since $\bi=(i_{1},\,i_{2},\,\dots,\,i_{m}) \in R(w_{0})$ is 
such that $i_{1}=j$. Therefore, it follows from the definition of 
the Kashiwara operator $f_{j}$ that 
$\mu_{\wi{l}}'=\mu_{\wi{l}}$ for all $1 \le l \le m$. 
Hence we have $\Ni{l}(\mu_{\bullet}')=\Ni{l}(\mu_{\bullet})$ 
for all $2 \le l \le m$. 
Because $b_{q} \ge 2$ for all $1 \le q \le p$, and 
$\Ni{a_{q}}(\mu_{\bullet})=0$ for all $1 \le q \le p+1$, 
we deduce from the definition of the sequence $\bb$ that 
$\Ni{b_q}(\mu_{\bullet}')=0$ for all $1 \le q \le p$. 
This proves the claim.~\bqed

\vsp

It follows from Claim~\ref{dp-c1} that 
$f_{j}^{N}P=P(\mu_{\bullet}') \in \mv_{x}(\lambda)$ 
(resp., $\in \mv_{x}(\infty)$), and hence the inclusion 
$\supset$ of \eqref{eq:demintp} (resp., \eqref{eq:deminfp}) 
is verified. 

\vsp

The inclusion $\subset$ of \eqref{eq:demintp} 
(resp., \eqref{eq:deminfp}) can be shown similarly.
Let $P=P(\mu_{\bullet}) \in \mv_{x}(\lambda)$ 
(resp., $P=P(\mu_{\bullet}) \in \mv_{x}(\infty)$) be 
an MV polytope with GGMS datum $\mu_{\bullet}=(\mu_{w})_{w \in W}$. 
Let $\mu_{\bullet}'=(\mu_{w}')_{w \in W}$ denote 
the GGMS (hence MV) datum of the MV polytope 
$e_{j}^{\max}P:=e_{j}^{\ve_{j}(P)}P$, i.e., 
$e_{j}^{\max}P=P(\mu_{\bullet}')$. 
We define a sequence 
$\bb=(b_{1},\,b_{2},\,\dots,\,b_{p+1}) \in 
S(s_{j}xw_{0},\,\bi)$ as follows. 
By Proposition~\ref{prop:every}, there exists 
a sequence $\ba=(a_{1},\,a_{2},\,\dots,\,a_{p}) \in 
S(xw_{0},\,\bi)$ such that 
$\Ni{a_{q}}(\mu_{\bullet})=0$ 
for all $1 \le q \le p$.
Note that $a_{1} \ge 2$, since $s_{j}xw_{0} > xw_{0}$ and 
$\bi=(i_{1},\,i_{2},\,\dots,\,i_{m}) \in R(w_{0})$ is 
such that $i_{1}=j$. 
We now define $\bb=(b_{1},\,b_{2},\,\dots,\,b_{p+1})$ 
to be the sequence 
$\bb=(1,\,a_{1},\,a_{2},\,\dots,\,a_{p})$.
Then it follows immediately that 
$\bb \in S(s_{j}xw_{0},\,\bi)$ since $i_{1}=j$. 
%
%
\begin{claim} \label{dp-c2}
We have $\Ni{b_{q}}(\mu_{\bullet}')=0$ for all $1 \le q \le p+1$. 
\end{claim}

\noindent
{\it Proof of Claim~\ref{dp-c2}. }
Note that $e_{j}(e_{j}^{\max}P) = \bzero$ by the definition. 
Hence it follows from the definition of 
the Kashiwara operator $e_{j}$ that 
$\mu_{\wi{0}}'=\mu_{e}'=\mu_{s_{j}}'=\mu_{\wi{1}}'$. 
Also, it follows from the definition of $e_{j}$ that 
$\mu_{\wi{l}}'=\mu_{\wi{l}}$ for all $1 \le l \le m$, 
since $i_{1}=j$ implies $s_{j}\wi{l} < \wi{l}$.  
Therefore, we have 
$\Ni{b_1}(\mu_{\bullet}')=\Ni{1}(\mu_{\bullet}')=0$, and 
$\Ni{l}(\mu_{\bullet}')=\Ni{l}(\mu_{\bullet})$ 
for all $2 \le l \le m$. 
Because $a_{q} \ge a_{1} \ge 2$ for all $1 \le q \le p$, and 
$\Ni{a_{q}}(\mu_{\bullet})=0$ for all $1 \le q \le p$, 
we deduce that $\Ni{b_q}(\mu_{\bullet}')=
\Ni{a_{q-1}}(\mu_{\bullet}')=
\Ni{a_{q-1}}(\mu_{\bullet})=0$ for all $2 \le q \le p+1$. 
Hence we have 
$\Ni{b_q}(\mu_{\bullet}')=0$ for all $1 \le q \le p+1$. 
This proves the claim. \bqed

\vsp

It follows from Claim~\ref{dp-c2} that 
$e_{j}^{\max}P=P(\mu_{\bullet}') \in \mv_{s_{j}x}(\lambda)$, 
(resp., $\in \mv_{s_{j}x}(\infty)$). 
Therefore, we conclude that 
$P \in f_{j}^{\ve_{j}(P)}\mv_{s_jx}(\lambda)$ 
(resp., $\in f_{j}^{\ve_{j}(P)}\mv_{s_jx}(\infty)$), 
which implies the inclusion 
$\subset$ of \eqref{eq:demintp} 
(resp., \eqref{eq:deminfp}). 
Thus, we have shown \eqref{eq:demintp} and \eqref{eq:deminfp}, 
thereby completing the proof of the proposition.
\end{proof}

Part (1) of Theorem~\ref{thm:main} follows 
immediately by combining \eqref{eq:demint1}, 
\eqref{eq:demint2} and the corresponding assertions
for $\mv_{x}(\lambda)$, $x \in W$, in 
Proposition~\ref{prop:demply}; 
also, recall Remark~\ref{rem:int}\,(1). 
Similarly, part (2) of Theorem~\ref{thm:main} follows 
immediately by combining \eqref{eq:deminf1}, 
\eqref{eq:deminf2} and the corresponding assertions
for $\mv_{x}(\infty)$, $x \in W$, in 
Proposition~\ref{prop:demply}; also, recall 
Remark~\ref{rem:inf}. 

%
\subsection{Opposite Demazure crystals.}
\label{subsec:opdem}

Let $x \in W$. 
The opposite Demazure module $V^{x}(\lambda)$ is defined to be 
the $U_{q}^{-}(\Fg^{\vee})$-submodule of $V(\lambda)$ 
generated by the one-dimensional weight space 
$V(\lambda)_{x \cdot \lambda}  \subset V(\lambda)$ of 
weight $x \cdot \lambda \in X_{\ast}(T) \subset \Fh_{\BR}$. 
Recall from \cite{Kas} that 
the opposite Demazure crystal $\CB^{x}(\lambda)$ is 
a subset of $\CB(\lambda)$ such that 
%
%
\begin{equation} \label{eq:opdem-mod}
V(\lambda) \supset 
V^{x}(\lambda) = 
 \bigoplus_{b \in \CB^{x}(\lambda)}
 \BC(q) G_{\lambda}(b),
\end{equation}
where $G_{\lambda}(b)$, $b \in \CB(\lambda)$, 
form the lower global basis of $V(\lambda)$.

%
\begin{rem} \label{rem:opdemcos}
If $x,\,y \in W$ satisfies 
$x \cdot \lambda=y \cdot \lambda$, then 
we have $V^{x}(\lambda)=V^{y}(\lambda)$ since 
$V(\lambda)_{x \cdot \lambda}=
 V(\lambda)_{y \cdot \lambda}$. 
Therefore, it follows from \eqref{eq:opdem-mod} that 
$\CB^{x}(\lambda)=\CB^{y}(\lambda)$. 
\end{rem}

We know from \cite[\S4]{Kas} that 
the opposite Demazure crystals $\CB^{x}(\lambda)$, 
$x \in W$, are characterized by 
the (descending) inductive relations: 
%
%
\begin{align}
& \CB^{w_{0}}(\lambda)=
   \bigl\{u_{w_{0} \cdot \lambda}\bigr\}, 
   \label{eq:opdemint1} \\[1.5mm]
& \CB^{x}(\lambda)=
    \bigcup_{N \ge 0} e_{j}^{N} 
     \CB^{s_{j}x}(\lambda) 
     \setminus \{\bzero\}
  \quad \text{for $x \in W$ and $j \in I$ with $s_{j}x > x$}. 
  \label{eq:opdemint2}
\end{align}
Furthermore, 
we see from \cite[Proposition~4.2\,(i)]{Kas} and 
\eqref{eq:opdemint2} that 
for $x \in W$ and $j \in I$ with $s_{j}x > x$, 
$b \in \CB^{x}(\lambda)$ holds if and only if 
$f_{j}^{\max}b \in \CB^{s_{j}x}(\lambda)$ holds, 
where for $b \in \CB(\lambda)$ and $j \in I$, we set 
$f_{j}^{\max}b:=f_{j}^{\vp_{j}(b)}b$, with 
$\vp_{j}(b):=\max \bigl\{N \ge 0 \mid f_{j}^{N}b \ne \bzero\bigr\}$. 
Using this fact successively, we obtain the following lemma 
(cf. \cite[Proposition~9.1.3\,(2)]{Kasb} 
for a similar result for Demazure crystals).
%
%
\begin{lem} \label{lem:fmax}
Let $x \in W$, and let $i_{1},\,i_{2},\,\dots,\,i_{p}$ 
be a sequence of elements in $I$ such that 
$\ell(\si{p} \cdots \si{2}\si{1}x)=\ell(x)+p$ and 
$\si{p} \cdots \si{2}\si{1}x=w_{0}$. An element $b \in \CB(\lambda)$ 
is contained in the opposite Demazure crystal $\CB^{x}(\lambda)$ 
associated to $x$ if and only if 
$f_{i_p}^{\max} \cdots 
 f_{i_2}^{\max}
 f_{i_1}^{\max}b=u_{w_{0} \cdot \lambda}$. 
\end{lem}

%
\subsection{Condition for an MV polytope to lie in an opposite Demazure crystal.}
\label{subsec:op-main}

Let us fix an arbitrary $x \in W$, and denote by $p$ 
the length $\ell(xw_{0})$ of $xw_{0} \in W$. 
Let $\mv^{x}(\lambda)$ denote the subset of $\mv(\lambda)$ 
consisting of those elements $P(\mu_{\bullet}) \in \mv(\lambda)$ 
with GGMS datum $\mu_{\bullet}=(\mu_{w})_{w \in W}$ 
which satisfy the condition: 

\vsp

(Op. Dem.) 
for some $\bi=(i_{1},\,i_{2},\,\dots,\,i_{m}) \in R(w_{0})$ 
such that
$\wi{p}=\si{1}\si{2} \cdots \si{p}=xw_{0}$, there holds 
$\mu_{\wi{l}}=\wi{l}w_{0} \cdot \lambda$ 
for all $p \le l \le m$. 

\vsp

The following is the second main result of this paper. 
%
%
\begin{thm} \label{thm:opdem}
Under the isomorphism 
$\Psi_{\lambda}:\CB(\lambda) 
 \stackrel{\sim}{\rightarrow} 
 \mv(\lambda)$ of Theorem~\ref{thm:kam-int}, 
the opposite Demazure crystal 
$\CB^{x}(\lambda) \subset \CB(\lambda)$ 
associated to $x \in W$ is mapped to 
$\mv^{x}(\lambda)$, that is, 
\begin{equation*}
\Psi_{\lambda}(\CB^{x}(\lambda))=\mv^{x}(\lambda).
\end{equation*}
\end{thm}

\begin{proof}
First, we prove that 
$\Psi_{\lambda}(\CB^{x}(\lambda)) \subset \mv^{x}(\lambda)$.
Let $P=P(\mu_{\bullet}) \in \Psi_{\lambda}(\CB^{x}(\lambda))$ 
be an MV polytope with GGMS datum 
$\mu_{\bullet}=(\mu_{w})_{w \in W}$, and 
take an arbitrary 
$\bi=(i_{1},\,i_{2},\,\dots,\,i_{m}) \in R(w_{0})$ 
such that $\wi{p}=\si{1}\si{2} \cdots \si{p}=xw_{0}$.
We will show by descending induction on $l$ that 
$\mu_{\wi{l}}=\wi{l}w_{0} \cdot \lambda$ for every $p \le l \le m$. 
Since $\mu_{\wi{m}}=\mu_{w_{0}}=\lambda=w_{0}w_{0} \cdot \lambda=
\wi{m}w_{0} \cdot \lambda$, the assertion holds when $l=m$. 
For $p+1 \le l \le m$, 
we have
\begin{align*}
\mu_{\wi{l-1}}
 & =\mu_{\wi{l}}-\Ni{l}\wi{l-1} \cdot h_{i_l}
   \quad \text{by the length formula} \\
 & =\wi{l}w_{0} \cdot \lambda - \Ni{l}\wi{l-1} \cdot h_{i_l}
   \quad \text{by the induction hypothesis} \\
 & = \wi{l-1}\si{l} w_{0} \cdot \lambda - \Ni{l} \wi{l-1} \cdot h_{i_l} \\
 & = \wi{l-1} \cdot 
    \Bigl(
      w_{0} \cdot \lambda - 
      \pair{\alpha_{i_l}}{w_{0} \cdot \lambda}h_{i_l}
    \Bigr) - \Ni{l} \wi{l-1} \cdot h_{i_l} \\
 & = \wi{l-1}w_{0} \cdot \lambda -
    \bigl\{\pair{\alpha_{i_l}}{w_{0} \cdot \lambda}+\Ni{l}\bigr\}
    \wi{l-1} \cdot h_{i_l}. 
\end{align*}
Therefore, in order to show that 
$\mu_{\wi{l-1}}=\wi{l-1}w_{0} \cdot \lambda$, 
it suffices to prove the following claim. 
%
%
\begin{claim} \label{c:opdem2}
We have $\Ni{l}=-\pair{\alpha_{i_l}}{w_{0} \cdot \lambda}$. 
\end{claim}

\noindent 
{\it Proof of Claim~\ref{c:opdem2}.}
As in \cite[\S6.3]{Kam2}, we set
%
%
\begin{equation} \label{eq:opdem01}
\begin{array}{l}
N_{1}:=
 \vp_{i_1}(P), \quad
N_{2}:=
 \vp_{i_2}(f_{i_{1}}^{\max}P), \quad
N_{3}:=
 \vp_{i_3}(f_{i_{2}}^{\max}f_{i_{1}}^{\max}P), \quad \dots\dots \\[3mm]
\hspace*{15mm}
\dots\dots, \quad 
N_{m}:=
 \vp_{i_m}(f_{i_{m-1}}^{\max} \cdots f_{i_{2}}^{\max}f_{i_{1}}^{\max}P),
\end{array}
\end{equation}
where for $P \in \mv(\lambda)$ and $j \in I$, we set
$f_{j}^{\max}P:=f_{j}^{\vp_{j}(P)}P \in \mv(\lambda)$. 
Then we have $N_{p+1}=N_{p+2}=\cdots=N_{m}=0$. Indeed, 
it follows from the equality 
$\si{1}\si{2} \cdots \si{p}=xw_{0}$ that 
$\si{p} \cdots \si{2}\si{1}x=(xw_{0})^{-1}x=w_{0}$, 
and hence
\begin{equation*}
\ell(\si{p} \cdots \si{2}\si{1}x)=
\ell(w_{0})=m=(m-p)+p=\ell(x)+p, 
\end{equation*}
since $\ell(x)=\ell(w_{0})-\ell(xw_{0})=m-p$. 
Therefore, from Lemma~\ref{lem:fmax}, 
we deduce that 
\begin{equation*}
f_{i_{p}}^{\max} \cdots f_{i_{2}}^{\max}f_{i_{1}}^{\max}P = 
 P_{w_{0} \cdot \lambda}.
\end{equation*}
Because $P_{w_{0} \cdot \lambda} \in \mv(\lambda)$ is 
the lowest weight element of weight $w_{0} \cdot \lambda$ 
(recall Remark~\ref{rem:int}\,(2)), we have 
$\vp_{j}(P_{w_{0} \cdot \lambda})=0$ for all $j \in I$, 
and hence $N_{p+1}=N_{p+2}=\cdots=N_{m}=0$, as desired.
In particular, $N_{l}=0$ since $p+1 \le l \le m$. 
Also, we know from \cite[Theorem~6.6]{Kam2} 
(and the equation preceding it) that 
\begin{equation*}
N_{l}=\frac{1}{2}
 \pair{ \wi{l-1} \cdot \alpha_{i_l} }
      { \mu_{\wi{l-1}}+\mu_{\wi{l}} }. 
\end{equation*}
Consequently, we obtain 
%
%
\begin{equation} \label{eq:opdem05}
 \pair{ \wi{l-1} \cdot \alpha_{i_l} }
      { \mu_{\wi{l-1}} }
=
-\pair{ \wi{l-1} \cdot \alpha_{i_l} }
      { \mu_{\wi{l}} }.
\end{equation}
Here we note that by the length formula 
$\mu_{\wi{l}}-\mu_{\wi{l-1}}=\Ni{l}\wi{l-1} \cdot h_{i_l}$, 
there holds 
%
%
\begin{equation} \label{eq:opdem1}
\pair{\wi{l-1} \cdot \alpha_{i_l}}{\mu_{\wi{l}}-\mu_{\wi{l-1}}}=
\Ni{l} \pair{\wi{l-1} \cdot \alpha_{i_l}}{\wi{l-1} \cdot h_{i_l}}.
\end{equation}
The left-hand side of \eqref{eq:opdem1} is equal to: 
\begin{align*}
\pair{\wi{l-1} \cdot \alpha_{i_l}}{\mu_{\wi{l}}-\mu_{\wi{l-1}}} & =
\pair{\wi{l-1} \cdot \alpha_{i_l}}{\mu_{\wi{l}}} - 
\pair{\wi{l-1} \cdot \alpha_{i_l}}{\mu_{\wi{l-1}}} \\
& = 2 \pair{\wi{l-1} \cdot \alpha_{i_l}}{\mu_{\wi{l}}}
\quad \text{by \eqref{eq:opdem05}} \\
& = 2 \pair{\wi{l-1} \cdot \alpha_{i_l}}{\wi{l}w_{0} \cdot \lambda}
\quad \text{by the induction hypothesis} \\
& = 2 \pair{\wi{l-1} \cdot \alpha_{i_l}}{\wi{l-1}\si{l}w_{0} \cdot \lambda} \\
& = 2 \pair{\alpha_{i_l}}{\si{l}w_{0} \cdot \lambda} \\
& = -2 \pair{\alpha_{i_l}}{w_{0} \cdot \lambda}. 
\end{align*}
The right-hand side of \eqref{eq:opdem1} is equal to:
\begin{equation*}
\Ni{l} \pair{\wi{l-1} \cdot h_{i_l}}{\wi{l-1} \cdot \alpha_{i_l}}=
\Ni{l} \pair{h_{i_l}}{\alpha_{i_l}}=2\Ni{l}.
\end{equation*}
Substituting these equalities into \eqref{eq:opdem1}, we conclude 
that $\Ni{l}=-\pair{\alpha_{i_l}}{w_{0} \cdot \lambda}$. 
This proves the claim. \bqed

\vsp

Thus, we have shown that 
$\mu_{\wi{l}}=\wi{l}w_{0} \cdot \lambda$ for all $p \le l \le m$. 
This implies that $P \in \mv^{x}(\lambda)$, and hence that 
$\Psi_{\lambda}(\CB^{x}(\lambda)) \subset \mv^{x}(\lambda)$. 

Next, we prove the reverse inclusion 
$\Psi_{\lambda}(\CB^{x}(\lambda)) \supset \mv^{x}(\lambda)$.
Let $P=P(\mu_{\bullet}) \in \mv^{x}(\lambda)$ be an MV polytope 
with GGMS datum $\mu_{\bullet}=(\mu_{w})_{w \in W}$. 
Since $P \in \mv^{x}(\lambda)$, there exists 
$\bi=(i_{1},\,i_{2},\,\dots,\,i_{m}) \in R(w_{0})$ 
with $\wi{p}=\si{1}\si{2} \cdots \si{p}=xw_{0}$ such that 
$\mu_{\wi{l}}=\wi{l}w_{0} \cdot \lambda$ 
for all $p \le l \le m$. 
Define nonnegative integers 
$N_{l}$, $1 \le l \le m$, as in \eqref{eq:opdem01}:
\begin{equation*}
N_{l}:=\vp_{i_l}
 (f_{i_{l-1}}^{\max} \cdots f_{i_{2}}^{\max}f_{i_{1}}^{\max}P). 
\end{equation*}
%
%
\begin{claim} \label{c:opdem3}
We have $N_{l}=0$ for all $p+1 \le l \le m$. 
\end{claim}

\noindent 
{\it Proof of Claim~\ref{c:opdem3}.}
Let $p+1 \le l \le m$. We know from 
\cite[Theorem~6.6]{Kam2} 
(and the equation preceding it) that 
\begin{equation*}
N_{l}=\frac{1}{2}
 \pair{ \wi{l-1} \cdot \alpha_{i_l} }
      { \mu_{\wi{l-1}}+\mu_{\wi{l}} }. 
\end{equation*}
Using the equalities 
$\mu_{\wi{l-1}}=\wi{l-1}w_{0} \cdot \lambda$ and 
$\mu_{\wi{l}}=\wi{l}w_{0} \cdot \lambda$, 
we calculate as follows: 
\begin{align*}
\pair{ \wi{l-1} \cdot \alpha_{i_l} }
     { \mu_{\wi{l-1}}+\mu_{\wi{l}} } & = 
\pair{ \wi{l-1} \cdot \alpha_{i_l} }
     { \wi{l-1}w_{0} \cdot \lambda } +
\pair{ \wi{l-1} \cdot \alpha_{i_l} }
     { \wi{l}w_{0} \cdot \lambda } \\
& =
\pair{ \wi{l-1} \cdot \alpha_{i_l} }
     { \wi{l-1}w_{0} \cdot \lambda } +
\pair{ \wi{l-1} \cdot \alpha_{i_l} }
     { \wi{l-1}\si{l}w_{0} \cdot \lambda } \\
& =
\pair{ \alpha_{i_l} }
     { w_{0} \cdot \lambda } +
\pair{ \alpha_{i_l} }
     { \si{l}w_{0} \cdot \lambda } \\
& = 
\pair{ \alpha_{i_l} }{ w_{0} \cdot \lambda } -
\pair{ \alpha_{i_l} }{ w_{0} \cdot \lambda } =0.
\end{align*}
Therefore, we obtain $N_{l}=0$, as desired. \bqed

\vsp

We know (see, for example, \cite[Proposition~6.5]{Kam2}) 
that
\begin{equation*}
f_{i_{m}}^{\max} \cdots f_{i_{2}}^{\max}f_{i_{1}}^{\max}P = 
f_{i_{m}}^{N_{m}} \cdots f_{i_{2}}^{N_{2}}f_{i_{1}}^{N_{1}}P=
P_{w_{0} \cdot \lambda}.
\end{equation*}
Here, by Claim~\ref{c:opdem3}, we have
\begin{equation*}
f_{i_{m}}^{N_{m}} \cdots f_{i_{2}}^{N_{2}}f_{i_{1}}^{N_{1}}P=
f_{i_{p}}^{N_{p}} \cdots f_{i_{2}}^{N_{2}}f_{i_{1}}^{N_{1}}P=
f_{i_{p}}^{\max} \cdots f_{i_{2}}^{\max}f_{i_{1}}^{\max}P.
\end{equation*}
Hence we obtain
$f_{i_{p}}^{\max} \cdots f_{i_{2}}^{\max}f_{i_{1}}^{\max}P=
P_{w_{0} \cdot \lambda}$.
Since $\si{1}\si{2} \cdots \si{p}=xw_{0}$, we deduce that 
$\si{p} \cdots \si{2}\si{1}x=w_{0}$, and hence 
$\ell(\si{p} \cdots \si{2}\si{1}x)=\ell(x)+p$, as in 
the proof of Claim~\ref{c:opdem2}. 
Therefore, from Lemma~\ref{lem:fmax}, we conclude that 
$P \in \Psi_{\lambda}(\CB^{x}(\lambda))$. 
Thus, we have shown that 
$\Psi_{\lambda}(\CB^{x}(\lambda)) \supset \mv^{x}(\lambda)$, 
and hence that 
$\Psi_{\lambda}(\CB^{x}(\lambda))=\mv^{x}(\lambda)$. 
This completes the proof of the theorem. 
\end{proof}

\begin{rem}
We see from the proof of Theorem~\ref{thm:opdem} that
we can replace the phrase 
``for some $\bi=(i_{1},\,i_{2},\,\dots,\,i_{m}) \in R(w_{0})$''
in the condition (Op. Dem.) for $\mv^{x}(\lambda)$
with the phrase 
``for every $\bi=(i_{1},\,i_{2},\,\dots,\,i_{m}) \in R(w_{0})$''. 
\end{rem}

%
\section{
  Extremal MV polytopes and its relation 
  with opposite Demazure crystals.}
\label{sec:extremal}

We fix (once and for all) an arbitrary dominant coweight 
$\lambda \in X_{\ast}(T) \subset \Fh_{\BR}$. 
For each $x \in W$, we denote by 
$P_{x \cdot \lambda}$ the image of the extremal 
element $u_{x \cdot \lambda} \in \CB(\lambda)$ of 
weight $x \cdot \lambda \in X_{\ast}(T) \subset \Fh_{\BR}$ 
under the isomorphism 
$\Psi_{\lambda}:\CB(\lambda) 
 \stackrel{\sim}{\rightarrow} 
 \mv(\lambda)$
of Theorem~\ref{thm:kam-int}; 
we call $P_{x \cdot \lambda} \in \mv(\lambda)$ 
the extremal MV polytope
of weight $x \cdot \lambda$. 
Recall (from \cite[Lemma~4.2]{N}, for example) that 
for every reduced expression 
$x=\si{1}\si{2} \cdots \si{l}$ of $x$, 
we have $u_{x \cdot \lambda}=
 f_{i_1}^{\max}f_{i_2}^{\max} \cdots f_{i_l}^{\max} u_{\lambda}$.
The aim of this section is to give an explicit description 
as a pseudo-Weyl polytope for 
the extremal MV polytopes $P_{x \cdot \lambda}$, $x \in W$, 
and to give a polytopal condition for an MV polytope 
$P \in \mv(\lambda)$ to lie in the opposite Demazure crystal 
$\mv^{x}(\lambda)$ for $x \in W$.

%
\subsection{Explicit description of extremal MV polytopes.}
\label{subsec:extremal}

We fix $x \in W$, and denote by $p$ 
the length $\ell(xw_{0})$ of $xw_{0}$. 
Let $\bi=(i_{1},\,i_{2},\,\dots,\,i_{m})$ 
be an arbitrary element of $R(w_{0})$, with $m=\ell(w_{0})$. 
For a sequence 
$\ba=(a_{1},\,a_{2},\,\dots,\,a_{p}) \in S(xw_{0},\,\bi)$ and 
$0 \le l_{1},\,l_{2} \le m$, let $[l_{1},\,l_{2}] \cap \ba$ 
denote the subsequence of $\ba$ consisting of 
those $a_{q}$'s such that $a_{q} \in [l_{1},\,l_{2}]$.
Also, recall that the lexicographic ordering $\succeq$ on 
$S(xw_{0},\,\bi)$ is defined as follows: 
$(a_{1},\,a_{2},\,\dots,\,a_{p}) \succ
 (b_{1},\,b_{2},\,\dots,\,b_{p})$ 
if there exists some integer $1 \le q_0 \le p$ such that 
$a_{q}=b_{q}$ for all $1 \le q \le q_0-1$ 
and $a_{q_0} > b_{q_0}$; 
we denote by $\min S(xw_{0},\,\bi)$ the 
minimum element of $S(xw_{0},\,\bi)$ 
with respect to the lexicographic ordering.

We define a sequence 
$\xii{0},\,\xii{1},\,\dots,\,\xii{m}$ of elements 
of the $W$-orbit $W \cdot \lambda \ (\subset \Fh_{\BR})$ 
inductively by the following formula: 
%
%
\begin{equation} \label{eq:xii}
\xii{m}=\lambda, \qquad 
\xii{l-1}=
 \begin{cases}
 \xii{l} & \text{if $l$ appears in $\min S(xw_{0},\,\bi)$}, \\[1.5mm]
 s_{\bti{l}} \cdot \xii{l} & \text{otherwise}
 \end{cases}
\end{equation}
for $1 \le l \le m$, 
where we set $\bti{l}:=\wi{l-1} \cdot \alpha_{i_{l}}$ 
for $1 \le l \le m$, and denote 
by $s_{\beta} \in W$ the reflection 
with respect to a root $\beta$. 

\begin{rem}
It is well-known (see, for example, 
\cite[Chap.\,5, \S2, Lemma~2 and Proposition~3]{MP}) that 
$\bti{l}$, $1 \le l \le m$, exhaust all the positive roots. 
\end{rem}
%
%
\begin{ex} \label{ex:1sk}
We know (see, for example, \cite[Proposition~3.1.2]{BB}) that 
there exists $\bi \in R(w_{0})$ such that $\wi{p}=xw_{0}$. 
It is obvious that for this $\bi \in R(w_{0})$, 
$\min S(xw_{0},\,\bi)=(1,\,2,\,\dots,\,p)$. 
Hence it follows from the definition that 
$\xii{0}=\xii{1}= \cdots = \xii{p}$ and 
$\xii{l}=s_{\bti{l+1}}s_{\bti{l+2}} \cdots s_{\bti{m}} \cdot \lambda$ 
for all $p \le l \le m$. Moreover, we see 
from \eqref{eq:xi} below that 
$\xii{l}=\wi{l}w_{0} \cdot \lambda$ for all $p \le l \le m$. 
\end{ex}
%
%
\begin{lem} \label{lem:length}
We have $\xii{l}-\xii{l-1} \in \BZ_{\ge 0}(\wi{l-1} \cdot h_{i_{l}})$ 
for every $1 \le l \le m$. 
\end{lem}

The proof of this lemma will be given 
in \S\ref{subsec:length}. 
From this lemma, using \cite[Theorem~7.1]{Kam1}, 
we see that there exists a unique MV datum 
$\mu_{\bullet}^{x,\,\bi}=(\mu_{w}^{x,\,\bi})_{w \in W}$ 
such that $\mu_{\wi{l}}^{x,\,\bi}=\xii{l}$ 
for every $0 \le l \le m$. 
%
%
\begin{prop} \label{prop:1sk}
Let $\bj \in R(w_{0})$ be another reduced word for $w_{0}$, and 
define an MV datum $\mu_{\bullet}^{x,\,\bj}$ 
in the same way as above, with $\bi$ replaced by $\bj$. 
Then, we have $\mu_{\bullet}^{x,\,\bi}=\mu_{\bullet}^{x,\,\bj}$. 
\end{prop}

The proof of this proposition will be given 
in \S\ref{subsec:1sk}. It follows from 
this proposition that 
the MV datum $\mu_{\bullet}^{x}=(\mu_{w}^{x})_{w \in W}:=
\mu_{\bullet}^{x,\,\bi}$ does not depend on 
the choice of $\bi \in R(w_{0})$.
Moreover, the MV polytope $P(\mu_{\bullet}^{x})$ 
associated to the MV datum $\mu_{\bullet}^{x}$ 
(by \eqref{eq:poly}) is an element of $\mv(\lambda)$. 
Indeed, recall that each $w \in W$ can be written as 
$w=\wi{l}$ for some $\bi \in R(w_{0})$ and an integer 
$0 \le l \le m$. Then it follows that 
$\mu_{w}^{x}=\mu_{\wi{l}}^{x}=
 \mu_{\wi{l}}^{x,\,\bi}=\xii{l} \in W \cdot \lambda$
by the definition of $\mu_{\bullet}^{x}$. 
Since $w_{0}=\wi{m}$ for all $\bi \in R(w_{0})$, and 
since $\xii{m}=\lambda$, we get
$\mu_{w_{0}}^{x}=\lambda$. 
Also, because $\xii{l} \in W \cdot \lambda$ 
for all $\bi \in R(w_{0})$ and $0 \le l \le m$, 
we have 
%
%
\begin{align}
P(\mu_{\bullet}^{x}) & = 
 \Conv\,\bigl\{\mu_{w}^{x} \mid w \in W\bigr\} \nonumber \\
 & =
 \Conv\,\bigl\{\xii{l} \mid \bi \in R(w_{0}),\, 0 \le l \le m \bigr\} 
 \subset \Conv (W \cdot \lambda). \label{eq:inc}
\end{align}
%
%
\begin{thm} \label{thm:ext}
{\rm (1)} The weight of 
the MV polytope $P(\mu_{\bullet}^{x}) \in \mv(\lambda)$ 
is equal to $x \cdot \lambda$. 
Therefore, $P(\mu_{\bullet}^{x})$ is the extremal MV polytope 
$P_{x \cdot \lambda}$ of weight $x \cdot \lambda$. 

{\rm (2)} 
The extremal MV polytope 
$P_{x \cdot \lambda}=P(\mu_{\bullet}^{x})$ 
of weight $x \cdot \lambda$ is identical to 
the convex hull $\Conv(W_{\le x} \cdot \lambda)$ 
in $\Fh_{\BR}$ of the set $W_{\le x} \cdot \lambda$, 
where $W_{\le x}$ denotes the subset 
$\bigl\{z \in W \mid z \le x\bigr\}$ of $W$. 
\end{thm}

The proof of this theorem will be given 
in \S\ref{subsec:prf-thmext}. 
By combining Theorems~\ref{thm:opdem} and \ref{thm:ext}, 
we obtain a polytopal condition for an MV polytope 
$P \in \mv(\lambda)$ to lie in the opposite Demazure crystal 
$\mv^{x}(\lambda)$ for $x \in W$.

\begin{rem}
The statement of Theorem~\ref{thm:ext}\,(2) seems to be 
known to some experts although we were unable to find it explicitly 
stated in the literature (cf. \cite{AP}, and also \cite[\S4]{Atiyah}). 
In fact, as was reported to us by Kato \cite{Kat}, the equality 
$P_{x \cdot \lambda}=\Conv(W_{\le x} \cdot \lambda)$ 
follows from the first equality in \cite[(3.6)]{MV2} 
(more precisely, from \cite[Lemma~3.2]{Haines} or 
\cite[Lemme~5.2]{NP}) by using the geometry of the affine 
Grassmannian $\CG r$ and the finer Bruhat decomposition:
\begin{equation*}
G=\bigsqcup_{z \in W^{\lambda}_{\min}} N\dot{z}P_{\lambda}, 
\end{equation*}
where $P_{\lambda} \ (\supset B)$ is 
a standard parabolic subgroup of $G$ determined by the set of 
roots $\alpha$ for which $\pair{\alpha}{\lambda} \le 0$. 
%

However, the equality 
$P_{x \cdot \lambda}=\Conv(W_{\le x} \cdot \lambda)$ 
is not enough for our purposes. To get an understanding 
of the extremal MV polytope $P_{x \cdot \lambda}$ 
as a pseudo-Weyl polytope, we need to identify explicitly the vertex 
$\mu_{w}$ of $P_{x \cdot \lambda}$ for each $w \in W$; 
such data can be obtained from our construction above 
of $P_{x \cdot \lambda}$. In particular, it follows 
immediately from Theorems~\ref{thm:main} and \ref{thm:opdem} 
that $P_{x \cdot \lambda} \in 
\mv_{x}(\lambda) \cap \mv^{x}(\lambda)$. 
\end{rem}

%
\subsection{Proof of Lemma~\ref{lem:length}.}
\label{subsec:length}
Keep the notation and assumptions of \S\ref{subsec:extremal}. 
Let $\bi=(i_{1},\,i_{2},\,\dots,\,i_{m})$ 
be an arbitrary element of $R(w_{0})$. 
For each $0 \le l \le m$, we denote by 
$[l+1,\,m] \setminus \min S(xw_{0},\,\bi)$ 
the sequence obtained by enumerating (in increasing order)
the integers in $[l+1,\,m]$ that do not appear in the 
sequence $\min S(xw_{0},\,\bi)$. 
Now, we write the sequence 
$[l+1,\,m] \setminus \min S(xw_{0},\,\bi)$ as 
$(b_{1},\,b_{2},\,\dots,\,b_{t_l})$, and set 
\begin{equation*}
\yi{l}:=
 s_{\bti{b_1}}s_{\bti{b_2}} \cdots s_{\bti{b_{t_l}}}; 
\end{equation*}
if all the integers in $[l+1,\,m]$ appear in the sequence 
$\min S(xw_{0},\,\bi)$, then we set 
$[l+1,\,m] \setminus \min S(xw_{0},\,\bi)=\emptyset$, and 
$\yi{l}:=e \in W$. 
Note that $\xii{l}=\yi{l} \cdot \lambda$ 
for all $0 \le l \le m$. Also, for each $0 \le l \le m$, 
we write the sequence 
$[l+1,\,m] \cap \min S(xw_{0},\,\bi)$ as 
$(c_{1},\,c_{2},\,\dots,\,c_{u_l})$, 
with $u_l+t_l=m-l$, and set 
\begin{equation*}
\vi{l}:=\si{c_{1}}\si{c_{2}} \cdots \si{c_{u_l}};
\end{equation*}
if $[l+1,\,m] \cap \min S(xw_{0},\,\bi)=\emptyset$, 
then we set $\vi{l}:=e \in W$. 
%
%
\begin{lem} \label{lem:prod}
With the notation above, 
we have $\yi{l}=\wi{l}\vi{l}w_{0}^{-1}$
for every $0 \le l \le m$. 
Hence, for every $0 \le l \le m$, we have
%
%
\begin{equation} \label{eq:xi}
\xii{l}=\wi{l}\vi{l}w_{0}^{-1} \cdot \lambda.
\end{equation}
\end{lem}

\begin{proof}
We show the lemma by descending induction on $l$. 
Assume first that $l=m$. Then we have $\yi{m}=e$ and 
$\vi{m}=e$ by definition. Also, from the definition, 
we have $\wi{m}=w_{0}$. Hence we get 
$\wi{m}\vi{m}w_{0}^{-1}=w_{0}ew_{0}^{-1}=e=\yi{m}$. 
Assume now that $l < m$. 
If $l+1$ appears in the sequence $\min S(xw_{0},\,\bi)$, 
then we have 
\begin{align*}
& [l+1,\,m] \setminus \min S(xw_{0},\,\bi) = 
  [l+2,\,m] \setminus \min S(xw_{0},\,\bi), \\
& [l+1,\,m] \cap \min S(xw_{0},\,\bi) =
  \Bigl(l+1,\,[l+2,\,m] \cap \min S(xw_{0},\,\bi)\Bigr),
\end{align*}
and hence $\yi{l} = \yi{l+1}$ and $\vi{l}=\si{l+1}\vi{l+1}$. 
Therefore, we conclude that 
\begin{align*}
\yi{l} & = \yi{l+1} = \wi{l+1}\vi{l+1}w_{0}^{-1} 
  \quad \text{by the induction hypothesis} \\
& = \wi{l}\si{l+1}\vi{l+1}w_{0}^{-1} = \wi{l}\si{l+1}\si{l+1}\vi{l}w_{0}^{-1} 
  = \wi{l}\vi{l}w_{0}^{-1}.
\end{align*}
If $l+1$ does not appear in the sequence 
$\min S(xw_{0},\,\bi)$, then we have 
\begin{align*}
& [l+1,\,m] \setminus \min S(xw_{0},\,\bi) = 
  \Bigl(l+1,\,[l+2,\,m] \setminus \min S(xw_{0},\,\bi)\Bigr), \\
& [l+1,\,m] \cap \min S(xw_{0},\,\bi) =
  [l+2,\,m] \cap \min S(xw_{0},\,\bi),
\end{align*}
and hence $\yi{l} = s_{\bti{l+1}}\yi{l+1}$ and $\vi{l}=\vi{l+1}$.
By the induction hypothesis, we get 
\begin{equation*}
\yi{l} = s_{\bti{l+1}} \yi{l+1} = s_{\bti{l+1}} \wi{l+1}\vi{l+1}w_{0}^{-1}.
\end{equation*}
Since
$s_{\bti{l+1}}\wi{l+1}
 = \bigl(\wi{l}\si{l+1}(\wi{l})^{-1}\bigr)\wi{l+1} 
 = \wi{l}$ 
(note that $\si{l+1}(\wi{l})^{-1}\wi{l+1}=
(\wi{l}\si{l+1})^{-1}\wi{l+1}=(\wi{l+1})^{-1}\wi{l+1}=e$), 
it follows that 
\begin{equation*}
\yi{l}= s_{\bti{l+1}} \wi{l+1}\vi{l+1}w_{0}^{-1}=\wi{l}\vi{l+1}w_{0}^{-1}=
\wi{l}\vi{l}w_{0}^{-1}.
\end{equation*}
Thus we have shown that $\yi{l}=\wi{l}\vi{l}w_{0}^{-1}$ 
for all $0 \le l \le m$. This proves the lemma. 
\end{proof}

\begin{proof}[Proof of Lemma~\ref{lem:length}]
Let $\bi=(i_{1},\,i_{2},\,\dots,\,i_{m})$ 
be an arbitrary element of $R(w_{0})$, and 
let $1 \le l \le m$. 
First we observe that $\wi{l-1} \cdot h_{i_{l}}=
(\wi{l-1} \cdot \alpha_{i_{l}})^{\vee}=(\bti{l})^{\vee}$, 
where $\beta^{\vee}$ denotes the dual root of 
a positive root $\beta$, and hence that 
\begin{equation*}
s_{\bti{l}} \cdot \xi^{\bi}_{l}=
\xi^{\bi}_{l}-\pair{\bti{l}}{\xi^{\bi}_{l}}(\bti{l})^{\vee}=
\xi^{\bi}_{l}-\pair{\bti{l}}{\xi^{\bi}_{l}}(\wi{l-1} \cdot h_{i_{l}}). 
\end{equation*}
Therefore, we see from the definition \eqref{eq:xii} that 
%
%
\begin{equation} \label{eq:xii-dif}
\xi^{\bi}_{l}-\xi^{\bi}_{l-1}=
\begin{cases}
 0 & \text{if $l$ appears in $\min S(xw_{0},\,\bi)$}, \\[1.5mm]
 \pair{\bti{l}}{\xi^{\bi}_{l}}(\wi{l-1} \cdot h_{i_{l}})
   & \text{otherwise}.
\end{cases}
\end{equation}
So, it suffices to show that 
$\pair{\bti{l}}{\xi^{\bi}_{l}} \ge 0$ 
under the assumption that $l$ does not appear in 
the sequence $\min S(xw_{0},\,\bi)$. 
By \eqref{eq:xi}, we have
\begin{equation*}
\xi^{\bi}_{l} = \wi{l}\vi{l}w_{0}^{-1} \cdot \lambda = 
  \wi{l}\,(\si{c_1}\si{c_2} \cdots \si{c_{u_l}})\,
  w_{0}^{-1} \cdot \lambda,
\end{equation*}
where, as above, 
$[l+1,\,m] \cap \min S(xw_{0},\,\bi) 
  = (c_{1},\,c_{2},\,\dots,\,c_{u_{l}})$. 
Hence we calculate:
\begin{align*}
\pair{\bti{l}}{\xi^{\bi}_{l}} 
& = 
\Bpair{ \wi{l-1} \cdot \alpha_{i_{l}} }
 {\wi{l}\,(\si{c_1}\si{c_2} \cdots \si{c_{u_l}})\,
   w_{0}^{-1} \cdot \lambda
 } \\[1.5mm]
& = 
\Bpair{ (\wi{l})^{-1}\wi{l-1} \cdot \alpha_{i_{l}} }
 {(\si{c_1}\si{c_2} \cdots \si{c_{u_l}})\,
   w_{0}^{-1} \cdot \lambda
 } \\[1.5mm]
& = 
\Bpair {\si{l} \cdot \alpha_{i_{l}} }
 {(\si{c_1}\si{c_2} \cdots \si{c_{u_l}})\,
   w_{0}^{-1} \cdot \lambda
  } \\[1.5mm]
& =
-\Bpair{ \alpha_{i_{l}} }
  {(\si{c_1}\si{c_2} \cdots \si{c_{u_l}})\,
   w_{0}^{-1} \cdot \lambda
  } \\[1.5mm]
& = 
-\Bpair{ w_{0}(\si{c_{u_l}} \cdots \si{c_2}\si{c_1}) \cdot \alpha_{i_{l}} }
  {\lambda}.
\end{align*}
Since $\lambda \in X_{\ast}(T) \subset \Fh_{\BR}$ 
is a dominant coweight, it suffices to show that 
$(\si{c_{u_l}} \cdots \si{c_2}\si{c_1}) \cdot \alpha_{i_{l}}$
is a positive root. Suppose, contrary to our claim, 
that it is a negative root. 
Then, by the exchange condition, 
$\si{c_1}\si{c_2} \cdots \si{c_{u_l}}$ has 
a reduced expression of the form: 
\begin{equation*}
\si{l} \si{c_1} \cdots \si{c_{r-1}}\si{c_{r+1}} \cdots \si{c_{u_l}}
\end{equation*}
for some $1 \le r \le u_{l}$.
We now write the sequence 
$[1,\,l] \cap \min S(xw_{0},\,\bi)$ as: 
\begin{equation*}
[1,\,l] \cap \min S(xw_{0},\,\bi)
  = \bigl\{d_{1},\,d_{2},\,\dots,\,d_{p-u_l}\bigr\};
\end{equation*}
note that $d_{p-u_l} < l$ since $l$ does not 
appear in $\min S(xw_{0},\,\bi)$. 
Then, since $\ell(xw_{0})=p=(p-u_{l})+u_{l}$, 
\begin{align*}
xw_{0} & = 
 \si{d_1}\si{d_2} \cdots \si{d_{p-u_l}}
 \si{c_1}\si{c_2} \cdots  \si{c_{u_l}} \\
 & =
 \si{d_1}\si{d_2} \cdots \si{d_{p-u_l}}
 \si{l} \si{c_1} \cdots \si{c_{r-1}}\si{c_{r+1}} \cdots \si{c_{u_l}}
\end{align*}
is a reduced expression of $xw_{0}$, and hence 
the sequence 
\begin{equation*}
(d_1,\,d_2,\,\dots,\,d_{p-u_l},\,
 l,\,
 c_1,\,\dots,\,c_{r-1},\,c_{r+1},\,\dots,\,c_{u_l})
\end{equation*}
is an element of $S(xw_{0},\,\bi)$. 
However, since $l < l+1 \le c_{1}$, 
this sequence is strictly less than
$\min S(xw_{0},\,\bi)$ with respect to 
the lexicographic ordering $\succeq$,
which is a contradiction. 
Thus we have proved Lemma~\ref{lem:length}. 
\end{proof}

%
\subsection{Proof of Proposition~\ref{prop:1sk}.}
\label{subsec:1sk}

Keep the notation and assumptions of \S\ref{subsec:extremal}.
Assume that 
$\bi=(i_{1},\,i_{2},\,\dots,\,i_{m}) \in R(w_{0})$ and 
$\bj=(j_{1},\,j_{2},\,\dots,\,j_{m}) \in R(w_{0})$ 
are related by a $2$-move or a $3$-move. 
We will study the relation between the two elements
$\min S(xw_{0},\,\bi)$ and $\min S(xw_{0},\,\bj)$.

Assume first that 
$\bi$ and $\bj$ are related by a $2$-move 
as in \eqref{eq:2move}, i.e., 
\begin{equation*}
\begin{array}{l}
\bi=(i_{1},\,\dots,\,i_{k},\,i,\,j,\,i_{k+3},\,\dots,\,i_{m}), \\[1.5mm]
\bj=(i_{1},\,\dots,\,i_{k},\,j,\,i,\,i_{k+3},\,\dots,\,i_{m})
\end{array}
\end{equation*}
for some indices $i,\,j \in I$ with $a_{ij}=a_{ji}=0$ 
and an integer $0 \le k \le m-2$. 
We define a map 
$\sigma_{\bi,\,\bj}:S(xw_{0},\,\bi) \rightarrow S(xw_{0},\,\bj)$ 
(resp., $\sigma_{\bj,\,\bi}:S(xw_{0},\,\bj) \rightarrow S(xw_{0},\,\bi)$) 
as follows. 
Let $\ba=(a_{1},\,a_{2},\,\dots,\,a_{p}) \in S(xw_{0},\,\bi)$ 
(resp., $\ba=(a_{1},\,a_{2},\,\dots,\,a_{p}) \in S(xw_{0},\,\bj)$). 
We have 
\begin{align*}
& [1,\,k] \cap \ba
  =(a_{1},\,\dots,\,a_{u_0}), \\
& [k+1,\,k+2] \cap \ba
  =(a_{u_0+1},\,\dots,\,a_{u_1}), \\
& [k+3,\,m] \cap \ba
  =(a_{u_1+1},\,\dots,\,a_{p})
\end{align*}
for $0 \le u_{0} \le u_{1} \le p$. 
Then we set $\ba':=[k+1,\,k+2] \cap \ba=(a_{u_{0}+1},\,\dots,\,a_{u_{1}})$, and 
define a strictly increasing sequence 
$\bb=(b_{1},\,b_{2},\,\dots,\,b_{p})$ by: 
\begin{equation*}
b_{q}=a_{q} \quad \text{for $1 \le q \le u_{0}$ and $u_{1}+1 \le q \le p$}, 
\end{equation*}
\begin{equation*}
(b_{u_{0}+1},\,\dots,\,b_{u_{1}})=
\begin{cases}
\emptyset & 
  \text{if $\ba'=\emptyset$}, \\[1.5mm]
(k+1) & 
  \text{if $\ba'=(k+2)$}, \\[1.5mm]
(k+2) & 
  \text{if $\ba'=(k+1)$}, \\[1.5mm]
(k+1,\,k+2) & 
  \text{if $\ba'=(k+1,\,k+2)$}.
\end{cases}
\end{equation*}
It is obvious that 
$\bb \in S(xw_{0},\,\bj)$ (resp., $\bb \in S(xw_{0},\,\bi)$). 
We now set $\sigma_{\bi,\,\bj}(\ba):=\bb$ 
(resp., $\sigma_{\bj,\,\bi}(\ba):=\bb$).
%
%
\begin{rem} \label{rem:2-min}
Assume that $\bi$ and $\bj$ are related by a $2$-move as above. 
Obviously, 
$\sigma_{\bj,\,\bi} \circ \sigma_{\bi,\,\bj} = \id_{S(xw_0,\,\bi)}$ and 
$\sigma_{\bi,\,\bj} \circ \sigma_{\bj,\,\bi} = \id_{S(xw_0,\,\bj)}$. 
\end{rem}

Assume next that $\bi$ and $\bj$ are related 
by a $3$-move as in \eqref{eq:3move}, i.e., 
\begin{equation*}
\begin{array}{l}
\bi=(i_{1},\,\dots,\,i_{k},\,i,\,j,\,i,\,i_{k+4},\,\dots,\,i_{m}), \\[1.5mm]
\bj=(i_{1},\,\dots,\,i_{k},\,j,\,i,\,j,\,i_{k+4},\,\dots,\,i_{m})
\end{array}
\end{equation*}
for some indices $i,\,j \in I$ with $a_{ij}=a_{ji}=-1$ 
and an integer $0 \le k \le m-3$. 
We define a map 
$\sigma_{\bi,\,\bj}:S(xw_{0},\,\bi) \rightarrow S(xw_{0},\,\bj)$ 
(resp., $\sigma_{\bj,\,\bi}:S(xw_{0},\,\bj) \rightarrow S(xw_{0},\,\bi)$) 
as follows. 
Let $\ba=(a_{1},\,a_{2},\,\dots,\,a_{p}) \in S(xw_{0},\,\bi)$ 
(resp., $\ba=(a_{1},\,a_{2},\,\dots,\,a_{p}) \in S(xw_{0},\,\bj)$). 
We have 
\begin{align*}
& [1,\,k] \cap \ba
  =(a_{1},\,\dots,\,a_{u_0}), \\
& [k+1,\,k+3] \cap \ba
  =(a_{u_0+1},\,\dots,\,a_{u_1}), \\
& [k+4,\,m] \cap \ba
  =(a_{u_1+1},\,\dots,\,a_{p})
\end{align*}
for $0 \le u_{0} \le u_{1} \le p$. 
Then we set $\ba':=[k+1,\,k+3] \cap \ba=(a_{u_{0}+1},\,\dots,\,a_{u_{1}})$, and 
define a strictly increasing sequence 
$\bb=(b_{1},\,b_{2},\,\dots,\,b_{p})$ by: 
\begin{equation*}
b_{q}=a_{q} \quad \text{for $1 \le q \le u_{0}$ and $u_{1}+1 \le q \le p$}, 
\end{equation*}
\begin{equation*}
(b_{u_{0}+1},\,\dots,\,b_{u_{1}}) = 
\begin{cases}
\emptyset & 
  \text{if $\ba'=\emptyset$}, \\[1.5mm]
(k+2) & 
  \text{if $\ba'=(k+1)$ or $(k+3)$}, \\[1.5mm]
(k+1) & 
  \text{if $\ba'=(k+2)$}, \\[1.5mm]
(k+2,\,k+3) & 
  \text{if $\ba'=(k+1,\,k+2)$}, \\[1.5mm]
(k+1,\,k+2) & 
  \text{if $\ba'=(k+2,\,k+3)$}, \\[1.5mm]
(k+1,\,k+2,\,k+3) & 
  \text{if $\ba'=(k+1,\,k+2,\,k+3)$}.
\end{cases}
\end{equation*}
It is easily seen that 
$\bb \in S(xw_{0},\,\bj)$ (resp., $\bb \in S(xw_{0},\,\bi)$). 
We now set $\sigma_{\bi,\,\bj}(\ba):=\bb$ 
(resp., $\sigma_{\bj,\,\bi}(\ba):=\bb$).
%
%
\begin{rem} \label{rem:3-min}
Assume that $\bi$ and $\bj$ are related by a $3$-move as above, and 
take $\ba=(a_{1},\,a_{2},\,\dots,\,a_{p})$ to be the minimum element 
$\min S(xw_{0},\,\bi)$. It follows from the minimality of 
$\ba=(a_{1},\,a_{2},\,\dots,\,a_{p}) \in S(xw_{0},\,\bi)$ 
with respect to the lexicographic ordering that 
if $u_{1}=u_{0}+1$, then $\ba'=(k+1)$ or $(k+2)$. 
Using this, we see easily from the definitions of the maps 
$\sigma_{\bi,\,\bj}$ and $\sigma_{\bj,\,\bi}$ that 
\begin{equation*}
\sigma_{\bj,\,\bi}\bigl(\sigma_{\bi,\,\bj}(\min S(xw_0,\,\bi))\bigr)=\min S(xw_0,\,\bi).
\end{equation*}
Similarly, we obtain 
\begin{equation*}
\sigma_{\bi,\,\bj}\bigl(\sigma_{\bj,\,\bi}(\min S(xw_0,\,\bj))\bigr)=\min S(xw_0,\,\bj).
\end{equation*}
\end{rem}

%
\begin{prop} \label{prop:bmove}
Assume that $\bi=(i_{1},\,i_{2},\,\dots,\,i_{m}),\,
\bj=(j_{1},\,j_{2},\,\dots,\,j_{m}) \in R(w_{0})$ 
are related by a $2$-move or a $3$-move as above.
Then, we have 
\begin{equation*}
\sigma_{\bi,\,\bj}(\min S(xw_0,\,\bi))=\min S(xw_0,\,\bj) 
\quad \text{and} \quad
\sigma_{\bj,\,\bi}(\min S(xw_0,\,\bj))=\min S(xw_0,\,\bi).
\end{equation*}
\end{prop}

\begin{proof}
We prove that $\sigma_{\bi,\,\bj}(\min S(xw_0,\,\bi))=\min S(xw_0,\,\bj)$ 
in the case that $\bi$ and $\bj$ are related by a $3$-move as in \eqref{eq:3move}; 
the proof of the equation 
$\sigma_{\bj,\,\bi}(\min S(xw_0,\,\bj))=\min S(xw_0,\,\bj)$ is similar, 
and the proof for the case of $2$-move is simpler. 
For simplicity of notation, we set 
\begin{align*}
& \ba =(a_{1},\,a_{2},\,\dots,\,a_{p}):=\min S(xw_0,\,\bi), \\
& \bb =(b_{1},\,b_{2},\,\dots,\,b_{p}):=
  \sigma_{\bi,\,\bj}(\min S(xw_0,\,\bi))=\sigma_{\bi,\,\bj}(\ba), \\ 
& \bc =(c_{1},\,c_{2},\,\dots,\,c_{p}):=\min S(xw_0,\,\bj), \\
& \bd =(d_{1},\,d_{2},\,\dots,\,d_{p}):=
  \sigma_{\bj,\,\bi}(\min S(xw_0,\,\bj))=\sigma_{\bj,\,\bi}(\bc).
\end{align*}
We have 
%
%
\begin{equation} \label{eq:ba}
\begin{cases}
[1,\,k] \cap \ba
  =(a_{1},\,\dots,\,a_{u_0}), \\[1.5mm]
[k+1,\,k+3] \cap \ba
  =(a_{u_0+1},\,\dots,\,a_{u_1}), \\[1.5mm]
[k+4,\,m] \cap \ba
  =(a_{u_1+1},\,\dots,\,a_{p})
\end{cases}
\end{equation}
for $0 \le u_{0} \le u_{1} \le p$. 
From the definition of 
the map $\sigma_{\bi,\,\bj}$, it is obvious that
%
%
\begin{equation} \label{eq:bq}
\begin{array}{l}
b_{q}=a_{q} \quad
 \text{for all $1 \le q \le u_{0}$ and $u_{1}+1 \le q \le p$}, \\[3mm]
[k+1,\,k+3] \cap \bb =
 (b_{u_0+1},\,\dots,\,b_{u_1}).
\end{array}
\end{equation}
Also, we have 
%
%
\begin{equation} \label{eq:bc}
\begin{cases}
[1,\,k] \cap \bc
  =(c_{1},\,\dots,\,c_{t_0}), \\[1.5mm]
[k+1,\,k+3] \cap \bc
  =(c_{t_0+1},\,\dots,\,c_{t_1}), \\[1.5mm]
[k+4,\,m] \cap \bc
  =(c_{t_1+1},\,\dots,\,c_{p})
\end{cases}
\end{equation}
for $0 \le t_{0} \le t_{1} \le p$. 
From the definition of 
the map $\sigma_{\bj,\,\bi}$, 
it is obvious that
%
%
\begin{equation} \label{eq:dq}
\begin{array}{l}
d_{q}=c_{q} \quad
 \text{for all $1 \le q \le t_{0}$ and $t_{1}+1 \le q \le p$}, \\[3mm]
[k+1,\,k+3] \cap \bd =
 (d_{t_0+1},\,\dots,\,d_{t_1}).
\end{array}
\end{equation}
We will prove that $\bb=\bc$.

Since $\bb=\sigma_{\bi,\,\bj}(\ba) \in S(xw_0,\,\bj)$, 
it follows that $\bb \succeq \bc=\min S(xw_0,\,\bj)$.
Suppose, contrary to our claim, that $\bb \succneqq \bc$. 
Then, there exists an integer $1 \le q_0 \le p$ 
such that $b_{q}=c_{q}$ for all $1 \le q < q_0$ and 
$b_{q_0} > c_{q_0}$. If $q_{0} \le \min \{u_{0},\,t_{0}\}$, 
then it would follow from \eqref{eq:bq} and \eqref{eq:dq} that 
$a_{q}=d_{q}$ for all $1 \le q < q_0$ and $a_{q_0} > d_{q_0}$. 
This means that $\ba \succneqq \bd$, contradicting the minimality of 
$\ba=\min S(xw_{0},\,\bi)$. 
Therefore, we obtain $q_{0} > \min \{u_{0},\,t_{0}\}$, 
and hence $b_{q}=c_{q}$ for all 
$1 \le q \le \min \{u_{0},\,t_{0}\}$. 
We now show that $u_{0}=t_{0}$. Indeed, 
if $u_{0} > t_{0}$, then we would have 
$b_{t_{0}+1} \le b_{u_0} \le k$ 
by \eqref{eq:ba} and \eqref{eq:bq}.
Since $c_{t_{0}+1} \ge k+1$ by \eqref{eq:bc}, 
it follows that 
$b_{t_{0}+1} < c_{t_{0}+1}$. 
However, since 
$b_{q}=c_{q}$ for all 
$1 \le q \le \min \{u_{0},\,t_{0}\}=t_{0} < q_{0}$
as shown above, this implies that 
$\bb \precneqq \bc=\min S(xw_0,\,\bj)$, a contradiction. 
Similarly, if $u_{0} < t_{0}$, then we would have 
$d_{u_{0}+1} \le d_{t_{0}} \le k$ by \eqref{eq:bc} and \eqref{eq:dq}. 
Since $a_{u_{0}+1} \ge k+1$ by \eqref{eq:ba}, 
it follows that $d_{u_{0}+1} < a_{u_{0}+1}$. 
Also, since $b_{q}=c_{q}$ for all 
$1 \le q \le \min \{u_{0},\,t_{0}\}=u_{0}$, 
it follows from \eqref{eq:bq} and \eqref{eq:dq} that 
$d_{q}=a_{q}$ for all $1 \le q \le u_{0}$. 
Combining these, we obtain 
$\bd \precneqq \ba=\min S(xw_0,\,\bi)$, a contradiction. 
Hence we conclude that $u_{0}=t_{0}$ and 
$b_{q}=c_{q}$ for all $1 \le q \le u_{0}=t_{0}$; 
in addition, we see from \eqref{eq:bq} and \eqref{eq:dq} that 
$a_{q}=d_{q}$ for all $1 \le q \le u_{0}=t_{0}$. 
Thus, we have shown that 
\begin{align*}
& [1,\,k] \cap \bb=(c_{1},\,\dots,\,c_{t_{0}})=[1,\,k] \cap \bc, \\
& [k+1,\,k+3] \cap \bb=(b_{t_0+1},\,\dots,\,b_{u_1}).
\end{align*}

Next, let us show that 
$[k+1,\,k+3] \cap \bb=[k+1,\,k+3] \cap \bc$. 
Here we give a proof in the case that 
$[k+1,\,k+3] \cap \bc=(k+1,\,k+2)$; the proofs for 
the other cases are similar. 
Since $[1,\,k] \cap \bb=[1,\,k] \cap \bc$ as shown above 
and $\bb \succneqq \bc$ by assumption, the sequence 
$[k+1,\,k+3] \cap \bb$ is not equal to 
$(k+1,\,k+2,\,k+3)$.
Also, since $\bb=\sigma_{\bi,\,\bj}(\ba)$, 
it follows from the definition of the map $\sigma_{\bi,\,\bj}$ 
that the sequence $[k+1,\,k+3] \cap \bb$ cannot be equal to 
$(k+3)$ nor to $(k+1,\,k+3)$. 
Furthermore, 
if $[k+1,\,k+3] \cap \bb=\emptyset$, then 
we would have $t_{0}=u_{1}$ and $a_{t_{0}+1}=a_{u_{1}+1} \ge k+4$. 
Since $a_{q}=d_{q}$ for all $1 \le q \le t_{0}$ as seen above, and 
since $d_{t_{0}+1}=k+2 \ (< a_{t_{0}+1})$ by the definition of 
the map $\sigma_{\bj,\,\bi}$, 
along with the assumption that $[k+1,\,k+3] \cap \bc=(k+1,\,k+2)$, 
it follows that $\ba \succneqq \bd=\sigma_{\bj,\,\bi}(\bc)$, 
which contradicts the minimality of $\ba=\min S(xw_{0},\,\bi)$. 
Similarly, if $[k+1,\,k+3] \cap \bb=(k+1)$, then 
we would have $t_{0}+1=u_{1}$ and 
$a_{t_{0}+2}=a_{u_{1}+1} \ge k+4$; 
also, since $\sigma_{\bj,\,\bi}(\bb)=\ba$ 
by Remark~\ref{rem:3-min}, it follows from 
the definition of the map $\sigma_{\bj,\,\bi}$ 
that $a_{t_{0}+1}=a_{u_{1}}=k+2$. 
Since $a_{q}=d_{q}$ for all $1 \le q \le t_{0}$ as seen above, and 
since $d_{t_{0}+1}=k+2 \ (=a_{t_{0}+1})$ and 
$d_{t_{0}+2}=k+3 \ (< a_{t_{0}+2})$ 
by the definition of the map $\sigma_{\bj,\,\bi}$ 
along with the assumption that $[k+1,\,k+3] \cap \bc=(k+1,\,k+2)$, 
it follows that $\ba \succneqq \bd=\sigma_{\bj,\,\bi}(\bc)$, contradicting 
the minimality of $\ba$. 
Now, suppose, contrary to our claim, that 
$[k+1,\,k+3] \cap \bb=(k+2)$ or $(k+2,\,k+3)$; 
note that $b_{t_{0}+1}=k+2$. 
We set $w':=
 \sj{b_{1}}\sj{b_{2}} \cdots \sj{b_{t_0}}=
 \sj{c_{1}}\sj{c_{2}} \cdots \sj{c_{t_0}}$, 
and $w:=(w')^{-1}xw_{0}$. Then, the $w$ has reduced expressions 
of the form
\begin{equation*}
w  =\sj{b_{t_0+1}}\sj{b_{t_0+2}} \cdots \sj{b_{p}}
   =\sj{c_{t_0+1}}\sj{c_{t_0+2}} \cdots \sj{c_{p}}.
\end{equation*}
If we set $z:=\sj{c_{t_1+1}}\sj{c_{t_1+2}} \cdots \sj{c_{p}}$, 
then 
\begin{equation*}
w=(\sj{c_{t_0+1}} \dots \sj{c_{t_1}})z
 =\sj{k+1}\sj{k+2}z
 =s_{j}s_{i}z
\end{equation*}
since $(c_{t_0+1},\,\dots,\,c_{t_1})=(k+1,\,k+2)$ by assumption 
and $j_{k+1}=j$, $j_{k+2}=i$. 
Also, since $b_{t_{0}+1}=k+2$ by assumption and $j_{k+2}=i$, 
we deduce by \cite[Chap.\,5, \S3, Proposition~2\,(i)]{MP} that 
$w^{-1} \cdot \alpha_{j_{k+2}}=w^{-1} \cdot \alpha_{i}$ is 
a negative root, and so 
$z^{-1} \cdot \alpha_{j}=z^{-1}s_{i}s_{j} \cdot \alpha_{i}=
 (s_{j}s_{i}z)^{-1} \cdot \alpha_{i}=
 w^{-1} \cdot \alpha_{i}$ is a negative root.
Hence it follows from the exchange condition 
that $z$ has a reduced expression of the form: 
\begin{equation*}
z=s_{j} 
  \sj{c_{t_1+1}} \cdots \sj{c_{r-1}} \sj{c_{r+1}} \cdots \sj{c_{p}}
\end{equation*}
for some $t_{1}+1 \le r \le p$. Therefore, we obtain
\begin{align*}
xw_{0} & = w'w = w's_{j}s_{i}z \\
 & = 
 \sj{c_{1}}\sj{c_{2}} \cdots \sj{c_{t_0}}
 s_{j}s_{i}s_{j}
 \sj{c_{t_1+1}} \cdots \sj{c_{r-1}} \sj{c_{r+1}} \cdots \sj{c_{p}} \\
 & = 
 \sj{c_{1}}\sj{c_{2}} \cdots \sj{c_{t_0}}
 s_{j_{k+1}}s_{j_{k+2}}s_{j_{k+3}}
 \sj{c_{t_1+1}} \cdots \sj{c_{r-1}} \sj{c_{r+1}} \cdots \sj{c_{p}}, 
\end{align*}
which is a reduced expression of $xw_{0}$, and hence 
the sequence 
\begin{equation*}
(c_{1},\,c_{2},\,\dots,\,c_{t_0},\,
 k+1,\,k+2,\,k+3,\,
 c_{t_1+1},\,\dots,\,c_{r-1},\,c_{r+1},\,\dots,\,c_{p})
\end{equation*}
is an element of $S(xw_{0},\,\bj)$. 
However, this sequence  is strictly less than 
$\bc=\min S(xw_{0},\,\bj)$ with respect to 
the lexicographic ordering, which is a contradiction. 
Hence we conclude that 
$[k+1,\,k+3] \cap \bb=(k+1,\,k+2)=[k+1,\,k+3] \cap \bc$. 
Thus, we have shown that $u_{1}=t_{1}$ and 
$b_{q}=c_{q}$ for $t_{0}+1 \le q \le u_{1}=t_{1}$, and so
%
%
\begin{equation} \label{eq:bb2}
\begin{cases}
[1,\,k] \cap \bb =
  (c_{1},\,\dots,\,c_{t_{0}}) =
  [1,\,k] \cap \bc, \\[1.5mm]
[k+1,\,k+3] \cap \bb =
  (c_{t_{0}+1},\,\dots,\,c_{t_{1}}) =
  [k+1,\,k+3] \cap \bc, \\[1.5mm]
[k+4,\,m] \cap \bb =
  (b_{t_{1}+1},\,\dots,\,b_{p}).
\end{cases}
\end{equation}

The equations \eqref{eq:bb2} 
imply that $q_{0} > t_{1}=u_{1}$. 
In this case, since $\bc=\min S(xw_{0},\,\bj) \precneqq \bb$, 
it follows from the definition of the map $\sigma_{\bj,\,\bi}$ that 
$\sigma_{\bj,\,\bi}(\bc) \precneqq \sigma_{\bj,\,\bi}(\bb)=\ba$, 
contrary to the minimality of $\ba$. 
Thus we have proved that $\bb=\bc$, 
thereby completing the proof of 
Proposition~\ref{prop:bmove}. 
\end{proof}

\begin{proof}[Proof of Proposition~\ref{prop:1sk}.]
By Remark~\ref{rem:braid}, we may assume that 
$\bi,\,\bj \in R(w_{0})$ are related by a $2$-move or 
a $3$-move. Furthermore, in view of \cite[Theorem~7.1]{Kam1}, 
it suffices to show that
$\mu_{\wj{l}}^{x,\,\bi}=\xij{l}$ for every $0 \le l \le m$.
We give a proof in the case that 
$\bi,\,\bj \in R(w_{0})$ are related 
by a $3$-move as in \eqref{eq:3move}, i.e., 
\begin{equation*}
\begin{array}{l}
\bi=(i_{1},\,\dots,\,i_{k},\,i,\,j,\,i,\,i_{k+4},\,\dots,\,i_{m}), \\[1.5mm]
\bj=(i_{1},\,\dots,\,i_{k},\,j,\,i,\,j,\,i_{k+4},\,\dots,\,i_{m})
\end{array}
\end{equation*}
for some indices $i,\,j \in I$ with $a_{ij}=a_{ji}=-1$ 
and an integer $0 \le k \le m-3$; 
the proof for the case of $2$-move is 
similar (or, even simpler). 

Assume first that $l \ne k+1,\,k+2$. 
Then, it is obvious that $\wi{l}=\wj{l}$. 
Also, because $\min S(xw_{0},\,\bj)=
\sigma_{\bi,\,\bj}(\min S(xw_{0},\,\bi))$ 
by Proposition~\ref{prop:bmove}, 
we deduce from the definition (see \S\ref{subsec:length}) 
that $\vi{l}=\vj{l}$. Therefore, 
it follows immediately from Lemma~\ref{lem:prod} that 
\begin{equation*}
\xii{l}=\wi{l}\vi{l}w_{0}^{-1} \cdot \lambda
  =\wj{l}\vj{l}w_{0}^{-1} \cdot \lambda=\xij{l}. 
\end{equation*}
Hence we obtain $\mu_{\wj{l}}^{x,\,\bi}=
\mu_{\wi{l}}^{x,\,\bi}=\xii{l}=\xij{l}$, as desired. 
Thus, it remains to show that 
%
%
\begin{equation}\label{eq:remain}
\mu_{\wj{k+1}}^{x,\,\bi}=\xij{k+1} \quad \text{and} \quad
\mu_{\wj{k+2}}^{x,\,\bi}=\xij{k+2}.
\end{equation} 
Recall that the integers 
$\Ni{l}=\Ni{l}(\mu_{\bullet}^{x,\,\bi}) \in \BZ_{\ge 0}$ and 
$\Nj{l}=\Nj{l}(\mu_{\bullet}^{x,\,\bi}) \in \BZ_{\ge 0}$ for 
$l=k+1,\,k+2,\,k+3$ are defined via the length formula:
\begin{equation*}
\mu_{\wi{l}}^{x,\,\bi}-\mu_{\wi{l-1}}^{x,\,\bi}=
  \Ni{l} (\wi{l-1} \cdot h_{i_{l}})
\quad \text{and} \quad
\mu_{\wj{l}}^{x,\,\bi}-\mu_{\wj{l-1}}^{x,\,\bi}=
  \Nj{l} (\wj{l-1} \cdot h_{j_{l}}),
\end{equation*}
respectively. Also, we know from 
(the proof of) Lemma~\ref{lem:length} 
that for $l=k+1,\,k+2,\,k+3$, 
there exists an integer $\NNj{l} \in \BZ_{\ge 0}$ such that 
\begin{equation*}
\xij{l}-\xij{l-1}=\NNj{l} (\wj{l-1} \cdot h_{j_{l}});
\end{equation*}
recall that the integer $\NNj{l} \in \BZ_{\ge 0}$ is given by:
%
%
\begin{equation} \label{eq:NNj}
\NNj{l} = 
  \begin{cases}
   0 & \text{if $l$ appears in $\min S(xw_{0},\,\bj)$}, \\[1.5mm]
   \pair{\btj{l}}{\xij{l}} & \text{otherwise}.
  \end{cases}
\end{equation}
Since $\mu_{\wj{k+3}}^{x,\,\bi}=\xij{k+3}$ as shown above, 
equations \eqref{eq:remain} will follow if we can show that 
$\Nj{k+3}=\NNj{k+3}$ and $\Nj{k+2}=\NNj{k+2}$. 
We will show that $\Nj{l}=\NNj{l}$ for $l=k+1,\,k+2,\,k+3$. 

We know that the sequence $[k+1,\,k+3] \cap \min S(xw_{0},\,\bi)$
is equal to one of the following: 
(i) $\emptyset$, (ii) $(k+1)$, (iii) $(k+2)$, (iv) $(k+1,\,k+2)$, 
(v) $(k+2,\,k+3)$, (vi) $(k+1,\,k+2,\,k+3)$. 
Also, since $\mu_{\wi{l}}^{x,\,\bi}=\xii{l}$ for all $0 \le l \le m$
and hence $\mu_{\wi{l}}^{x,\,\bi}-\mu_{\wi{l-1}}^{x,\,\bi}=
\xii{l}-\xii{l-1}$ for all $1 \le l \le m$, 
it follows from \eqref{eq:xii-dif} that 
%
%
\begin{equation} \label{eq:k1}
\Ni{k+1} = 
  \begin{cases}
   0 & \text{if $k+1$ appears in $\min S(xw_{0},\,\bi)$}, \\[1.5mm]
   \pair{\bti{k+1}}{\xii{k+1}} & \text{otherwise},
  \end{cases}
\end{equation}
%
%
\begin{equation} \label{eq:k2}
\Ni{k+2} = 
  \begin{cases}
   0 & \text{if $k+2$ appears in $\min S(xw_{0},\,\bi)$}, \\[1.5mm]
   \pair{\bti{k+2}}{\xii{k+2}} & \text{otherwise}, 
  \end{cases}
\end{equation}
%
%
\begin{equation} \label{eq:k3}
\Ni{k+3} = 
  \begin{cases}
   0 & \text{if $k+3$ appears in $\min S(xw_{0},\,\bi)$}, \\[1.5mm]
   \pair{\bti{k+3}}{\xii{k+3}} & \text{otherwise}. 
  \end{cases}
\end{equation}
For simplicity of notation, we set 
$\gamma_{i}:=\wi{k} \cdot \alpha_i$ and 
$\gamma_{j}:=\wi{k} \cdot \alpha_j$. Then we calculate: 
\begin{align*}
& \bti{k+1}=
   \wi{k} \cdot \alpha_{i_{k+1}}=
   \wi{k} \cdot \alpha_{i}=
   \gamma_{i}, \\
& \bti{k+2}=
   \wi{k+1} \cdot \alpha_{i_{k+2}}=
   \wi{k}\si{k+1} \cdot \alpha_{i_{k+2}}=
   \wi{k}s_{i} \cdot \alpha_{j}=
   \wi{k} \cdot (\alpha_{i}+\alpha_{j})=\gamma_{i}+\gamma_{j}, \\
& \bti{k+3}=
   \wi{k+2} \cdot \alpha_{i_{k+3}}=
   \wi{k}\si{k+1}\si{k+2} \cdot \alpha_{i_{k+3}}=
   \wi{k}s_{i}s_{j} \cdot \alpha_{i}=
   \wi{k} \cdot \alpha_{j}=\gamma_{j}.
\end{align*}
Also, we set
$A:=\pair{\gamma_{i}}{\xii{k+3}}$ and 
$B:=\pair{\gamma_{j}}{\xii{k+3}}$.
Then, we obtain
%
%
\begin{equation} \label{eq:lg1}
(\Ni{k+1},\,\Ni{k+2},\,\Ni{k+3}) =
\begin{cases}
(B,\,A,\,B) & \text{in case (i)}, \\[1.5mm]
(0,\,A,\,B) & \text{in case (ii)}, \\[1.5mm]
(A+B,\,0,\,B) & \text{in case (iii)}, \\[1.5mm]
(0,\,0,\,B)  & \text{in case (iv)}, \\[1.5mm]
(A,\,0,\,0) & \text{in case (v)}, \\[1.5mm]
(0,\,0,\,0)  & \text{in case (vi)}.
\end{cases}
\end{equation}
As an example, let us give a proof in case (iii).
Since $k+3$ does not appear in $\min S(xw_{0},\,\bi)$, we have 
$\xii{k+2}=s_{\bti{k+3}} \cdot \xii{k+3}$ 
by definition \eqref{eq:xii}, and 
$\Ni{k+3} = \pair{\bti{k+3}}{\xii{k+3}} = 
 \pair{\gamma_{j}}{\xii{k+3}}=B$ by \eqref{eq:k3}. 
Since $k+2$ appears in $\min S(xw_{0},\,\bi)$, we have 
$\xii{k+1}=\xii{k+2}=s_{\bti{k+3}} \cdot \xii{k+3}$ 
by definition \eqref{eq:xii}, 
and $\Ni{k+2}=0$ by \eqref{eq:k2}. 
Since $k+1$ does not appear in $\min S(xw_{0},\,\bi)$, 
we conclude by \eqref{eq:k1} that 
\begin{align*}
\Ni{k+1} & = \pair{\bti{k+1}}{\xii{k+1}} = 
 \pair{\bti{k+1}}{s_{\bti{k+3}} \cdot \xii{k+3}} = 
 \pair{\gamma_{i}}{s_{\gamma_{j}} \cdot \xii{k+3}} \\ 
 & =  \pair{\gamma_{i}+\gamma_{j}}{\xii{k+3}}=A+B. 
\end{align*}
The proofs for the other cases are similar. 
Hence it follows from \eqref{eq:3bm} that
%
%
\begin{equation} \label{eq:lg2}
(\Nj{k+1},\,\Nj{k+2},\,\Nj{k+3}) =
\begin{cases}
(A,\,B,\,A) & \text{in case (i)}, \\[1.5mm]
(A+B,\,0,\,A) & \text{in case (ii)}, \\[1.5mm]
(0,\,B,\,A) & \text{in case (iii)}, \\[1.5mm]
(B,\,0,\,0)  & \text{in case (iv)}, \\[1.5mm]
(0,\,0,\,A) & \text{in case (v)}, \\[1.5mm]
(0,\,0,\,0)  & \text{in case (vi)}.
\end{cases}
\end{equation}

Next, we compute $(\NNj{k+1},\,\NNj{k+2},\,\NNj{k+3})$. 
From Proposition~\ref{prop:bmove}, we deduce by the definition of 
the map $\sigma_{\bi,\,\bj}$ that 
\begin{align*}
[k+1,\,k+3] \cap \min S(xw_{0},\,\bj) = 
\begin{cases}
\emptyset & \text{in case (i)}, \\[1.5mm]
(k+2) & \text{in case (ii)}, \\[1.5mm]
(k+1) & \text{in case (iii)}, \\[1.5mm]
(k+2,\,k+3)  & \text{in case (iv)}, \\[1.5mm]
(k+1,\,k+2) & \text{in case (v)}, \\[1.5mm]
(k+1,\,k+2,\,k+3) & \text{in case (vi)}.
\end{cases}
\end{align*}
Also, since $\wi{k}=\wj{k}$ and 
$\xii{k+3}=\xij{k+3}$ as seen above, 
we have $\gamma_{i}=\wj{k} \cdot \alpha_i$, 
$\gamma_{j}=\wj{k} \cdot \alpha_j$, 
and also $A=\pair{\gamma_{i}}{\xij{k+3}}$, 
$B=\pair{\gamma_{j}}{\xij{k+3}}$.
Using these, by an argument similar to 
the one for \eqref{eq:lg1}, we obtain
%
%
\begin{equation} \label{eq:lg3}
(\NNj{k+1},\,\NNj{k+2},\,\NNj{k+3}) =
\begin{cases}
(A,\,B,\,A) & \text{in case (i)}, \\[1.5mm]
(A+B,\,0,\,A) & \text{in case (ii)}, \\[1.5mm]
(0,\,B,\,A) & \text{in case (iii)}, \\[1.5mm]
(B,\,0,\,0)  & \text{in case (iv)}, \\[1.5mm]
(0,\,0,\,A) & \text{in case (v)}, \\[1.5mm]
(0,\,0,\,0)  & \text{in case (vi)}.
\end{cases}
\end{equation}
As an example, let us give a proof in case (iii); 
recall that $[k+1,\,k+3] \cap \min S(xw_{0},\,\bj)=(k+1)$. 
Since $k+3$ does not appear in $\min S(xw_{0},\,\bj)$, we have 
$\xij{k+2}=s_{\btj{k+3}} \cdot \xij{k+3}$ 
by definition \eqref{eq:xii}, and 
$\NNj{k+3} = \pair{\btj{k+3}}{\xij{k+3}}$ by \eqref{eq:NNj}.
Also, we calculate: 
\begin{equation*}
\btj{k+3}=
   \wj{k+2} \cdot \alpha_{j_{k+3}}=
   \wj{k}\sj{k+1}\sj{k+2} \cdot \alpha_{j_{k+3}}=
   \wj{k}s_{j}s_{i} \cdot \alpha_{j}=
   \wj{k} \cdot \alpha_{i}=\gamma_{i}.
\end{equation*}
Hence we get $\NNj{k+3}=\pair{\gamma_{i}}{\xij{k+3}}=A$. 
Similarly, since $k+2$ does not appear in $\min S(xw_{0},\,\bj)$, 
we have $\NNj{k+2} = \pair{\btj{k+2}}{\xij{k+2}}$ and 
$\btj{k+2}=\gamma_{i}+\gamma_{j}$. Therefore, we conclude that 
\begin{align*}
\NNj{k+2} & = \pair{\btj{k+2}}{\xij{k+2}} = 
\pair{\btj{k+2}}{s_{\btj{k+3}} \cdot \xij{k+3}} \\
& = \pair{\gamma_{i}+\gamma_{j}}{s_{\gamma_{i}} \cdot \xij{k+3}} = 
\pair{\gamma_{j}}{\xij{k+3}} = B. 
\end{align*}
Since $k+1$ appears in $\min S(xw_{0},\,\bj)$, we have 
$\xij{k}=\xij{k+1}$ by definition \eqref{eq:xii}, 
and hence $\NNj{k+1}=0$. The proofs for the other cases 
are similar. 

Combining \eqref{eq:lg2} and \eqref{eq:lg3}, 
we obtain 
$(\Nj{k+1},\,\Nj{k+2},\,\Nj{k+3})=
 (\NNj{k+1},\,\NNj{k+2},\,\NNj{k+3})$, as desired. 
This completes the proof of 
Proposition~\ref{prop:1sk}. 
\end{proof}

%
\subsection{Proof of Theorem~\ref{thm:ext}.}
\label{subsec:prf-thmext}

Keep the notation and assumptions of \S\ref{subsec:extremal}.
Let $\bi \in R(w_{0})$, and write the element 
$\min S(xw_{0},\,\bi)$ as: 
\begin{equation*}
\min S(xw_{0},\,\bi)=(a_{1},\,a_{2},\,\dots,\,a_{p}).
\end{equation*}
Then, we have 
\begin{align*}
\mu_{e}^{x}
 & = \mu_{\wi{0}}^{x}=\xi^{\bi}_{0}=
     \wi{0}\vi{0}w_{0}^{-1} \cdot \lambda \quad \text{by \eqref{eq:xi}} \\
 & = \wi{0}\,(\si{a_1}\si{a_2} \cdots \si{a_{p}})\,
     w_{0}^{-1} \cdot \lambda
   = e(xw_{0})w_{0}^{-1} \cdot \lambda
   =x \cdot \lambda. 
\end{align*}
Therefore, it follows that 
$\wt(P(\mu_{\bullet}^{x}))=\mu_{e}^{x}=x \cdot \lambda$. 
Because $P_{x \cdot \lambda}$ is the unique element 
of $\mv(\lambda)$ whose weight is 
$x \cdot \lambda$, we conclude that 
$P(\mu_{\bullet}^{x})=P_{x \cdot \lambda}$. 
Thus, we have proved part (1) of Theorem~\ref{thm:ext}. 

Let us prove part (2). We first show that 
$P_{x \cdot \lambda}=P(\mu_{\bullet}^{x}) 
 \subset \Conv(W_{\le x} \cdot \lambda)$. 
By \eqref{eq:inc}, it suffices to show that 
$\xii{l} \in W_{\le x} \cdot \lambda$ 
for all $\bi \in R(w_{0})$ and $0 \le l \le m$.
We take and fix an arbitrary $\bi \in R(w_{0})$, and 
show the assertion by induction on $l$. 
If $l=0$, then the assertion is obvious since 
$\xii{0}=x \cdot \lambda$ 
as shown above. 
Assume now that $l > 0$. 
If $l$ appears in $\min S(xw_{0},\,\bi)$, then 
$\xii{l}=\xii{l-1}$ by definition \eqref{eq:xii}, and hence 
$\xii{l} \in W_{\le x} \cdot \lambda$ 
by our induction hypothesis. 
So, we may assume that 
$l$ does not appear in $\min S(xw_{0},\,\bi)$; 
note that $\xii{l-1}=s_{\bti{l}} \cdot \xii{l}$ 
by definition \eqref{eq:xii}. In this case, we know 
from the proof of Lemma~\ref{lem:length} that 
$\pair{\bti{l}}{\xii{l}} \ge 0$. 
When $\pair{\bti{l}}{\xii{l}}=0$, we have 
$\xii{l}=\xii{l-1}$ by definition (see also \eqref{eq:xii-dif}), 
and hence $\xii{l} \in W_{\le x} \cdot \lambda$ 
by our induction hypothesis. 
Now, it remains to consider the case
$\pair{\bti{l}}{\xii{l}} > 0$. 
By our induction hypothesis, there exists 
$z \in W_{\le x}$ such that 
$\xi^{\bi}_{l-1}=z \cdot \lambda$. 
Then, we calculate: 
\begin{equation*}
\pair{z^{-1} \cdot \bti{l}}{\lambda} = 
\pair{\bti{l}}{z \cdot \lambda} = 
\pair{\bti{l}}{\xi^{\bi}_{l-1}}=
\pair{\bti{l}}{s_{\bti{l}} \cdot \xii{l}}=
-\pair{\bti{l}}{\xii{l}} < 0. 
\end{equation*}
Since $\lambda \in X_{\ast}(T) \subset \Fh_{\BR}$ 
is a dominant coweight, 
this implies that $z^{-1} \cdot \bti{l}$ 
is a negative root. Therefore, we deduce from 
\cite[Chap.\,5, \S3, Proposition~2\,(i)]{MP} that 
$s_{\bti{l}}z < z \le x$, and hence that 
$\xii{l}=s_{\bti{l}} \cdot \xi^{\bi}_{l-1}=
 s_{\bti{l}}z \cdot \lambda \in 
W_{\le x} \cdot \lambda$. 
Thus, we have shown that 
$P_{x \cdot \lambda}=P(\mu_{\bullet}^{x}) 
 \subset \Conv(W_{\le x} \cdot \lambda)$ in all cases. 

Next, we show the reverse inclusion: 
$P_{x \cdot \lambda}=P(\mu_{\bullet}^{x}) 
 \supset \Conv(W_{\le x} \cdot \lambda)$. 
By \eqref{eq:inc}, it suffices to show that 
for each $z \in W$ with $z \le x$, 
there exist some $\bi \in R(w_{0})$ and 
an integer $0 \le l \le m$ 
such that $\xii{l}=z \cdot \lambda$. 
Let $z \in W$ be such that $z \le x$; 
note that $zw_{0} \ge xw_{0}$ 
(see, for example, \cite[Proposition~2.3.4\,(i)]{BB}).
Denote by $l$ the length $\ell(zw_{0})$ of $zw_{0}$, 
and take $\bi \in R(w_{0})$ such that $\wi{l}=zw_{0}$. 
We define a subset $S(xw_{0},\,\bi)_{\le l}$ of 
$S(xw_{0},\,\bi)$ by: 
\begin{equation*}
S(xw_{0},\,\bi)_{\le l}=
 \bigl\{
   (a_{1},\,a_{2},\,\dots,\,a_{p}) \in S(xw_{0},\,\bi) \mid 
   a_{p} \le l
 \bigr\} \subset S(xw_{0},\,\bi);
\end{equation*}
since $\wi{l}=zw_{0} \ge xw_{0}$, it follows from 
the ``subword condition'' (see, for example, 
\cite[Chap.\,5, \S4, Proposition~2]{MP}) 
that $S(xw_{0},\,\bi)_{\le l}$ is not empty.

\begin{claim*}
The minimum element $\min S(xw_{0},\,\bi)_{\le l}$ 
of $S(xw_{0},\,\bi)_{\le l}$ 
with respect to the lexicographic ordering $\succeq$ 
is identical to the minimum element $\min S(xw_{0},\,\bi)$ 
of $S(xw_{0},\,\bi)$. 
In particular, we have 
$[l+1,\,m] \cap \min S(xw_{0},\,\bi)=\emptyset$.
\end{claim*}

\noindent 
{\it Proof of Claim.} Let us write the elements 
$\min S(xw_{0},\,\bi)_{\le l}$ and $\min S(xw_{0},\,\bi)$ as
\begin{equation*}
\min S(xw_{0},\,\bi)_{\le l} = (a_{1},\,a_{2},\,\dots,\,a_{p})
\quad \text{and} \quad
\min S(xw_{0},\,\bi) = (b_{1},\,b_{2},\,\dots,\,b_{p}), 
\end{equation*}
respectively. Suppose, contrary to our claim, 
that $\min S(xw_{0},\,\bi)_{\le l} \succneqq \min S(xw_{0},\,\bi)$.
Then, there exists some integer $1 \le q \le p$ such that 
$a_{1}=b_{1},\,a_{2}=b_{2},\,\dots,\,a_{q-1}=b_{q-1}$ and 
$a_{q} > b_{q}$; note that if $q \ge 2$, then
$a_{q} > b_{q} > b_{q-1}=a_{q-1}$.
We set $w:=\si{a_{q}} \cdots \si{a_{p}}$. 
Since
$\si{a_{1}} \cdots \si{a_{q-1}}=
 \si{b_{1}} \cdots \si{b_{q-1}}$ by the definition of $q$, 
and since 
\begin{equation*}
\si{a_{1}} \cdots \si{a_{q-1}}\si{a_{q}} \cdots \si{a_{p}}=
 \si{b_{1}} \cdots \si{b_{q-1}}\si{b_{q}} \cdots \si{b_{p}}=xw_{0},
\end{equation*}
we get $w=\si{b_{q}} \cdots \si{b_{p}}$.
It follows from \cite[Chap.\,5, \S3, Proposition~2\,(i)]{MP} that 
$w^{-1} \cdot \alpha_{i_{b_{q}}} = 
(\si{b_{q}} \cdots \si{b_{p}})^{-1} \cdot \alpha_{i_{b_{q}}}$
is a negative root. Therefore, we see 
by the exchange condition 
that $w$ has a reduced expression of the form 
$\si{b_{q}}\si{a_{q}} \cdots \si{a_{r-1}}\si{a_{r+1}} \cdots \si{a_{p}}$
for some $q \le r \le p$, and hence that $xw_{0}$ has 
a reduced expression of the form
\begin{equation*}
xw_{0}=\si{a_{1}} \cdots \si{a_{q-1}}
\si{b_{q}}\si{a_{q}} \cdots \si{a_{r-1}}\si{a_{r+1}} \cdots \si{a_{p}}.
\end{equation*}
It follows that the sequence
\begin{equation*}
(a_{1},\,\dots,\,a_{q-1},\,
 b_{q},\,a_{q},\,\dots,\,a_{r-1},\,a_{r+1},\,\dots,\,a_{p})
\end{equation*}
is an element of $S(xw_{0},\,\bi)_{\le l}$, and is strictly less than 
$\ba=\min S(xw_{0},\,\bi)_{\le l}$ with respect to 
the lexicographic ordering, which is a contradiction. 
Thus, we have shown that $\min S(xw_{0},\,\bi)_{\le l}=\min S(xw_{0},\,\bi)$, 
as desired. This proves the claim. \bqed

\vsp

By \eqref{eq:xi}, we have 
$\xii{l}=\wi{l}\vi{l}w_{0}^{-1} \cdot \lambda$. 
Since $[l+1,\,m] \cap \min S(xw_{0},\,\bi)=\emptyset$ 
by the claim above, it follows from the definition that 
$\vi{l}=e$. Hence we conclude that
\begin{equation*}
\xii{l}=
  \wi{l}\vi{l}w_{0}^{-1} \cdot \lambda = 
  \wi{l}w_{0}^{-1} \cdot \lambda = 
  (zw_{0})w_{0}^{-1} \cdot \lambda = z \cdot \lambda, 
\end{equation*}
since $\wi{l}=zw_{0}$ for the $\bi \in R(w_{0})$. 
Thus, we have shown that 
$P_{x \cdot \lambda}=P(\mu_{\bullet}^{x}) 
 \supset \Conv(W_{\le x} \cdot \lambda)$, 
thereby completing the proof of part (2) 
of Theorem~\ref{thm:ext}.

%
\subsection{Relation between extremal MV polytopes and 
opposite Demazure crystals.}
\label{subsec:p-opdem}
By combining Theorems~\ref{thm:opdem} and \ref{thm:ext}, 
we obtain a polytopal condition for an MV polytope 
$P \in \mv(\lambda)$ to lie in the opposite Demazure crystal 
$\mv^{x}(\lambda)$ for $x \in W$.
%
%
\begin{thm} \label{thm:p-opdem}
Let $x \in W$. 
An element $P \in \mv(\lambda)$ is contained in 
$\mv^{x}(\lambda)=\Psi_{\lambda}(\CB^{x}(\lambda))$ 
if and only if the MV polytope $P$ contains 
{\rm(}as a set\,{\rm)} 
the extremal MV polytope $P_{x \cdot \lambda}=\Conv(W_{\le x} \cdot \lambda)$. 
\end{thm}

First, we prove the ``only if'' part of the theorem. 
Before giving the proof, we make the following remarks. 

%
\begin{rem} \label{rem:succ1}
We know 
(see \cite[\S4.1]{BaG} and \cite[Theorem~4.7]{Kam2}) that 
for $P \in \mv(\lambda)$ and $j \in I$ 
with $f_{j}P \ne \bzero$, 
there holds $f_{j}P \supset P$. 
Therefore, 
if an MV polytope $P \in \mv^{x}(\lambda)$ were
obtained from the extremal MV polytope 
$P_{x \cdot \lambda}$
by successive application of 
the lowering Kashiwara operators $f_{j}$, $j \in I$, 
then it would follow immediately that 
the polytope $P$ contains $P_{x \cdot \lambda}$. 
However, in general, not all 
MV polytopes in $\mv^{x}(\lambda)$ can be 
obtained from 
$P_{x \cdot \lambda}$ by successively applying 
$f_{j}$, $j \in I$, as Example~\ref{ex:lowering} below shows.
\end{rem}
%
%
\begin{ex} \label{ex:lowering}
Let $\Fg$ be the simple Lie algebra of type $A_{2}$. 
Set $\lambda:=h_{1}+h_{2} \in X_{\ast}(T) \subset \Fh_{\BR}$, 
and $x:=s_{1} \in W$. 
It follows from Lemma~\ref{lem:fmax} that 
an element $P \in \mv(\lambda)$ is contained in
$\mv^{x}(\lambda)$ if and only if 
$f_{1}^{\max}f_{2}^{\max}P=P_{w_{0} \cdot \lambda}$. 
Hence we see from the crystal graph \eqref{eq:cg2} below of 
$\mv(\lambda) \cong \CB(\lambda)$
that $f_{1}^{2}f_{2}P_{\lambda}$
is contained in $\mv^{x}(\lambda)$. 
However, we deduce from the crystal graph \eqref{eq:cg2}
that $f_{1}^{2}f_{2}P_{\lambda}$ cannot be obtained from 
$P_{x \cdot \lambda}$ by successively applying 
$f_{1}$ and $f_{2}$. 
%
%
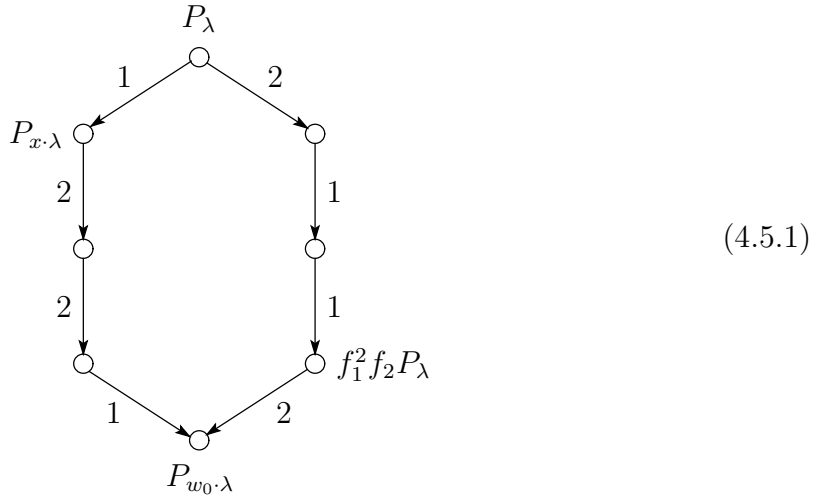
\begin{equation} \label{eq:cg2}
\hspace*{-25mm}
\begin{array}{c}
\unitlength 0.1in
\begin{picture}( 24.4000, 24.0500)(  0.0500,-25.1500)
%
\special{pn 8}%
\special{ar 1796 396 50 50  0.0000000 6.2831853}%
%
\special{pn 8}%
\special{ar 1196 796 50 50  0.0000000 6.2831853}%
%
\special{pn 8}%
\special{ar 1196 1396 50 50  0.0000000 6.2831853}%
%
\special{pn 8}%
\special{ar 1196 1996 50 50  0.0000000 6.2831853}%
%
\special{pn 8}%
\special{ar 1796 2396 50 50  0.0000000 6.2831853}%
%
\special{pn 8}%
\special{ar 2396 1996 50 50  0.0000000 6.2831853}%
%
\special{pn 8}%
\special{ar 2396 1396 50 50  0.0000000 6.2831853}%
%
\special{pn 8}%
\special{ar 2396 796 50 50  0.0000000 6.2831853}%
%
\special{pn 8}%
\special{pa 1756 416}%
\special{pa 1236 756}%
\special{fp}%
\special{sh 1}%
\special{pa 1236 756}%
\special{pa 1302 736}%
\special{pa 1280 726}%
\special{pa 1280 702}%
\special{pa 1236 756}%
\special{fp}%
%
\special{pn 8}%
\special{pa 1196 846}%
\special{pa 1196 1346}%
\special{fp}%
\special{sh 1}%
\special{pa 1196 1346}%
\special{pa 1216 1278}%
\special{pa 1196 1292}%
\special{pa 1176 1278}%
\special{pa 1196 1346}%
\special{fp}%
%
\special{pn 8}%
\special{pa 1196 1446}%
\special{pa 1196 1946}%
\special{fp}%
\special{sh 1}%
\special{pa 1196 1946}%
\special{pa 1216 1878}%
\special{pa 1196 1892}%
\special{pa 1176 1878}%
\special{pa 1196 1946}%
\special{fp}%
%
\special{pn 8}%
\special{pa 1226 2036}%
\special{pa 1746 2376}%
\special{fp}%
\special{sh 1}%
\special{pa 1746 2376}%
\special{pa 1700 2322}%
\special{pa 1700 2346}%
\special{pa 1678 2356}%
\special{pa 1746 2376}%
\special{fp}%
%
\special{pn 8}%
\special{pa 2396 1446}%
\special{pa 2396 1946}%
\special{fp}%
\special{sh 1}%
\special{pa 2396 1946}%
\special{pa 2416 1878}%
\special{pa 2396 1892}%
\special{pa 2376 1878}%
\special{pa 2396 1946}%
\special{fp}%
%
\special{pn 8}%
\special{pa 2396 846}%
\special{pa 2396 1346}%
\special{fp}%
\special{sh 1}%
\special{pa 2396 1346}%
\special{pa 2416 1278}%
\special{pa 2396 1292}%
\special{pa 2376 1278}%
\special{pa 2396 1346}%
\special{fp}%
%
\special{pn 8}%
\special{pa 1836 426}%
\special{pa 2356 766}%
\special{fp}%
\special{sh 1}%
\special{pa 2356 766}%
\special{pa 2310 712}%
\special{pa 2310 736}%
\special{pa 2288 746}%
\special{pa 2356 766}%
\special{fp}%
%
\special{pn 8}%
\special{pa 2356 2026}%
\special{pa 1836 2366}%
\special{fp}%
\special{sh 1}%
\special{pa 1836 2366}%
\special{pa 1902 2346}%
\special{pa 1880 2336}%
\special{pa 1880 2312}%
\special{pa 1836 2366}%
\special{fp}%
\put(17.9500,-1.9500){\makebox(0,0){$P_{\lambda}$}}%
\put(21.5000,-5.5000){\makebox(0,0)[lb]{$2$}}%
\put(14.5000,-5.5000){\makebox(0,0)[rb]{$1$}}%
\put(10.9500,-10.9500){\makebox(0,0){$2$}}%
\put(10.9500,-16.9500){\makebox(0,0){$2$}}%
\put(24.9500,-16.9500){\makebox(0,0){$1$}}%
\put(24.9500,-10.9500){\makebox(0,0){$1$}}%
\put(13.9500,-21.9500){\makebox(0,0)[rt]{$1$}}%
\put(21.9500,-21.9500){\makebox(0,0)[lt]{$2$}}%
\put(27.5000,-20.0000){\makebox(0,0){$f_1^2f_2P_{\lambda}$}}%
\put(9.5000,-8.0000){\makebox(0,0){$P_{x \cdot \lambda}$}}%
\put(18.0000,-26.0000){\makebox(0,0){$P_{w_{0} \cdot \lambda}$}}%
\end{picture}%
\end{array}
\end{equation}

\vsp

\end{ex}

\begin{proof}[Proof of the ``only if'' part of Theorem~\ref{thm:p-opdem}]
Let $P=P(\mu_{\bullet}) \in \mv^{x}(\lambda)$ be an MV polytope 
with GGMS datum $\mu_{\bullet}=(\mu_{w})_{w \in W}$. 
Since $P=P(\mu_{\bullet})=
\Conv\,\bigl\{\mu_{w} \mid w \in W\bigr\}$, and 
since $P_{x \cdot \lambda}=\Conv (W_{\le x} \cdot \lambda)$
by Theorem~\ref{thm:ext}, 
it suffices to show that for each $z \in W_{\le x}$, 
there exists some $w \in W$ such that $\mu_{w}=z \cdot \lambda$. 
Now, let us take an arbitrary $z \in W_{\le x}$, 
and set $q:=\ell(zw_{0})$. 
Then we know (see \cite[Proposition~3.2.4]{Kas} and 
the comment following \cite[Proposition~4.3]{Kas}) that 
$\mv^{z}(\lambda) \supset \mv^{x}(\lambda)$, which implies 
that $P \in \mv^{z}(\lambda)$. 
Therefore, there exists 
$\bi=(i_{1},\,i_{2},\,\dots,\,i_{m}) \in R(w_{0})$ with 
$\wi{q}=zw_{0}$ such that $\mu_{\wi{l}}=
\wi{l}w_{0} \cdot \lambda$ for all $q \le l \le m$. 
In particular, $\mu_{\wi{q}}=\wi{q}w_{0} \cdot \lambda= 
(zw_{0})w_{0} \cdot \lambda= z \cdot \lambda$.
Thus, we have proved the ``only if'' part of 
Theorem~\ref{thm:p-opdem}. 
\end{proof}

In order to prove the ``if'' part of Theorem~\ref{thm:p-opdem}, 
we recall from \cite[\S2.3]{Kam1} the notion of 
Berenstein-Zelevinsky datum. 
We set $\Gamma:=
 \bigl\{w \cdot \Lambda_{i} \mid w \in W,\,i \in I\bigr\}$, 
where $\Lambda_{i}$, $i \in I$, are 
the fundamental weights for $\Fg$. 
Let $P=P(\mu_{\bullet}) \in \mv(\lambda)$ be an MV polytope 
with GGMS datum $\mu_{\bullet}=(\mu_{w})_{w \in W}$. 
For each $\gamma \in \Gamma$, 
we define $M_{\gamma} \in \BQ$ by: 
\begin{equation*}
M_{\gamma}:=\pair{w \cdot \Lambda_{i}}{\mu_{w}} \in \BQ
\quad
\text{if $\gamma=w \cdot \Lambda_{i}$ 
  for some $w \in W$ and $i \in I$};
\end{equation*}
note that the rational number $M_{\gamma} \in \BQ$ 
does note depend on 
the expression $\gamma=w \cdot \Lambda_{i}$, 
$w \in W$, $i \in I$, of $\gamma \in \Gamma$.
We call the collection 
$M_{\bullet}=(M_{\gamma})_{\gamma \in \Gamma}$ 
of rational numbers 
the Berenstein-Zelevinsky (BZ for short) datum of 
the MV polytope $P$. We know from 
\cite[Proposition~2.2]{Kam1} that
%
%
\begin{equation} \label{eq:polybz}
P=P(\mu_{\bullet})=
 \bigl\{h \in \Fh_{\BR} \mid 
   \pair{\gamma}{h} \ge M_{\gamma} \ 
   \text{for all } \gamma \in \Gamma
 \bigr\}.
\end{equation}
%
%
\begin{lem} \label{lem:bz}
Let $P_{1}=P(\mu_{\bullet}^{(1)}),\,
P_{2}=P(\mu_{\bullet}^{(2)}) \in \mv(\lambda)$ be MV polytopes 
with GGMS data 
$\mu_{\bullet}^{(1)}=(\mu_{w}^{(1)})_{w \in W}$, 
$\mu_{\bullet}^{(2)}=(\mu_{w}^{(2)})_{w \in W}$, and 
denote by 
$M_{\bullet}^{(1)}=(M_{\gamma}^{(1)})_{\gamma \in \Gamma}$, 
$M_{\bullet}^{(2)}=(M_{\gamma}^{(2)})_{\gamma \in \Gamma}$ 
the BZ data of $P_{1}$, $P_{2}$, respectively. 
Then, $P_{1} \subset P_{2}$ if and only if 
$M_{\gamma}^{(1)} \ge M_{\gamma}^{(2)}$ 
for all $\gamma \in \Gamma$. 
\end{lem}

\begin{proof}
The ``if'' part is obvious by \eqref{eq:polybz}. 
We show the ``only if'' part. 
Since $P_{1}=\Conv\,\bigl\{\mu_{w}^{(1)} \mid w \in W\bigr\}$, 
and since $P_{1} \subset P_{2}$ by assumption, 
it follows that $\mu_{w}^{(1)} \in P_{2}$ for all $w \in W$. 
Take an arbitrary element $\gamma=w \cdot \Lambda_{i}$, 
$w \in W$, $i \in I$, of $\Gamma$. 
Because $\mu_{w}^{(1)} \in P_{2}$ as above, 
we have $\pair{\gamma}{\mu_{w}^{(1)}} \ge M_{\gamma}^{(2)}$ 
by \eqref{eq:polybz}. 
Hence, by the definition of BZ data, we get
$M_{\gamma}^{(1)}=
 \pair{w \cdot \Lambda_{i}}{\mu_{w}^{(1)}}=
 \pair{\gamma}{\mu_{w}^{(1)}} \ge M_{\gamma}^{(2)}$. 
This proves the lemma. 
\end{proof}

\begin{proof}[Proof of the ``if'' part of Theorem~\ref{thm:p-opdem}]
Let $P=P(\mu_{\bullet}) \in \mv(\lambda)$ be an MV polytope 
with GGMS datum $\mu_{\bullet}=(\mu_{w})_{w \in W}$ such that 
$P \supset P_{x \cdot \lambda}$, and denote by 
$M_{\bullet}=(M_{\gamma})_{\gamma \in \Gamma}$ 
the BZ datum of $P$. 
Take $\bi=(i_{1},\,i_{2},\,\dots,\,i_{m}) \in R(w_{0})$ 
such that $\wi{p}=\si{1}\si{2} \cdots \si{p}=xw_{0}$, 
with $p=\ell(xw_{0})$.
In order to prove the ``if'' part, it suffices to show 
that $M_{\wi{l-1} \cdot \Lambda_{i_{l}}}=
M_{\wi{l} \cdot \Lambda_{i_{l}}}$ for all $p+1 \le l \le m$. 
Indeed, we know from \cite[Theorem~6.6]{Kam1} that 
for $1 \le l \le m$, 
\begin{equation*}
N_{l}:=\vp_{i_l}
 (f_{i_{l-1}}^{\max} \cdots f_{i_{2}}^{\max}f_{i_{1}}^{\max}P)
\end{equation*}
is equal to 
$M_{\wi{l-1} \cdot \Lambda_{i_{l}}}-
 M_{\wi{l} \cdot \Lambda_{i_{l}}}$. 
Hence, if $M_{\wi{l-1} \cdot \Lambda_{i_{l}}}=
M_{\wi{l} \cdot \Lambda_{i_{l}}}$ holds for all $p+1 \le l \le m$, 
then it follows that $N_{l}=0$ for all $p+1 \le l \le m$. 
Therefore, the argument after Claim~\ref{c:opdem3} 
in the proof of Theorem~3.5.1 shows that 
$P \in \Psi_{\lambda}(\CB^{x}(\lambda))=\mv^{x}(\lambda)$. 

We will show that $M_{\wi{l-1} \cdot \Lambda_{i_{l}}}=
M_{\wi{l} \cdot \Lambda_{i_{l}}}$ for all $p+1 \le l \le m$. 
By the definition of $\mv(\lambda)$, 
the polytope $P \in \mv(\lambda)$ is contained in 
$\Conv(W \cdot \lambda)=P_{w_{0} \cdot \lambda}$. 
Hence we have
\begin{equation*}
P_{x \cdot \lambda} \subset P \subset \Conv(W \cdot \lambda)=
P_{w_{0} \cdot \lambda}. 
\end{equation*}
Therefore, if we denote by 
$M_{\bullet}^{(1)}=(M_{\gamma}^{(1)})_{\gamma \in \Gamma}$ and 
$M_{\bullet}^{(2)}=(M_{\gamma}^{(2)})_{\gamma \in \Gamma}$ 
the BZ data of $P_{x \cdot \lambda}$ and 
$P_{w_{0} \cdot \lambda}$, respectively, 
we see from Lemma~\ref{lem:bz} that 
%
%
\begin{equation} \label{eq:bz-ge}
M_{\gamma}^{(1)} \ge M_{\gamma} \ge M_{\gamma}^{(2)}
\quad
\text{for all $\gamma \in \Gamma$}.
\end{equation}

\begin{claim*}
For all $p+1 \le l \le m$, we have 
\begin{equation*}
  M_{\wi{l-1} \cdot \Lambda_{i_{l}}}^{(1)}=
  M_{\wi{l-1} \cdot \Lambda_{i_{l}}}^{(2)}=
  \pair{\Lambda_{i_{l}}}{w_{0} \cdot \lambda}, \qquad
  M_{\wi{l} \cdot \Lambda_{i_{l}}}^{(1)}=
  M_{\wi{l} \cdot \Lambda_{i_{l}}}^{(2)}=
  \pair{\Lambda_{i_{l}}}{w_{0} \cdot \lambda}. 
\end{equation*}
\end{claim*}

\noindent
{\it Proof of Claim.}
Let $\mu_{\bullet}^{(1)}=(\mu_{w}^{(1)})_{w \in W}$ and  
$\mu_{\bullet}^{(2)}=(\mu_{w}^{(2)})_{w \in W}$ be 
the GGMS data of $P_{x \cdot \lambda}$ and 
$P_{w_{0} \cdot \lambda}$, respectively. We deduce 
(see Example~\ref{ex:1sk}) that 
$\mu_{\wi{l}}^{(1)}=\wi{l}w_{0} \cdot \lambda$ 
for all $p \le l \le m$. Therefore, we have
\begin{align*}
& M_{\wi{l-1} \cdot \Lambda_{i_{l}}}^{(1)}=
 \pair{ \wi{l-1} \cdot \Lambda_{i_{l}} }
      { \mu_{\wi{l-1}} } =
 \pair{ \wi{l-1} \cdot \Lambda_{i_{l}} }
      { \wi{l-1}w_{0} \cdot \lambda } =
 \pair{ \Lambda_{i_{l}} }
      { w_{0} \cdot \lambda }, \\
& M_{\wi{l} \cdot \Lambda_{i_{l}}}^{(1)}=
 \pair{ \wi{l} \cdot \Lambda_{i_{l}} }
      { \mu_{\wi{l}} } =
 \pair{ \wi{l} \cdot \Lambda_{i_{l}} }
      { \wi{l}w_{0} \cdot \lambda } =
 \pair{ \Lambda_{i_{l}} }
      { w_{0} \cdot \lambda }.
\end{align*}

Also, we know (see Remark~\ref{rem:int}\,(2)) 
that $\mu_{\wi{l}}^{(2)}=
\wi{l}w_{0} \cdot \lambda$ for all $0 \le l \le m$.
Hence, by the same calculation as above, we get
\begin{equation*}
M_{\wi{l-1} \cdot \Lambda_{i_{l}}}^{(2)}=
 \pair{ \Lambda_{i_{l}} }
      { w_{0} \cdot \lambda }, \qquad
M_{\wi{l} \cdot \Lambda_{i_{l}}}^{(2)}=
 \pair{ \Lambda_{i_{l}} }
      { w_{0} \cdot \lambda }.
\end{equation*}
This proves the claim. \bqed

\vsp

Let us take $p+1 \le l \le m$ arbitrarily. 
By combining the claim above 
with \eqref{eq:bz-ge}, we see that
\begin{equation*}
\pair{\Lambda_{i_{l}}}{w_{0} \cdot \lambda} = 
M_{\wi{l-1} \cdot \Lambda_{i_{l}}}^{(1)} \ge 
M_{\wi{l-1} \cdot \Lambda_{i_{l}}} \ge 
M_{\wi{l-1} \cdot \Lambda_{i_{l}}}^{(2)} =
\pair{\Lambda_{i_{l}}}{w_{0} \cdot \lambda},
\end{equation*}
and hence that 
$M_{\wi{l-1} \cdot \Lambda_{i_{l}}}=
\pair{\Lambda_{i_{l}}}{w_{0} \cdot \lambda}$. 
Similarly, we see that 
$M_{\wi{l} \cdot \Lambda_{i_{l}}}=
\pair{\Lambda_{i_{l}}}{w_{0} \cdot \lambda}$. 
Thus, we obtain 
$M_{\wi{l-1} \cdot \Lambda_{i_{l}}}=
 \pair{\Lambda_{i_{l}}}{w_{0} \cdot \lambda}=
 M_{\wi{l} \cdot \Lambda_{i_{l}}}$. 
This completes the proof of the ``if'' part of 
Theorem~\ref{thm:p-opdem}. 
\end{proof}

%
\subsection{Open problem.}
\label{subsec:question}

In view of Theorem~\ref{thm:p-opdem}, 
it is natural to pose the following question. 
%
%
\begin{question} \label{q:dem}
Let us take an arbitrary $x \in W$. 
Are the MV polytopes lying in the Demazure crystal 
$\mv_{x}(\lambda)$ all contained (as a set)
in the extremal MV polytope 
$P_{x \cdot \lambda}=
 \Conv(W_{\le x} \cdot \lambda)$\,?
\end{question}
%
%
\begin{rem} \label{rem:q-dem}
We cannot expect that the converse statement holds, 
as the following example shows. 
Let $\Fg$ be the simple Lie algebra of type $A_{2}$; 
the set $R(w_{0})$ consists of two elements $\bi:=(1,2,1)$ 
and $\bj:=(2,1,2)$. Set $x:=s_{1}s_{2} \in W$ and 
$\lambda:=h_{1}+h_{2} \in X_{\ast}(T) \subset \Fh_{\BR}$.
Then, the MV polytope $f_{2}f_{1}P_{\lambda} \in \mv(\lambda)$ is 
contained (as a set) 
in the extremal MV polytope $P_{x \cdot \lambda}=
 \Conv(W_{\le x} \cdot \lambda)$, but 
 $f_{2}f_{1}P_{\lambda}$ is not an element of 
the Demazure crystal $\mv_{x}(\lambda)$. 
Indeed, let $\mu_{\bullet}=(\mu_{w})_{w \in W}$ be 
the GGMS datum of 
the MV polytope $f_{2}f_{1}P_{\lambda}$. 
Using the definition of the Kashiwara operators 
$f_{1}$, $f_{2}$ and 
the formula \eqref{eq:3bm}, we deduce that the integers
$\Ni{l}=\Ni{l}(\mu_{\bullet})$, $1 \le l \le 3$, and 
$\Nj{l}=\Nj{l}(\mu_{\bullet})$, $1 \le l \le 3$, are given as:
%
%
\begin{equation} \label{eq:ex-lgt}
\begin{cases}
\Ni{1}=0, \qquad \Ni{2}=1, \qquad \Ni{3}=0, \\[1.5mm]
\Nj{1}=1, \qquad \Nj{2}=0, \qquad \Nj{3}=1,
\end{cases}
\end{equation}
and hence the GGMS datum $\mu_{\bullet}=(\mu_{w})_{w \in W}$ 
of $f_{2}f_{1}P_{\lambda}$ is given as: 
\begin{align*}
& \mu_{w_{0}}=\mu_{\wi{3}}=\mu_{\wj{3}}=\lambda, \qquad 
  \mu_{s_1s_2}=\mu_{\wi{2}}=\lambda, \\
& \mu_{s_1}=\mu_{\wi{1}}=\lambda-(h_1+h_2)=0, \qquad
  \mu_{e}=\mu_{\wi{0}}=\mu_{\wj{0}}=\lambda-(h_1+h_2)=0, \\
& \mu_{s_2s_1}=\mu_{\wj{2}}=\lambda-h_1=s_{1} \cdot \lambda, \qquad
  \mu_{s_2}=\mu_{\wj{1}}=\lambda-h_1=s_{1} \cdot \lambda. 
\end{align*}
From these, it is easily seen that 
the MV polytope $f_{2}f_{1}P_{\lambda}$ is contained 
in the extremal MV polytope $P_{x \cdot \lambda}=
 \Conv(W_{\le x} \cdot \lambda)$ (see the figures below).
However, since $xw_{0}=s_{1}s_{2}s_{2}s_{1}s_{2}=s_{2}$, and hence 
$S(xw_{0},\,\bi)=\bigl\{(2)\bigr\}$ and 
$S(xw_{0},\,\bj)=\bigl\{(1),\,(3)\bigr\}$, it follows 
from \eqref{eq:ex-lgt} that $f_{2}f_{1}P_{\lambda}$ is not 
an element of $\mv_{x}(\lambda)$. 

\vspace{7mm}

\newcommand{\convp}{%
 P_{x \cdot \lambda}=\Conv(W_{\le x} \cdot \lambda)}

\hspace*{-15mm}
\unitlength 0.1in
\begin{picture}( 24.0000, 14.3500)(  4.1200,-19.6900)
%
\special{pn 8}%
\special{pa 2288 746}%
\special{pa 1838 1046}%
\special{fp}%
%
\special{pn 8}%
\special{pa 1838 1046}%
\special{pa 1838 1646}%
\special{fp}%
%
\special{pn 8}%
\special{pa 1838 1646}%
\special{pa 2288 1946}%
\special{fp}%
%
\special{pn 8}%
\special{pa 2288 1946}%
\special{pa 2738 1646}%
\special{fp}%
%
\special{pn 8}%
\special{pa 2738 1646}%
\special{pa 2738 1046}%
\special{fp}%
%
\special{pn 8}%
\special{pa 2738 1046}%
\special{pa 2288 746}%
\special{fp}%
\put(22.8800,-6.2900){\makebox(0,0){$\lambda$}}%
\put(28.1200,-10.4900){\makebox(0,0)[lb]{$s_2\lambda$}}%
\put(28.1200,-16.4900){\makebox(0,0)[lb]{$s_2s_1\lambda$}}%
\put(22.8800,-20.5400){\makebox(0,0){$w_0\lambda$}}%
\put(17.6200,-16.4900){\makebox(0,0)[rb]{$s_1s_2\lambda$}}%
\put(17.6200,-10.4900){\makebox(0,0)[rb]{$s_1\lambda$}}%
%
\special{pn 8}%
\special{sh 0.300}%
\special{ar 2288 742 22 22  0.0000000 6.2831853}%
%
\special{pn 8}%
\special{sh 0.300}%
\special{ar 2738 1042 22 22  0.0000000 6.2831853}%
\put(27.7500,-7.0400){\makebox(0,0)[lb]{$\convp$}}%
%
\special{pn 13}%
\special{pa 2738 742}%
\special{pa 2438 1042}%
\special{fp}%
\special{sh 1}%
\special{pa 2438 1042}%
\special{pa 2500 1008}%
\special{pa 2476 1004}%
\special{pa 2472 980}%
\special{pa 2438 1042}%
\special{fp}%
%
\special{pn 8}%
\special{sh 0.300}%
\special{ar 1838 1642 22 22  0.0000000 6.2831853}%
%
\special{pn 8}%
\special{sh 0.300}%
\special{ar 1838 1042 22 22  0.0000000 6.2831853}%
%
\special{pn 8}%
\special{pa 2738 1042}%
\special{pa 1838 1642}%
\special{fp}%
%
\special{pn 4}%
\special{pa 2340 1304}%
\special{pa 1988 952}%
\special{fp}%
\special{pa 2370 1290}%
\special{pa 2018 936}%
\special{fp}%
\special{pa 2392 1266}%
\special{pa 2040 914}%
\special{fp}%
\special{pa 2422 1252}%
\special{pa 2070 900}%
\special{fp}%
\special{pa 2446 1230}%
\special{pa 2092 876}%
\special{fp}%
\special{pa 2476 1214}%
\special{pa 2122 862}%
\special{fp}%
\special{pa 2506 1200}%
\special{pa 2152 846}%
\special{fp}%
\special{pa 2528 1176}%
\special{pa 2176 824}%
\special{fp}%
\special{pa 2430 1034}%
\special{pa 2206 810}%
\special{fp}%
\special{pa 2446 1004}%
\special{pa 2228 786}%
\special{fp}%
\special{pa 2476 990}%
\special{pa 2258 772}%
\special{fp}%
\special{pa 2498 966}%
\special{pa 2296 764}%
\special{fp}%
\special{pa 2520 944}%
\special{pa 2408 832}%
\special{fp}%
\special{pa 2558 1162}%
\special{pa 2438 1042}%
\special{fp}%
\special{pa 2580 1140}%
\special{pa 2476 1034}%
\special{fp}%
\special{pa 2610 1124}%
\special{pa 2490 1004}%
\special{fp}%
\special{pa 2640 1110}%
\special{pa 2512 982}%
\special{fp}%
\special{pa 2662 1086}%
\special{pa 2536 960}%
\special{fp}%
\special{pa 2692 1072}%
\special{pa 2558 936}%
\special{fp}%
\special{pa 2716 1050}%
\special{pa 2678 1012}%
\special{fp}%
\special{pa 2310 1320}%
\special{pa 1958 966}%
\special{fp}%
\special{pa 2288 1342}%
\special{pa 1936 990}%
\special{fp}%
\special{pa 2258 1356}%
\special{pa 1906 1004}%
\special{fp}%
\special{pa 2236 1380}%
\special{pa 1882 1026}%
\special{fp}%
\special{pa 2206 1394}%
\special{pa 1860 1050}%
\special{fp}%
\special{pa 2176 1410}%
\special{pa 1838 1072}%
\special{fp}%
\special{pa 2152 1432}%
\special{pa 1838 1116}%
\special{fp}%
\special{pa 2122 1446}%
\special{pa 1838 1162}%
\special{fp}%
\special{pa 2100 1470}%
\special{pa 1838 1206}%
\special{fp}%
\special{pa 2070 1484}%
\special{pa 1838 1252}%
\special{fp}%
%
\special{pn 4}%
\special{pa 2040 1500}%
\special{pa 1838 1296}%
\special{fp}%
\special{pa 2018 1522}%
\special{pa 1838 1342}%
\special{fp}%
\special{pa 1988 1536}%
\special{pa 1838 1386}%
\special{fp}%
\special{pa 1966 1560}%
\special{pa 1838 1432}%
\special{fp}%
\special{pa 1936 1574}%
\special{pa 1838 1476}%
\special{fp}%
\special{pa 1906 1590}%
\special{pa 1838 1522}%
\special{fp}%
\special{pa 1882 1612}%
\special{pa 1838 1566}%
\special{fp}%
\end{picture}%
\hspace*{15mm}
\unitlength 0.1in
\begin{picture}( 24.0000, 14.3500)(  8.1200,-19.7400)
%
\special{pn 8}%
\special{pa 2688 750}%
\special{pa 2238 1050}%
\special{fp}%
%
\special{pn 8}%
\special{pa 2238 1050}%
\special{pa 2238 1650}%
\special{fp}%
%
\special{pn 8}%
\special{pa 2238 1650}%
\special{pa 2688 1950}%
\special{fp}%
%
\special{pn 8}%
\special{pa 2688 1950}%
\special{pa 3138 1650}%
\special{fp}%
%
\special{pn 8}%
\special{pa 3138 1650}%
\special{pa 3138 1050}%
\special{fp}%
%
\special{pn 8}%
\special{pa 3138 1050}%
\special{pa 2688 750}%
\special{fp}%
\put(26.8800,-6.3400){\makebox(0,0){$\lambda$}}%
\put(32.1200,-10.5400){\makebox(0,0)[lb]{$s_2\lambda$}}%
\put(32.1200,-16.5400){\makebox(0,0)[lb]{$s_2s_1\lambda$}}%
\put(26.8800,-20.5900){\makebox(0,0){$w_0\lambda$}}%
\put(21.6200,-16.5400){\makebox(0,0)[rb]{$s_1s_2\lambda$}}%
\put(21.6200,-10.5400){\makebox(0,0)[rb]{$s_1\lambda$}}%
%
\special{pn 8}%
\special{sh 0.300}%
\special{ar 2688 746 22 22  0.0000000 6.2831853}%
%
\special{pn 8}%
\special{pa 2688 754}%
\special{pa 2688 1354}%
\special{fp}%
%
\special{pn 8}%
\special{sh 0.300}%
\special{ar 2238 1046 22 22  0.0000000 6.2831853}%
%
\special{pn 8}%
\special{sh 0.300}%
\special{ar 2688 1346 22 22  0.0000000 6.2831853}%
\put(31.7500,-7.0900){\makebox(0,0)[lb]{$f_2f_1P_{\lambda}$}}%
%
\special{pn 8}%
\special{pa 2688 1346}%
\special{pa 2238 1046}%
\special{fp}%
%
\special{pn 8}%
\special{pa 2688 926}%
\special{pa 2440 1174}%
\special{fp}%
\special{pa 2688 972}%
\special{pa 2462 1196}%
\special{fp}%
\special{pa 2688 1016}%
\special{pa 2492 1212}%
\special{fp}%
\special{pa 2688 1062}%
\special{pa 2516 1234}%
\special{fp}%
\special{pa 2688 1106}%
\special{pa 2546 1250}%
\special{fp}%
\special{pa 2688 1152}%
\special{pa 2576 1264}%
\special{fp}%
\special{pa 2688 1196}%
\special{pa 2598 1286}%
\special{fp}%
\special{pa 2688 1242}%
\special{pa 2628 1302}%
\special{fp}%
\special{pa 2688 1286}%
\special{pa 2650 1324}%
\special{fp}%
\special{pa 2688 882}%
\special{pa 2410 1160}%
\special{fp}%
\special{pa 2688 836}%
\special{pa 2380 1144}%
\special{fp}%
\special{pa 2688 792}%
\special{pa 2358 1122}%
\special{fp}%
\special{pa 2658 776}%
\special{pa 2328 1106}%
\special{fp}%
\special{pa 2522 866}%
\special{pa 2306 1084}%
\special{fp}%
\special{pa 2388 956}%
\special{pa 2276 1070}%
\special{fp}%
%
\special{pn 13}%
\special{pa 3138 746}%
\special{pa 2538 1046}%
\special{fp}%
\special{sh 1}%
\special{pa 2538 1046}%
\special{pa 2608 1034}%
\special{pa 2586 1022}%
\special{pa 2590 998}%
\special{pa 2538 1046}%
\special{fp}%
%
\special{pn 8}%
\special{pa 3138 1046}%
\special{pa 2238 1646}%
\special{dt 0.045}%
\end{picture}%

\vspace{5mm}

\end{rem}
%
%
\begin{rem} \label{rem:succ}
We know (see \cite[\S4.1]{BaG} and \cite[Theorem~4.7]{Kam2}) 
that for $P \in \mv(\lambda)$ and $j \in I$ with 
$e_{j}P \ne \bzero$, there holds $e_{j}P \subset P$. 
Therefore, 
if an MV polytope $P \in \mv_{x}(\lambda)$ were
obtained from the extremal MV polytope 
$P_{x \cdot \lambda}$ by successive application of 
the Kashiwara operators $e_{j}$, $j \in I$, 
then it would follow immediately that 
the polytope $P$ is contained in 
$P_{x \cdot \lambda}$. 
However, in general, not all 
MV polytopes in $\mv_{x}(\lambda)$ 
can be obtained from 
$P_{x \cdot \lambda}$ by successively applying 
$e_{j}$, $j \in I$, 
as Example~\ref{ex:raising} below shows.
\end{rem}
%
%
\begin{ex} \label{ex:raising}
As in Remark~\ref{rem:q-dem}, let 
$\Fg$ be the simple Lie algebra of type $A_{2}$, and 
set $x=s_{1}s_{2} \in W$ and 
$\lambda=h_{1}+h_{2} \in X_{\ast}(T) \subset \Fh_{\BR}$; 
the set $R(w_{0})$ consists of 
$\bi=(1,2,1)$ and $\bj=(2,1,2)$.
Then, $f_{1}P_{\lambda} \in \mv_{x}(\lambda)$. 
Indeed, let $\mu_{\bullet}=(\mu_{w})_{w \in W}$ be 
the GGMS datum of the MV polytope $f_{1}P_{\lambda}$. 
Using the definition of the Kashiwara operator $f_{1}$ 
and the formula \eqref{eq:3bm}, we deduce that the integers
$\Ni{l}=\Ni{l}(\mu_{\bullet})$, $1 \le l \le 3$, and 
$\Nj{l}=\Nj{l}(\mu_{\bullet})$, $1 \le l \le 3$, are given as:
%
%
\begin{equation} \label{eq:rai1}
\begin{cases}
\Ni{1}=1, \qquad \Ni{2}=0, \qquad \Ni{3}=0, \\[1.5mm]
\Nj{1}=0, \qquad \Nj{2}=0, \qquad \Nj{3}=1. 
\end{cases}
\end{equation}
Since $xw_{0}=s_{1}s_{2}s_{2}s_{1}s_{2}=s_{2}$, and hence 
$S(xw_{0},\,\bi)=\bigl\{(2)\bigr\}$ and 
$S(xw_{0},\,\bj)=\bigl\{(1),\,(3)\bigr\}$, it follows 
from \eqref{eq:rai1} that $f_{1}P_{\lambda}$ is 
an element of $\mv_{x}(\lambda)$. 
However, we deduce from the crystal graph \eqref{eq:cg1} below of 
$\mv(\lambda) \cong \CB(\lambda)$ that 
$f_{1}P_{\lambda}$ cannot be obtained from 
$P_{x \cdot \lambda}$ by successively applying 
$e_{1}$ and $e_{2}$. 
%
%
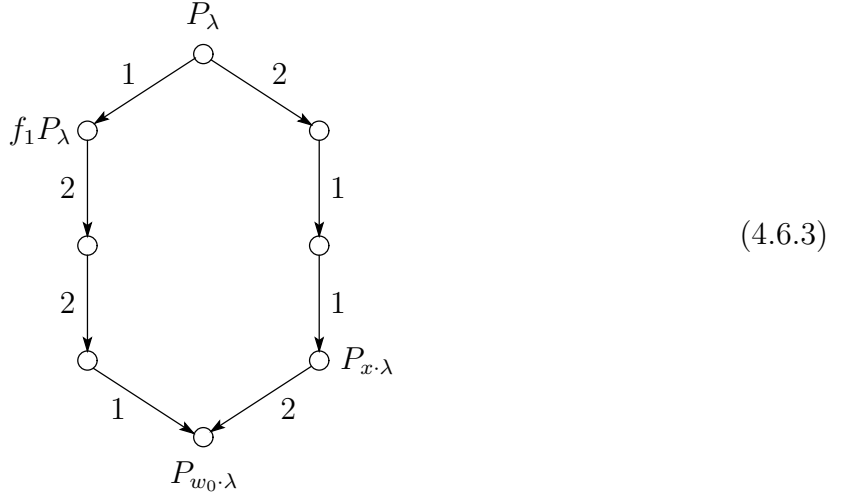
\begin{equation} \label{eq:cg1}
\hspace*{-25mm}
\begin{array}{c}
\unitlength 0.1in
\begin{picture}( 23.0500, 24.0500)(  1.4000,-25.1500)
%
\special{pn 8}%
\special{ar 1796 396 50 50  0.0000000 6.2831853}%
%
\special{pn 8}%
\special{ar 1196 796 50 50  0.0000000 6.2831853}%
%
\special{pn 8}%
\special{ar 1196 1396 50 50  0.0000000 6.2831853}%
%
\special{pn 8}%
\special{ar 1196 1996 50 50  0.0000000 6.2831853}%
%
\special{pn 8}%
\special{ar 1796 2396 50 50  0.0000000 6.2831853}%
%
\special{pn 8}%
\special{ar 2396 1996 50 50  0.0000000 6.2831853}%
%
\special{pn 8}%
\special{ar 2396 1396 50 50  0.0000000 6.2831853}%
%
\special{pn 8}%
\special{ar 2396 796 50 50  0.0000000 6.2831853}%
%
\special{pn 8}%
\special{pa 1756 416}%
\special{pa 1236 756}%
\special{fp}%
\special{sh 1}%
\special{pa 1236 756}%
\special{pa 1302 736}%
\special{pa 1280 726}%
\special{pa 1280 702}%
\special{pa 1236 756}%
\special{fp}%
%
\special{pn 8}%
\special{pa 1196 846}%
\special{pa 1196 1346}%
\special{fp}%
\special{sh 1}%
\special{pa 1196 1346}%
\special{pa 1216 1278}%
\special{pa 1196 1292}%
\special{pa 1176 1278}%
\special{pa 1196 1346}%
\special{fp}%
%
\special{pn 8}%
\special{pa 1196 1446}%
\special{pa 1196 1946}%
\special{fp}%
\special{sh 1}%
\special{pa 1196 1946}%
\special{pa 1216 1878}%
\special{pa 1196 1892}%
\special{pa 1176 1878}%
\special{pa 1196 1946}%
\special{fp}%
%
\special{pn 8}%
\special{pa 1226 2036}%
\special{pa 1746 2376}%
\special{fp}%
\special{sh 1}%
\special{pa 1746 2376}%
\special{pa 1700 2322}%
\special{pa 1700 2346}%
\special{pa 1678 2356}%
\special{pa 1746 2376}%
\special{fp}%
%
\special{pn 8}%
\special{pa 2396 1446}%
\special{pa 2396 1946}%
\special{fp}%
\special{sh 1}%
\special{pa 2396 1946}%
\special{pa 2416 1878}%
\special{pa 2396 1892}%
\special{pa 2376 1878}%
\special{pa 2396 1946}%
\special{fp}%
%
\special{pn 8}%
\special{pa 2396 846}%
\special{pa 2396 1346}%
\special{fp}%
\special{sh 1}%
\special{pa 2396 1346}%
\special{pa 2416 1278}%
\special{pa 2396 1292}%
\special{pa 2376 1278}%
\special{pa 2396 1346}%
\special{fp}%
%
\special{pn 8}%
\special{pa 1836 426}%
\special{pa 2356 766}%
\special{fp}%
\special{sh 1}%
\special{pa 2356 766}%
\special{pa 2310 712}%
\special{pa 2310 736}%
\special{pa 2288 746}%
\special{pa 2356 766}%
\special{fp}%
%
\special{pn 8}%
\special{pa 2356 2026}%
\special{pa 1836 2366}%
\special{fp}%
\special{sh 1}%
\special{pa 1836 2366}%
\special{pa 1902 2346}%
\special{pa 1880 2336}%
\special{pa 1880 2312}%
\special{pa 1836 2366}%
\special{fp}%
\put(17.9500,-1.9500){\makebox(0,0){$P_{\lambda}$}}%
\put(21.5000,-5.5000){\makebox(0,0)[lb]{$2$}}%
\put(14.5000,-5.5000){\makebox(0,0)[rb]{$1$}}%
\put(10.9500,-10.9500){\makebox(0,0){$2$}}%
\put(10.9500,-16.9500){\makebox(0,0){$2$}}%
\put(24.9500,-16.9500){\makebox(0,0){$1$}}%
\put(24.9500,-10.9500){\makebox(0,0){$1$}}%
\put(13.9500,-21.9500){\makebox(0,0)[rt]{$1$}}%
\put(21.9500,-21.9500){\makebox(0,0)[lt]{$2$}}%
\put(26.4500,-19.9500){\makebox(0,0){$P_{x \cdot \lambda}$}}%
\put(9.5000,-8.0000){\makebox(0,0){$f_{1}P_{\lambda}$}}%
\put(18.0000,-26.0000){\makebox(0,0){$P_{w_{0} \cdot \lambda}$}}%
\end{picture}%
\end{array}
\end{equation}

\vsp

\end{ex}

%
\section{
   Decomposition of Demazure crystals and 
   opposite Demazure crystals.}
\label{sec:decom}

We fix (once and for all) 
an arbitrary dominant coweight 
$\lambda \in X_{\ast}(T) \subset \Fh_{\BR}$. 

%
\subsection{Decomposition of Demazure crystals.}
\label{subsec:decom-dem}

We set 
$W_{\lambda}:=\bigl\{w \in W \mid w \cdot \lambda=\lambda\bigr\} \subset W$, 
and denote by $W^{\lambda}_{\min}$ the set of minimal coset representatives 
for the quotient set $W/W_{\lambda}$ 
with respect to the Bruhat ordering $\le$ on $W$ 
(see, for example, \cite[Chap.~2, \S4]{BB}).
For each $x \in W^{\lambda}_{\min}$, define a subset 
$\ol{\mv}_{x}(\lambda)$ of 
the Demazure crystal $\mv_{x}(\lambda)$ by (cf. \cite[\S9.1]{Kasb}): 
%
%
\begin{equation} \label{eq:dfn-md}
\ol{\mv}_{x}(\lambda)=
\mv_{x}(\lambda) \setminus 
\left(
 \bigcup_{z \in W^{\lambda}_{\min},\,\,z < x}
 \mv_{z}(\lambda)
\right); 
\end{equation}
recall from \cite[Proposition~3.2.4]{Kas} that 
for $z \in W^{\lambda}_{\min}$ with $z < x$, we have 
$\mv_{z}(\lambda) \subset \mv_{x}(\lambda)$. 
We can easily show by induction on $\ell(x)$ 
(see also \cite[\S9.1]{Kasb}) that 
for each $x \in W$, 
%
%
\begin{equation} \label{eq:init1}
\mv_{x}(\lambda)=
 \bigsqcup_{z \in W^{\lambda}_{\min},\,\,z \le x}
 \ol{\mv}_{z}(\lambda), 
\end{equation}
and in particular, 
%
%
\begin{equation} \label{eq:init2}
\mv(\lambda)=
 \bigsqcup_{z \in W^{\lambda}_{\min}}
 \ol{\mv}_{z}(\lambda); 
\end{equation}
note that $\mv_{w_0}(\lambda)=\mv(\lambda)$. 
Similarly, for each $x \in W$, define a subset 
$\ol{\mv}_{x}(\infty)$ of the Demazure crystal 
$\mv_{x}(\infty)$ by (cf. \cite[Proposition~9.1.6\,(2)]{Kasb}): 
\begin{equation*}
\ol{\mv}_{x}(\infty)=
\mv_{x}(\infty) \setminus 
\left(
 \bigcup_{z \in W,\,\,z < x}
 \mv_{z}(\infty)
\right). 
\end{equation*}
It follows from \cite[Proposition~9.1.6\,(2)]{Kasb} that 
for each $x \in W$, 
%
%
\begin{equation} \label{eq:initinf1}
\mv_{x}(\infty)=
 \bigsqcup_{z \in W,\,\,z \le x}
 \ol{\mv}_{z}(\infty), 
\end{equation}
and in particular, 
%
%
\begin{equation} \label{eq:initinf2}
\mv(\infty)=
 \bigsqcup_{z \in W}
 \ol{\mv}_{z}(\infty);
\end{equation}
note that $\mv_{w_0}(\infty)=\mv(\infty)$. 
For each $P \in \mv(\lambda)$ (resp., $P \in \mv(\infty)$), 
we denote by $\iota(P)$ 
the unique element $z \in W^{\lambda}_{\min}$ 
(resp., $z \in W$) for which 
$P \in \ol{\mv}_{z}(\lambda)$ 
(resp., $P \in \ol{\mv}_{z}(\infty)$); 
we see from \eqref{eq:init1} and \eqref{eq:init2} 
(resp., \eqref{eq:initinf1} and \eqref{eq:initinf2}) 
that the element $\iota(P) \in W$ specifies the 
smallest (with respect to the inclusion relation) 
Demazure crystal in which $P$ lies. 
%
%
\begin{rem} \label{rem:extinit}
Let $x \in W$, and denote by 
$x^{\lambda}_{\min} \in W^{\lambda}_{\min}$ 
the (unique) minimal element of the left coset 
$xW_{\lambda}$ with respect to the Bruhat ordering on $W$. 
From \eqref{eq:dfn-md}, it is easily seen 
that $\iota(P_{x \cdot \lambda})=x^{\lambda}_{\min}$. 
\end{rem}

Let $P=P(\mu_{\bullet})$ be an element of $\mv(\lambda)$ 
(resp., $\mv(\infty)$) with GGMS datum 
$\mu_{\bullet}=(\mu_{w})_{w \in W}$. For each 
$\bi=(i_{1},\,i_{2},\,\dots,\,i_{m}) \in R(w_{0})$, 
we set 
\begin{equation*}
T(P,\,\bi):=
\left\{
 \begin{array}{l|l}
 (a_{1},\,a_{2},\,\dots,\,a_{l}) \in [1,\,m]^{l} \ & \ 
\begin{array}{l}
0 \le l \le m \\[1.5mm]
1 \le a_{1} < a_{2} < \cdots < a_{l} \le m \\[1.5mm]
\Ni{a_q}=\Ni{a_q}(\mu_{\bullet})=0 \quad \text{for all $1 \le q \le l$}
\end{array}
\end{array}
\right\}, 
\end{equation*}
and define a subset $W(P,\,\bi)$ of $W$ by:
\begin{equation*}
W(P,\,\bi)=
\bigl\{ 
 \si{a_1}\si{a_2} \cdots \si{a_l}w_{0} \in W \mid 
 (a_{1},\,a_{2},\,\dots,\,a_{l}) \in T(P,\,\bi)
\bigr\}.
\end{equation*}
%
%
\begin{lem} \label{lem:init}
Let $P \in \mv(\lambda)$ 
{\rm(}resp., $P \in \mv(\infty)${\rm)}, and 
take an arbitrary $\bi \in R(w_{0})$. 
We have $\iota(P) \in W(P,\,\bi)$, and 
$z \ge \iota(P)$ for all $z \in W(P,\,\bi)$ 
with respect to the Bruhat ordering $\ge$ on $W$. 
\end{lem}

\begin{proof}
We give a proof only for the case that 
$P \in \mv(\lambda)$; the proof for the case 
that $P \in \mv(\infty)$ is similar. 
Since $P \in \ol{\mv}_{\iota(P)}(\lambda) \subset 
\mv_{\iota(P)}(\lambda)$ by the definition of $\iota(P)$, 
it follows from Corollary~\ref{cor:main} that
there exists a sequence 
$(a_{1},\,a_{2},\,\dots,\,a_{l}) \in \ha{S}(\iota(P)w_{0},\,\bi)$ 
such that $\Ni{a_{q}}=\Ni{a_{q}}(\mu_{\bullet})=0$ 
for all $1 \le q \le l$. Therefore, the sequence 
$(a_{1},\,a_{2},\,\dots,\,a_{l})$ is an element of 
$T(P,\,\bi)$, and hence
\begin{equation*}
\iota(P) 
 = (\iota(P)w_{0})w_{0}
 = \si{a_1}\si{a_2} \cdots \si{a_l}w_{0} \in W(P,\,\bi).
\end{equation*}

Now, let $z \in W(P,\,\bi)$. Then we deduce 
from Corollary~\ref{cor:main} that 
$P \in \mv_{z}(\lambda)$. 
Denote by $z^{\lambda}_{\min} \in W^{\lambda}_{\min}$ 
the minimal element of the left coset $zW_{\lambda}$; 
obviously, $z^{\lambda}_{\min} \le z$. 
Since $z^{\lambda}_{\min} \cdot \lambda=z \cdot \lambda$, 
we have $P \in \mv_{z}(\lambda)=
\mv_{z^{\lambda}_{\min}}(\lambda)$ (see Remark~\ref{rem:demcos}). 
Therefore, by the decomposition \eqref{eq:init1}, 
there exists an element 
$z' \in W^{\lambda}_{\min}$ with $z' \le z^{\lambda}_{\min}$ 
such that $P \in \ol{\mv}_{z'}(\lambda)$.
It follows from the uniqueness of $\iota(P)$ that 
$z' = \iota(P)$, and hence $\iota(P) \le z^{\lambda}_{\min}$.
Since $z^{\lambda}_{\min} \le z$ as noted above, we get 
$\iota(P) \le z$. This proves the lemma. 
\end{proof}
%
%
\begin{rem} \label{rem:init}
Let $P \in \mv(\lambda)$ 
(resp., $P \in \mv(\infty)$), and take an arbitrary 
$\bi \in R(w_{0})$. By Lemma~\ref{lem:init}, 
$\iota(P)$ is the minimum element $\min W(P,\,\bi)$ 
of $W(P,\,\bi)$ with respect to the Bruhat ordering on $W$.
Namely, we have 
$\iota(P)=\min W(P,\,\bi)$ for every $\bi \in R(w_{0})$. 
\end{rem}

The next proposition follows immediately 
from Remark~\ref{rem:init}. 
%
%
\begin{prop} \label{prop:init}
Let $P \in \mv(\lambda)$ 
{\rm(}resp., $P \in \mv(\infty)${\rm)}, and 
let $x \in W^{\lambda}_{\min}$ {\rm(}resp., $x \in W${\rm)}. 
Then, $P \in \ol{\mv}_{x}(\lambda)$ 
{\rm(}resp., $P \in \ol{\mv}_{x}(\infty)${\rm)} 
if and only if 
for some {\rm(}or equivalently, every\,{\rm)} $\bi \in R(w_{0})$, 
the element $x$ is identical to 
the minimum element $\min W(P,\,\bi)$ of 
$W(P,\,\bi)$ with respect to the Bruhat ordering on $W$.
\end{prop}

Here we give an application of our description above 
for the elements $\iota(P)$, $P \in \mv(\lambda)$. 
Let $\ast:U_{q}^{-}(\Fg^{\vee}) \rightarrow U_{q}^{-}(\Fg^{\vee})$ 
be the $\BC(q)$-algebra antiautomorphism 
that fixes the Chevalley generators $y_{j} \in U_{q}^{-}(\Fg^{\vee})$, 
$j \in I$, corresponding to the negative simple roots 
(see \cite[\S1.3]{Kasc}). 
We know from \cite[Theorem 2.1.1]{Kas} that 
this antiautomorphism 
$\ast:U_{q}^{-}(\Fg^{\vee}) \rightarrow U_{q}^{-}(\Fg^{\vee})$ 
induces an involution $\ast:\CB(\infty) \rightarrow \CB(\infty)$
on the crystal base $\CB(\infty)$ of $U_{q}^{-}(\Fg^{\vee})$, 
which we call the Kashiwara involution (or, $\ast$-operation) 
on $\CB(\infty)$. We then define an involution 
$\ast:\mv(\infty) \rightarrow \mv(\infty)$ so that 
the following diagram commutes: 
\begin{equation*}
\begin{CD}
\mv(\infty) @>{\ast}>> 
\mv(\infty)\phantom{,} \\
@V{\Psi}VV @VV{\Psi}V \\
\CB(\infty) @>{\ast}>> 
\CB(\infty), 
\end{CD}
\end{equation*}
and for $P \in \mv(\infty)$, 
denote by $P^{\ast}$ the image of $P$ 
under the involution 
$\ast:\mv(\infty) \rightarrow \mv(\infty)$. 
%
%
\begin{cor} \label{cor:init-ast}
Let $P \in \mv(\infty)$. Then, the element 
$\iota(P^{\ast}) \in W$ is identical to 
the inverse $\bigl(\iota(P)\bigr)^{-1}$ 
of the element $\iota(P) \in W$. 
\end{cor}

\begin{proof}[Proof of Corollary~\ref{cor:init-ast}]
Let $\mu_{\bullet}=(\mu_{w})_{w \in W}$ and 
$\mu_{\bullet}'=(\mu_{w}')_{w \in W}$ be GGMS data 
for $P \in \mv(\infty)$ and $P^{\ast} \in \mv(\infty)$, 
respectively. 
Take an arbitrary 
$\bi=(i_{1},\,i_{2},\,\dots,\,i_{m}) \in R(w_{0})$, and 
define $\bj=(j_{1},\,j_{2},\,\dots,\,j_{m}) \in R(w_{0})$ by:
$\bj=(\omega(i_{m}),\,\omega(i_{m-1}),\,\dots,\,\omega(i_{1}))$,
where $\omega:I \rightarrow I$ is the Dynkin diagram automorphism 
for $\Fg$ such that $\alpha_{\omega(j)}=-w_{0} \cdot \alpha_{j}$ 
for all $j \in I$. 
Then we deduce from \cite[Proposition~6.1]{Kam2} 
(see also \cite[Proposition~3.3\,(iii)]{BZ}) that
%
%
\begin{equation} \label{eq:ast-length}
\Bigl(
 \Nj{1}(\mu_{\bullet}'),\,
 \Nj{2}(\mu_{\bullet}'),\,\dots,\,
 \Nj{m}(\mu_{\bullet}')
\Bigr)
=
\Bigl(
 \Ni{m}(\mu_{\bullet}),\,
 \Ni{m-1}(\mu_{\bullet}),\,\dots,\,
 \Ni{1}(\mu_{\bullet})
\Bigr).
\end{equation}
From \eqref{eq:ast-length}, it follows that a sequence 
$(a_{1},\,a_{2},\,\dots,\,a_{l})$ is an element of 
$T(P,\,\bi)$ if and only if the sequence 
$(m-a_{l}+1,\,m-a_{l-1}+1,\,\dots,\,m-a_{1}+1)$ is 
an element of $T(P^{\ast},\,\bj)$. Therefore, we have
%
%
\begin{align}
& W(P^{\ast},\,\bj) 
=
\bigl\{ 
 \sj{b_1}\sj{b_2} \cdots \sj{b_l}w_{0} \in W \mid 
 (b_{1},\,b_{2},\,\dots,\,b_{l}) \in T(P^{\ast},\,\bj)
\bigr\} \nonumber \\
& \quad =
\bigl\{ 
 \sj{m-a_l+1}\sj{m-a_{l-1}+1} \cdots \sj{m-a_1+1}w_{0} \in W \mid 
 (a_{1},\,a_{2},\,\dots,\,a_{l}) \in T(P,\,\bi)
\bigr\} \nonumber \\
& \quad =
\bigl\{ 
 \si{\omega(a_l)}\si{\omega(a_{l-1})} \cdots \si{\omega(a_1)}w_{0} \in W \mid 
 (a_{1},\,a_{2},\,\dots,\,a_{l}) \in T(P,\,\bi)
\bigr\} \nonumber \\
& \quad =
\bigl\{ 
 w_{0}\si{a_l}\si{a_{l-1}} \cdots \si{a_1} \in W \mid 
 (a_{1},\,a_{2},\,\dots,\,a_{l}) \in T(P,\,\bi)
\bigr\} \nonumber \\
& \quad = 
\bigl\{w^{-1} \mid w \in W(P,\,\bi)\bigr\}.  \label{eq:W-inv}
\end{align}
We know from Proposition~\ref{prop:init} that 
the elements $\iota(P^{\ast})$ and $\iota(P)$ are 
identical to the minimum elements $\min W(P^{\ast},\,\bj)$ and 
$\min W(P,\,\bi)$, respectively. Combining this fact and 
\eqref{eq:W-inv} above, we conclude that 
$\iota(P^{\ast})=\bigl(\iota(P)\bigr)^{-1}$, as desired. 
\end{proof}

%
\subsection{Decomposition of opposite Demazure crystals.}
\label{subsec:decom-opdem}

Let us denote by 
$W^{\lambda}_{\max}$ the set of maximal coset representatives 
for the quotient set $W/W_{\lambda}$ 
with respect to the Bruhat ordering $\le$ on $W$ 
(see, for example, \cite[Chap.~2, \S4]{BB}).

\begin{rem}
It is clear that $W^{\lambda}_{\max}=
W^{\lambda}_{\min}w_{0,\lambda}$, where 
$w_{0,\lambda} \in W_{\lambda}$ denotes 
the longest element of $W_{\lambda}$.
\end{rem}

For each $x \in W^{\lambda}_{\max}$, define a subset 
$\ol{\mv}^{x}(\lambda)$ of 
the opposite Demazure crystal $\mv^{x}(\lambda)$ by:
%
%
\begin{equation} \label{eq:dfn-opmd}
\ol{\mv}^{x}(\lambda)=
\mv^{x}(\lambda) \setminus 
\left(
 \bigcup_{z \in W^{\lambda}_{\max},\,\,z > x}
 \mv^{z}(\lambda)
\right);
\end{equation}
recall from \cite[Proposition~3.2.4]{Kas} 
(along with the comment following 
\cite[Proposition~4.3]{Kas}) that 
for $z \in W^{\lambda}_{\max}$ with $z > x$, we have 
$\mv^{z}(\lambda) \subset \mv^{x}(\lambda)$. 
We can easily show by descending induction on 
$\ell(x)$ that for each $x \in W$, 
%
%
\begin{equation} \label{eq:fin1}
\mv^{x}(\lambda)=
 \bigsqcup_{z \in W^{\lambda}_{\max},\,\,z \ge x}
 \ol{\mv}^{z}(\lambda), 
\end{equation}
and in particular, 
%
%
\begin{equation} \label{eq:fin2}
\mv(\lambda)=
 \bigsqcup_{z \in W^{\lambda}_{\max}}
 \ol{\mv}^{z}(\lambda);
\end{equation}
note that $\mv^{e}(\lambda)=\mv(\lambda)$. 
For each $P \in \mv(\lambda)$, we denote by $\kappa(P)$ 
the unique element $z \in W^{\lambda}_{\max}$ for which 
$P \in \ol{\mv}^{z}(\lambda)$; 
we see from \eqref{eq:fin1} and \eqref{eq:fin2} 
that the element $\kappa(P) \in W$ specifies the 
smallest (with respect to the inclusion relation) 
opposite Demazure crystal in which $P$ lies. 
%
%
\begin{rem} \label{rem:extfin}
Let $x \in W$, and denote by 
$x^{\lambda}_{\max} \in W^{\lambda}_{\max}$ 
the (unique) maximal element of 
the left coset $xW_{\lambda}$ 
with respect to the Bruhat ordering on $W$;
note that $x^{\lambda}_{\max}=
x^{\lambda}_{\min} w_{0,\lambda}$. 
From \eqref{eq:dfn-opmd}, it is easily seen that 
$\kappa(P_{x \cdot \lambda})=x^{\lambda}_{\max}=
x^{\lambda}_{\min}w_{\lambda,0}$. 
\end{rem}

Let $P=P(\mu_{\bullet})$ be an element of $\mv(\lambda)$ 
with GGMS datum $\mu_{\bullet}=(\mu_{w})_{w \in W}$. 
For each $\bi=(i_{1},\,i_{2},\,\dots,\,i_{m}) \in R(w_{0})$, 
we set 
\begin{equation*}
U(P,\,\bi):= \bigl\{ 0 \le p \le m \mid 
 \text{$\mu_{\wi{l}}=\wi{l}w_{0} \cdot \lambda$ 
       for every $p \le l \le m$}
 \bigr\}.
\end{equation*}
Then we set 
\begin{equation*}
Y(P):=\bigl\{\wi{p}w_{0} \mid p \in U(P,\,\bi), \, \bi \in R(w_{0}) \bigr\}.
\end{equation*}
%
%
\begin{lem} \label{lem:fin}
Let $P \in \mv(\lambda)$.  
We have $\kappa(P) \in Y(P)$, and 
$\kappa(P) \ge z$ for all $z \in Y(P)$
with respect to the Bruhat ordering $\ge$ on $W$. 
\end{lem}

\begin{proof}
Let $P=P(\mu_{\bullet}) \in \mv(\lambda)$ be an MV polytope 
with GGMS datum $\mu_{\bullet}=(\mu_{w})_{w \in W}$. 
It follows from the definition of $\kappa(P)$ that
$P \in \ol{\mv}^{\kappa(P)}(\lambda) \subset 
\mv^{\kappa(P)}(\lambda)$. 
Therefore, by the definition of $\mv^{\kappa(P)}(\lambda)$, 
there exists $\bi \in R(w_{0})$ such that
$\wi{p}=\kappa(P)w_{0}$, with $p=\ell(\kappa(P)w_{0})$, 
and such that $\mu_{\wi{l}}=\wi{l}w_{0} \cdot \lambda$ 
for all $p \le l \le m$. Hence we have
\begin{equation*}
Y(P) \ni \wi{p}w_{0}
  =\bigl(\kappa(P)w_{0}\bigr)w_{0}
  =\kappa(P).
\end{equation*}

Now, let $z \in Y(P)$. Then we deduce 
from Theorem~\ref{thm:opdem} that $P \in \mv^{z}(\lambda)$.
Denote by $z^{\lambda}_{\max} \in W^{\lambda}_{\max}$ 
the maximal element of the left coset $zW_{\lambda}$; 
obviously, $z^{\lambda}_{\max} \ge z$. 
Since $z^{\lambda}_{\max} \cdot \lambda=z \cdot \lambda$, 
we have $P \in \mv^{z}(\lambda)=\mv^{z^{\lambda}_{\max}}(\lambda)$ 
(see Remark~\ref{rem:opdemcos}). 
Therefore, by the decomposition \eqref{eq:fin1}, 
there exists an element $z' \in W^{\lambda}_{\max}$ 
with $z' \ge z^{\lambda}_{\max}$ 
such that $P \in \ol{\mv}^{z'}(\lambda)$. 
It follows from the uniqueness of $\kappa(P)$ 
that $z' = \kappa(P)$, and hence 
$\kappa(P) \ge z^{\lambda}_{\max}$.
Since $z^{\lambda}_{\max} \ge z$ as noted above, 
we get $\kappa(P) \ge z$. This proves the lemma. 
\end{proof}
%
%
\begin{rem} \label{rem:fin}
Let $P \in \mv(\lambda)$. 
By Lemma~\ref{lem:fin}, 
$\kappa(P)$ is the maximum element 
$\max Y(P)$ of $Y(P)$ with respect to 
the Bruhat ordering on $W$. 
Namely, we have $\kappa(P)=\max Y(P)$. 
\end{rem}

The next proposition follows immediately 
from Remark~\ref{rem:fin}. 
%
%
\begin{prop} \label{prop:fin}
Let $P \in \mv(\lambda)$, and 
let $x \in W^{\lambda}_{\max}$. 
Then, $P \in \ol{\mv}^{x}(\lambda)$ 
if and only if the element $x$ is identical to 
the maximum element $\max Y(P)$ of $Y(P)$ 
with respect to the Bruhat ordering on $W$. 
\end{prop}


{\small
\setlength{\baselineskip}{13pt}
\renewcommand{\refname}{References}

}


\begin{thebibliography}{XXXX}

\bibitem[An]{A}
J. E. Anderson, 
A polytope calculus for semisimple groups, 
{\it Duke Math. J.} {\bf 116} (2003), 567--588.

\bibitem[At]{Atiyah}
M. F. Atiyah, 
Convexity and commuting Hamiltonians, 
{\it Bull. London Math. Soc.} {\bf 14} (1982), 1--15.

\bibitem[AP]{AP}
M. Atiyah and A. Pressley,
Convexity and loop groups, {\it in} 
``Arithmetic and geometry, Vol. II'' (M. Artin and J. Tate, Eds.), 
pp.~33--63, Progr. Math. Vol.~36, Birkh\"auser, Boston, 1983. 

\bibitem[BaG]{BaG}
P. Baumann and S. Gaussent, 
On Mirkovi\'c-Vilonen cycles and crystal combinatorics,
{\it Represent. Theory} {\bf 12} (2008), 83--130.
 
\bibitem[BeZ]{BZ}
A. Berenstein and A. Zelevinsky, 
Tensor product multiplicities, 
canonical bases and totally positive varieties, 
{\it Invent. Math.} {\bf 143} (2001), 77--128.

\bibitem[BjB]{BB}
A. Bj\"orner and F. Brenti, 
``Combinatorics of Coxeter Groups'', 
Graduate Texts in Mathematics Vol.~231, 
Springer, New York, 2005.

\bibitem[BFG]{BFG}
A. Braverman, M. Finkelberg, and D. Gaitsgory, 
Uhlenbeck spaces via affine Lie algebras, {\it in} 
``The Unity of Mathematics'' (P. Etingof et al., Eds.), 
pp.~17--135, Progr. Math. Vol.~244, 
Birkh\"auser, Boston, 2006.

\bibitem[BrG]{BG}
A. Braverman and D. Gaitsgory, 
Crystals via the affine Grassmannian, 
{\it Duke Math. J.} {\bf 107} (2001), 561--575.

\bibitem[GL]{GL}
S. Gaussent and P. Littelmann,
LS galleries, the path model, and MV cycles,
{\it Duke Math. J.} {\bf 127} (2005), 35--88.

\bibitem[Ha]{Haines}
T.J. Haines, 
Equidimensionality of convolution morphisms and 
applications to saturation problems, 
{\it Adv. Math.} {\bf 207} (2006), 297--327. 

\bibitem[Ho]{H}
J. Hong, 
The action of a Dynkin automorphism 
on Mirkovi\'{c}-Vilonen cycles and polytopes, 
preprint, 2007, arXiv:0711.0070. 

\bibitem[I]{Ion}
B. Ion, 
A weight multiplicity formula for Demazure modules, 
{\it Int. Math. Res. Not.} {\bf 2005} (2005), no.~5, 
311--323. 

\bibitem[Kam1]{Kam1}
J. Kamnitzer, 
Mirkovi\'{c}-Vilonen cycles and polytopes, 
preprint, 2005, arXiv:math.AG/ 0501365. 

\bibitem[Kam2]{Kam2}
J. Kamnitzer, 
The crystal structure on the set of 
Mirkovi\'{c}-Vilonen polytopes, 
{\it Adv. Math.} {\bf 215} (2007), 66--93. 

\bibitem[Kas1]{Kasc}
M. Kashiwara, 
On crystal bases of the $q$-analogue of 
universal enveloping algebras, 
{\it Duke Math. J.} {\bf 63} (1991), 465--516.

\bibitem[Kas2]{Kas}
M. Kashiwara, 
The crystal base and Littelmann's 
refined Demazure character formula, 
{\it Duke Math. J.} {\bf 71} (1993), 839--858. 

\bibitem[Kas3]{Kaslz}
M. Kashiwara, On level-zero representations of 
quantized affine algebras, 
{\it Duke Math. J.} {\bf 112} (2002), 117--175.

\bibitem[Kas4]{Kasb}
M. Kashiwara, 
Bases cristallines des groupes quantiques, 
Cours Sp\'ecialis\'es Vol. 9, 
Soci\'et\'e Math\'ematique de France, Paris, 2002.

\bibitem[Kat]{Kat}
S. Kato, Private communication, 2008. 

\bibitem[LP]{LP-IMRN}
C. Lenart and A. Postnikov,
Affine Weyl groups in $K$-Theory and representation theory, 
{\it Int. Math. Res. Not.} {\bf 2007} (2007), no.~12, 
Article ID rnm038, 65 pages.

\bibitem[Li1]{LitInv}
P. Littelmann, 
A Littlewood-Richardson rule for symmetrizable Kac-Moody algebra, 
{\it Invent. Math.} {\bf 116} (1994), 329--346.

\bibitem[Li2]{Li1}
P. Littelmann, 
Crystal graphs and Young tableaux, 
{\it J. Algebra} {\bf 175} (1995), 65--87. 

\bibitem[Li3]{Li2}
P. Littelmann, 
Cones, crystals, and patterns, 
{\it Transform. Groups} {\bf 3} (1998), 145--179. 

\bibitem[Lu1]{Lu1}
G. Lusztig, 
``Introduction to Quantum Groups'', 
Progr. Math. Vol.~110, Birkh\"{a}user, Boston, 1993.

\bibitem[Lu2]{Lu2}
G. Lusztig, Braid group action and canonical bases, 
{\it Adv. Math.} {\bf 122} (1996), 237--261. 

\bibitem[MiV1]{MV1}
I. Mirkovi\'c and K. Vilonen, 
Perverse sheaves on affine Grassmannians and Langlands duality,
{\it Math. Res. Lett.} {\bf 7} (2000), 13--24.

\bibitem[MiV2]{MV2}
I. Mirkovi\'c and K. Vilonen, 
Geometric Langlands duality and representations of 
algebraic groups over commutative rings, 
{\it Ann. of Math.} (2) {\bf 166} (2007), 95--143. 

\bibitem[MoP]{MP}
R. V. Moody and A. Pianzola, 
``Lie Algebras with Triangular Decompositions'',
Canadian Mathematical Society Series of Monographs and Advanced Texts, 
Wiley-Interscience, New York, 1995.

\bibitem[NS1]{NS1}
S. Naito and D. Sagaki,
Lakshmibai-Seshadri paths fixed by a diagram automorphism, 
{\it J. Algebra} {\bf 245} (2001), 395--412. 

\bibitem[NS2]{NS2}
S. Naito and D. Sagaki, 
Crystal bases and diagram automorphisms, 
{\it in} ``Representation Theory of Algebraic Groups and 
Quantum Groups'' (T. Shoji et al., Eds.), 
pp.~321--341, Adv. Stud. Pure Math. Vol.~40, 
Math. Soc. Japan, Tokyo, 2004. 

\bibitem[NS3]{NS3}
S. Naito and D. Sagaki,
A modification of the Anderson-Mirkovi\'c conjecture
for Mirkovi\'c-Vilonen polytopes in types $B$ and $C$, 
{\it J. Algebra} {\bf 320} (2008), 387--416. 

\bibitem[N]{N}
T. Nakashima, 
Polytopes for crystallized Demazure modules and extremal vectors, 
{\it Comm. Algebra} {\bf 30} (2002), 1349--1367. 

\bibitem[NP]{NP}
B. C. Ng\^o and P. Polo,
R\'esolutions de Demazure affines et formule de Casselman-Shalika g\'eom\'etrique, 
{\it J. Algebraic Geom.} {\bf 10} (2001), 515--547. 

\bibitem[Sa]{S}
A. Savage, Quiver varieties and Demazure modules, 
{\it Math. Ann.} {\bf 335} (2006), 31--46.

\bibitem[Sc]{Schw}
C. Schwer, 
Galleries and $q$-analogs in combinatorial 
representation theory, Dissertation, University of Cologne, 2006. 

\end{thebibliography}
\end{document}